\documentclass{article}

\usepackage[english]{babel}
\usepackage[utf8]{inputenc}
\usepackage{amsfonts}
\usepackage{mathrsfs}
\usepackage{textcomp}
\usepackage{newunicodechar}
\newunicodechar{⁺}{\textsuperscript{+}}
\usepackage[
  a4paper,
  top=2cm,
  bottom=2cm,
  left=3cm,
  right=3cm,
  marginparwidth=1.75cm]{geometry}
\usepackage{authblk}
\usepackage{graphicx}
\usepackage{wrapfig}
\usepackage{tikz}
\usetikzlibrary{spy}
\usepackage{graphicx}
  \usetikzlibrary{shapes.geometric}
  \usetikzlibrary{arrows.meta}
  \usetikzlibrary{positioning}
  \usetikzlibrary{calc}
\usepackage{subcaption}
\usepackage{adjustbox}
\usepackage{placeins}
\usepackage[right]{lineno}
\usepackage{amssymb}
\usepackage{amsmath}
\usepackage{amsthm}
  \theoremstyle{plain}
    \newtheorem{theorem}{Theorem}
    \newtheorem{lemma}[theorem]{Lemma}
    \newtheorem{proposition}[theorem]{Proposition}
     
    \newtheorem{example}[theorem]{Example}    
  \theoremstyle{definition}

  \theoremstyle{remark}
    \newtheorem{remark}{Remark}
\usepackage{hyperref}
\hypersetup{%
 unicode=true,  
 pdftoolbar=true,
 pdfmenubar=true,
 pdffitwindow=false,
 pdfstartview={FitH},
 pdfnewwindow=true,  
 colorlinks=true,   
 linkcolor=red,    
 citecolor=red,    
 filecolor=magenta,
 urlcolor=blue   
}
\usepackage{cleveref}
  \crefformat{figure}{#2Figure~#1#3}
  \Crefformat{figure}{#2Figure~#1#3}
  \crefformat{table}{#2Table~#1#3}
  \Crefformat{table}{#2Table~#1#3}
  \crefformat{section}{#2Section~#1#3}
  \Crefformat{section}{#2Section~#1#3}
  \crefformat{enumi}{#2(#1)#3}
  \Crefformat{enumi}{#2(#1)#3}
\usepackage{booktabs}
\usepackage{paralist}
\usepackage[%
  sorting=none,
  citestyle=numeric-comp,
  maxnames=10,
  backend=biber]{biblatex}
\usepackage{csquotes}
\usepackage{xcolor} 
\usepackage{comment}

\usepackage{localmacros}
\newcommand{\Natural}{\mathbb{N}}
\newcommand{\Integer}{\mathbb{Z}}

\newcommand{\Real}{\mathbb{R}}

\newcommand{\Expect}{\operatorname{\mathbb{E}}}

\newcommand{\StateSpace}{\mathcal{M}}
\newcommand{\Koop}{\mathcal{K}}

\title{Equation free data-driven modelling of chaotic processes}
\author[1]{Aurora Poggi}
\author[3]{Marco Martens}
\author[2]{Hjalmar Brismar}
\author[1]{Ozan Öktem}
\author[1]{Liviana Palmisano}
\affil[1]{Department of Mathematics, KTH - Royal Institute of Technology, Stockholm, Sweden}
\affil[2]{Science for Life Laboratory, Department of Applied Physics, KTH - Royal Institute of Technology, Stockholm, Sweden}
\affil[3]{Department of Mathematics, Stony Brook University, Stony Brook, NY USA}

\begin{document}
\maketitle

\textcolor{blue}{}\global\long\def\TDD#1{{\color{red}To\, Do~(#1)}}

\begin{abstract} 
We introduce a method for constructing predictive models of non cyclic physical processes directly from time-series data, without assuming an underlying differential equation. The observations define a discrete evolution rule whose recurrent behaviour captures the essential dynamics of the process. Analysing this behaviour across multiple geometric scales leads to probabilistic models in the form of Markov chains. Hyperbolicity criteria identify when these models provide a consistent statistical description of the data. The method is inspired by, and illustrated through, the analysis of a biological imaging data set referred to as the Cell Process.
\end{abstract}

\section{Introduction}\label{sec:introduction}
A central goal in science and engineering is to construct a model for a system that allows one to understand its behaviour, explain the main features of its dynamics, and make predictions. 
Much of applied and computational mathematics is motivated by the need to support such an endeavour. 
However, in many contemporary applications, particularly those involving complex systems in biology and physics, the governing equations are unknown or inaccessible, and the available information is finite noisy time-series data. 

First principles based approaches that seek to discover underlying  fundamental governing principles through a reductionist approach 
has proven challenging to use in such a setting.
However, capability to collect an increasing amount of data combined with advances in statistical methods and machine learning has catalysed a growing interest in data-driven approaches where scientific insight is extracted directly from the analysis of data \cite{Weinan:2021aa}.
This work introduces a mathematical framework for data-driven scientific discovery rooted in dynamical systems theory. It is \emph{equation free} in the sense that it operates directly on observational data \emph{without} making assumptions on governing equations for the system that generates the data. Central to the approach is to construct a model whose orbits consist of concatenations of actually observed orbits, i.e., pieces of the observed time-series.

\subsection{Challenges with scientific discovery for complex systems}
Processes that are \emph{cyclic}, that is, they return to the same state after a certain time, are inherently easier to predict on the basis of the process's current state. 
Mathematical analysis provides a well suited framework for making such predictions. In particular, such processes might be successfully modelled in terms of differential equations. 
Many processes in physics, chemistry, biology, and engineering are cyclic, but despite this one cannot generally assume that the process is cyclic. 
This applies in particular to complex physical processes, in fact, non-cyclic behaviour can imply chaos and many non-cyclic processes of interest exhibit \emph{chaotic} dynamics.
The latter is a phenomenon already observed in simple mechanical systems, such as the double pendulum. 

In such situations, modelling becomes more challenging because the dynamics is non-periodic, and reliable prediction from a single fitted differential equation breaks down due to several issues arising from the following two characteristic properties of chaotic systems.
\begin{itdesc}
\item[Sensitivity to Initial Conditions:] 
Chaotic systems exhibit strong sensitivity to initial conditions, meaning that two experiments performed under practically indistinguishable conditions may nevertheless produce markedly different outcomes. Even when the initial state of a second experiment is arranged to be indistinguishable from that of the first, small uncontrolled perturbations may lead to drastically different time-series.
\item[Parameter Sensitivity:] 
The model which describes a chaotic physical process always depends on parameters such as temperature, mass, or external supply conditions that cannot be determined with perfect accuracy. Even very small variations in these parameters may lead to significant changes in the observed behaviour of the model and may turn a model with chaotic behaviour to a cyclic one.
\end{itdesc}

\begin{wrapfigure}{r}{0.35\textwidth}
\begin{center}
  \vskip-3\baselineskip
  \includegraphics[width=0.8\linewidth]{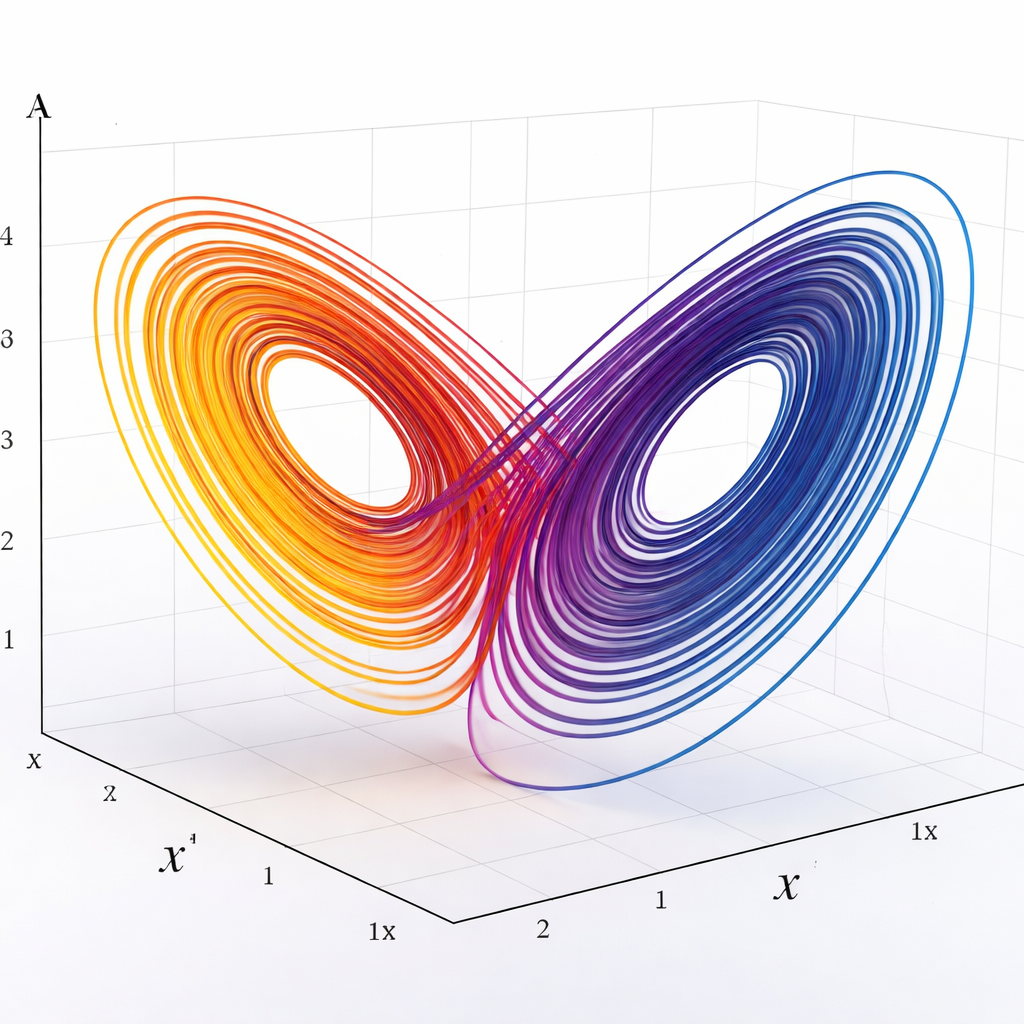}
\end{center}
\caption{The Lorenz butterfly}
\label{Chaos}
\end{wrapfigure}
Sensitivity to initial conditions indicates that even if one knows the explicit governing equations for the system one cannot characterize the actual dynamics. 
In particular, it is not possible to derive explicit solutions.
In addition, parameter sensitivity means there are no stable procedures for recovering the parameters for the governing equations. 
Stated in the language of inverse problems, \emph{both the forward and the inverse problem are ill-posed}.
A classical illustration (\cref{Chaos}) is given by the butterfly effect in the Lorenz system, see \cite{Lorenz:1963}. A general introduction to these issues can be found in \cite{GuckHolmes}. The challenge of modelling chaotic behaviour is therefore intrinsic to the dynamics itself, rather than a consequence of incomplete data or limited knowledge.

\leavevmode
\subsection{An historic account}
The difficulty of deriving explicit differential equations that accurately captured chaotic dynamical behaviour led, already in the nineteenth century, to the development of alternative approaches. Two complementary viewpoints emerged. Ludwig Boltzmann developed a probabilistic description based on statistical properties of the dynamics \cite{Boltzmann:1877}, while Henri Poincaré introduced a qualitative and geometrical approach aimed at understanding the global mechanisms governing the behaviour of solutions, rather than explicit formulas \cite{Poincare:1892}.

Both viewpoints lead naturally to the search for simpler objects that capture the long-term behaviour of the system and allow meaningful prediction. On the probabilistic side, the simplest possible stochastic model are given by finite-state Markov chains \cite{Billingsley,Feller}, where the evolution is described as transitions between a finite number of states, with future behaviour depending only on the present one. From a geometrical perspective, a fundamental class is given by hyperbolic systems, in which nearby trajectories separate in some directions and converge in others in a controlled way. This structure is stable under small perturbations, allowing the global organization of the dynamics to be understood without detailed knowledge of individual trajectories \cite{BrinGarrett,Poincare:1892,PalisTakens}.

An important feature of hyperbolic dynamics is that the qualitative and probabilistic viewpoints can be combined. Instead of following trajectories continuously, one records only the sequence of regions visited by the system. This leads to a symbolic description of the dynamics in terms of sequences of symbols. Such systems, called sub-shifts of finite type, capture the essential transitions between regions of the state space. When equipped with a suitable invariant measure, they become finite-state Markov processes. In this way, hyperbolic chaotic dynamics can be studied using probabilistic tools \cite{BrinGarrett,KatokHasselblatt,Mane}.

Unlike cyclic systems, hyperbolic systems are sensitive to initial conditions. At the same time, both are stable under small perturbations of parameters and thus provide natural models for physical processes subject to uncertainty: small errors in parameters or measurements do not alter their qualitative behaviour. Motivated by this stability, Stephen Smale conjectured that a generic dynamical system should be either cyclic or hyperbolic \cite{Smale:1967}.

While this picture holds in one-dimensional dynamics, it fails in higher dimensions. Sheldon Newhouse constructed families of systems where cyclic nor hyperbolic behaviour is generic, showing that non-hyperbolicity must be expected in general \cite{Newhouse:1974}. Nevertheless, many physical and theoretical processes can often be approximated by hyperbolic dynamics at appropriate geometric scales, and hyperbolic models remain an important point of reference. For an overview of hyperbolic dynamics and the so-called Newhouse phenomenon, see \cite{PalisTakens}.

\subsection{Underlying theory and overview of the approach}\label{sec:OverviewOfApproach}

Bearing in mind the preceding theoretical and historical background, the first step in modelling a physical process is to construct an appropriate dynamical system.
Rather than being defined by differential equations, this dynamical system is built directly from \emph{data}, namely a finite time series of noisy observations of the states of the physical process.

Inspired by the theory of hyperbolic systems, we organize the data into consecutive geometric scales. At each scale, this construction gives rise to a Markov chain whose realizations are obtained by concatenating segments of the original time series. The central question is then: at which scale does the corresponding Markov chain provide an accurate model of the process? For both very coarse and very fine scales, the resulting Markov chains are generally poor models. We therefore formulate the so called \emph{Hyperbolicity Consequences} with corresponding \emph{Validation Criteria} that identify the range of scales for which the Markov chain yields an effective description of the dynamics. At the same time, we quantify the precision of the resulting model.

These validation criteria are motivated by fundamental properties of hyperbolic dynamical systems and play a central role in the construction of the model, as will become clear throughout this section. Their formulation and proofs are presented in the Appendix.
A key important aspect that is that the procedure for constructing the model is robust: small perturbations of the data lead only to small changes in the associated Markov chain.

In conclusion, we stress once again that the model is constructed entirely from data. We make \emph{no assumptions about the underlying physical process}. The model should not be viewed as an approximation of the data, but rather as a structured organization of it.  

Before proceeding with presenting the model and its dynamics, we introduce several fundamental notions from dynamical systems theory.

\subsubsection{Basic dynamical systems theory}
Mathematically, a dynamical system is a map $f \colon X\to X$ where $X$ is a set, the \emph{state/phase space}, whose elements represent the possible states of the system. 
The map $f$ describes how the states of the system change and the orbit allows one to follow changes that a specific state $x\in X$ undergoes: 
\[
\operatorname{Orb}(x) := \bigl\{x, f(x), f^2(x), f^3(x), \dots\bigr\} \subset X
\quad\text{where $f^k := \underbrace{f \circ \ldots \circ f}_{\text{$k$ times}}$.}
\]
The orbit of a periodic point is called a {\it cycle}.
The state space $X$ is often a smooth manifold and the map $f$ is differentiable, but the state space will in our setting be a finite set that is constructed from the aforementioned time series. 

A central notion in dynamical systems theory is the concept of an {\it attractor} $A \subset X$, which is a subset of the state space where typical orbits concentrate over time. 
As such, this smaller set captures the essential structure of the long-term dynamics of the system. 
This also means that attempts at recovering the dynamical system from data can focus on recovering how trajectories of the system organise themselves on the attractor rather than the entire state space.

This above raises the question of how to analyse attractors.
As an example, cyclic behaviour would manifest itself as having closed loops in the state space toward which trajectories of the system converge (periodic attractors).
Not surprisingly, the situation is more complex for systems that are not cyclic. 
Renormalization is one possible mathematical framework for studying non-linear chaotic systems and it provides a systematic way to study the structure of attractors across different geometric scales. 
At each scale there is a partition of the attractor (\emph{dynamical partition}), see \cref{Fig8}. 
Each application of the map $f$ moves the pieces of this partition, giving rise to a symbolic description of the dynamics, see \cref{Fig7}. 
One can think of a dynamical partition as an approximation of the attractor at a given scale. 
This can be represented by an associated graph (\emph{renormalization graph}) with a characteristic structure: it consists of many loops sharing a common vertex.
An orbit of the system can then be viewed as a sequence of points moving between the elements of the partition, and therefore corresponds to a path in the renormalization graph. 
This leads to a symbolic model of the dynamics: the set of all possible paths in the graph forms a dynamical system on a discrete space of symbols that forms a \emph{sub-shift of finite type}. 

In general, given a dynamical system, renormalization will create at each geometrical scale a symbolic system.  These symbolic systems are in general richer than the original given dynamical system as they might contain paths that do not correspond to actual orbits of $f$. 
These renormalization sub-shifts are themselves hyperbolic systems, see \cref{Fig7}.  To summarize, renormalization produces a sequence of hyperbolic models that approximate the attractor at finer and finer scales. 
If the attractor is hyperbolic, there exists a scale at which this symbolic model is equivalent to the original dynamics on the attractor.
\begin{figure}[tbh!]
\centering
\begin{subfigure}{0.47\linewidth}
    \includegraphics[width=\linewidth]{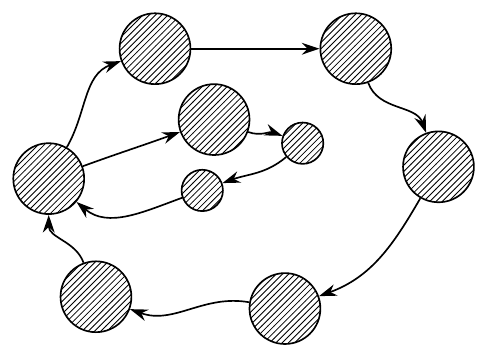}
    \caption{Dynamical partition.}
    \label{Fig8}
\end{subfigure}
\hfill
\begin{subfigure}{0.47\linewidth}
    \includegraphics[width=\linewidth]{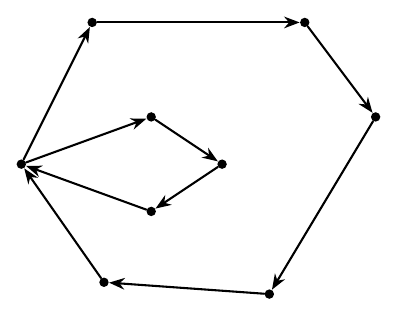}
    \caption{Renormalization sub-shift.}
    \label{Fig7}
\end{subfigure}     
\caption{Illustrating a dynamical partition and the corresponding renormalization sub-shift.}
\label{Fig7and8}
\end{figure}

While the symbolic description captures the combinatorial and geometric structure of the dynamics, it does not by itself provide quantitative information about typical behaviour. 
To address this, one introduces the notion of a \emph{physical measure}. 
To explain this notion, consider an observable represented by $\varphi \colon X \to \Real$, i.e., in each state $x\in X$ one can make a specified measurement giving rise to the value $\varphi(x)$. Given a starting point $x\in X$ one can repeatedly make this measurements along the orbit of $x$. If the system is chaotic then the series
\[ n\mapsto\varphi_n(x) := \varphi\bigl(f^n(x)\bigr) \]
typically behaves in an irregular and unpredictable way. 
Then the \emph{time average} 
\[ 
\lim_{T\to\infty}\frac{1}{T}\sum_{i=0}^{T-1}\varphi\bigl(f^i(x)\bigr)
\]
is one of the few meaningful numbers that one can associate to the observable $\varphi$. 
In many cases, this time average converges to the space average determined by a probability measure $\mu$ supported on the attractor.
This measure is called a \emph{physical measure} and, for typical points $x$ we have
\begin{equation}\label{eq:Birkhoff}
\lim_{T\to\infty}\frac{1}{T}\sum_{i=0}^{T-1}\varphi\bigl(f^i(x)\bigr)
=\int\varphi d\mu.
\end{equation}
This result, known as the Birkhoff Ergodic Theorem, provides a bridge between the dynamics of individual trajectories and statistical predictions \cite{Birkhoff}. 
Observe that the time average is independent of the starting point $x$, i.e., the orbits are distributed throughout the attractor in the same way and this distribution is described by the measure $\mu$. 

Together, the above concepts form a framework needed to construct meaningful statistical models of non-cyclic processes directly from data which are noisy time-series observations from a physical process. Note that the recovery of such statistical models that describes the observed data at a specified scale  is a stable procedure even in the setting when the system exhibits chaotic behaviour.

\subsubsection{Guiding principle for constructing a generative model}\label{Sec:HypCons}
The cycles in the attractor $A$ play a central role in the construction of the model. A chaotic system typically contains a vast number of periodic orbits, but they are not all equally informative. Suppose for a moment that the physical process were perfectly described by a hyperbolic dynamical system with a physical measure. Then already a small, finite collection of periodic orbits would reproduce the long-term statistics of the entire system to any prescribed accuracy. We refer to such orbits as \emph{typical} cycles. The remainaing periodic orbits, the vast majority, do not reflect the statistics of the system and are therefore statistically irrelevant. They are however , also almost never observed in practice. Since the cycles used in the construction of $A$ arise from actual measurements and observations, it is natural to regard them as candidates for typical cycles and to investigate whether they indeed capture the statistical behavior of the system.

As explained above, and motivated by Smale's conjecture that generic dynamical systems are hyperbolic, we use hyperbolicity as a guiding principle throughout the construction of the model. Our objective is not to prove that the physical process is hyperbolic, something that is impossible to establish from a finite data set, but rather to determine whether the observed dynamics exhibit at some scale characteristic properties of hyperbolic systems. To this end, we derive several consequences of hyperbolicity and formulate corresponding Validation Criteria.

These criteria play a dual role. First, they provide a systematic procedure for assessing whether the observed dynamics resemble those of a hyperbolic system. At each stage of the construction, satisfaction of a criterion supports the continuation of the model-building process, whereas failure of a criterion may indicate either insufficient data or a more fundamental issue of the physical process. Second, the criteria allow us to distinguish statistically relevant periodic orbits from the vast collection of periodic orbits present in a general chaotic system. Indeed, while chaotic systems typically possess infinitely many periodic orbits, only some of them share common statistical properties and collectively capture the statistical behavior of the system. The Validation Criteria are designed to identify such families of orbits. We check that the cycles extracted from the data satisfy these criteria and consequently exhibit the expected common statistical features. This provides evidence that the observed cycles belong to the physically relevant class of periodic orbits and can therefore be used to approximate the underlying physical measure.

\subsubsection{Outline of the method}\label{sec:Outline}
We here present an outline of the data driven method for constructing a probabilistic generative model for a dynamical system from finitely many observations.  
The full detailed description is given in \cref{sec:method}.

As already mentioned, we assume that data represents a time series of noisy observations generated by a physical process.
More precisely, at each time instance we record a data point, and all these data points lie in the same high-dimensional Euclidean space. 
The dimension depends on the number of measurements that are recorded at each time instance. 
The collection of all these points forms the state space $X$ and one should anticipate that the state space may contain a very large number of points, possibly millions.
Since data are sequential, they naturally define a dynamical system $f \colon X\to X$: the map $f$ 
simply sends each observed state to the state observed at the next time instance.". 

\paragraph{The attractor:}
The first step is to determine the attractor for the aforementioned dynamical system, which in this discrete setting consists of finitely many periodic orbits where each periodic orbit carries an invariant measure.

Data obtained by actual measurements never return to exactly the same state, so periodicity is here understood at a given geometric scale: an orbit that returns sufficiently close to its starting point is regarded as a cycle at that scale.

Before proceeding, it is important to emphasize that the attractor is expected to be much smaller than the original state space, forming a reduced subset that nevertheless captures the essential dynamics of the system. 
This expected compression is motivated by the fact that in dissipative systems, i.e., systems with friction, one expects the attractor to occupy only a small portion of the state space.

\paragraph{Renormalization and Markov chains:} The renormalization process, as described in the previous paragraph, can be applied to the attractor and produces, at consecutive geometrical scales, a renormalization sub-shift of finite type. The invariant measure associated with the periodic orbits then induces a Markov structure on these symbolic systems.  
Left is to select the appropriate geometrical scale at which the corresponding Markov chain is a reasonable model. In this way from the consecutive geometrical scales we obtain a graded set of Markov chains. 

\paragraph{Entropy as notion of model complexity:}
Entropy is a fundamental quantity that quantifies the complexity of a system, intuitively, how much new information the system produces per step as it evolves. In accordance with the guiding principle underlying our construction, each stage of the model construction is motivated by the behaviour of hyperbolic dynamical systems with physical measures. In this setting, the entropies of the Markov chains associated with different scales vary in a controlled manner, departures from a controlled behavior indicates  that a scale is too coarse or too fine. This observation leads to an additional validation criterion, which can be used to identify the scale at which the Markov chain provides the most appropriate description of the physical process, thereby selecting the relevant model.

\paragraph{Summary and conclusions:}
The states of the aforementioned Markov chain model correspond to points in the original state space. 
Since their coordinates are given by actual measurements, they can be used directly to make probabilistic predictions about the evolution of the process.
Hence, the Markov chain is a probabilistic generative model that describes the long-term behaviour of the physical process at a specified geometrical scale.

It is important to note that reliable statistical conclusions require sufficiently long time series. 
Even in simple random processes such as coin tossing, one needs a large number of trials before observing the asymptotic behaviour predicted by the central limit theorem. 
Similarly, in chaotic dynamics, meaningful statistical properties emerge only over long time scales. 
If the hyperbolicity-based validation criteria are not satisfied during the construction, this may indicate that the available data are insufficient.

\subsubsection{Detecting and recovering emergent properties from data}\label{subsec:LinkToIP}
A key feature of physical systems that evolve in time is that beyond a certain complexity, they acquire emergent properties—characteristics of the system's collective behaviour that cannot be understood by studying individual components alone, one needs to study the whole system's dynamics.
Examples of such emergent properties are phase transitions, pattern formation, synchronization phenomena, self-organization, or critical phenomena at phase transitions.

A particularly well-known setting is hyperbolic system with a physical measure, which are known to behave like a Markov chain.
A natural task is therefore to determine whether data generated by a physical system displays such behaviour at an appropriately chosen scale. 
If this turns out to be the case, then this Markov model constitutes an emergent property for the physical process.
For example, at that scale all parts of the physical process, in our application,  all regions of the cell population, behave statistically in the same way. Such a property cannot be understood by studying any single part in isolation, it belongs to the system as a whole. A related task is to recover the aforementioned Markov model from the observed data in a stable manner.
The latter simply means that small variations in data results in small changes to the recovered Markov model.

The method outlined in \cref{sec:Outline} provides tools for both tasks above: deciding whether the physical process behaves like a Markov chain and, if so, recovering that Markov chain.

The reasoning behind the first task deserves to be spelled out, since it lies at the heart of the method. The renormalization procedure of \cref{sec:Outline} can be applied to any time series, and it always produces a graded set of Markov chains; producing a Markov chain is therefore not, by itself, evidence of anything. However, among all Markov chains that could conceivably come out of this procedure, only a very small class can arise from data generated by a hyperbolic dynamical system with a physical measure. This class is characterised by the Hyperbolicity Criteria, which are concrete conditions that can be checked in practice. A time series picked at random would almost certainly fail them. Hence, if the graded set of Markov chains obtained from the observed data do satisfy the criteria, the coincidence is too improbable to dismiss as incidental, and we conclude that the data are generated by a physical system that behaves like a Markov chain. This is the emergent property one typically expects to see in hyperbolic dynamical systems. A formal statement of this argument is given in \cref{sec:method}.

To summarize the above, the inverse problem of recovering the governing equations of a physical system from a single observed noisy time series that is generated by this system is highly ill-posed, meaning that there are many possible equations that ``fit the data'' and a small change to data is likely to result in very different set of equations.
One strategy is therefore to focus on recovering a property of the physical system generating the data, which here is a Markov chain that models its dynamics and thus represents an emergent property. 
However, before attempting this through renormalization, one must first test whether it is possible, i.e., whether the physical system generating data can be well approximated by a hyperbolic system. 
The hyperbolicity criteria offers a test for this and they are also used to single out the appropriate Markov model from the graded Markov chains obtained by applying renormalization to the data.
As this process is stable with respect to perturbations to data, the entire process can be seen as a ``regularization'' where renormalization is used to rephrase the initial inverse problem. The hyperbolicity criteria are used to single out the appropriate scale where the associated Markov chain constitutes the emergent property of the system that one can infer from data.

\subsection{Application to scientific discovery in cell signalling dynamics}\label{sec:CellProcess}
To demonstrate the data-driven approach outlined in \cref{sec:Outline}, we consider a scientific discovery task drawn from cellular and molecular biology: modelling cell signalling dynamics from microscopy data. This is central in understanding the complex and intricate behaviours of living cells. It will also serve as a running example throughout the paper, illustrating each step of the method in turn.

As with any data-driven approach, the success of the approach outlined in \cref{sec:Outline} rests on having access to sufficient amount of observations, which for this use case are measurements of spatiotemporal cell signalling dynamics. Obtaining such measurements with high precision and sensitivity requires advanced imaging technologies, and emerging tools for in situ spatiotemporal monitoring have significantly enhanced the ability to gather such data from living cells \cite{Allport:2001aa,Hsieh:2025aa,Zhuang:2026aa}.
Unlike traditional experimental methods that often rely on fixed time-point analyses, such in situ monitoring offers dynamic, real-time visualization of cells, allowing researchers to investigate processes such as proliferation, differentiation, and migration under physiological conditions \cite{Wang:2008aa}.
Live-cell imaging, a groundbreaking development in this area, enables the continuous observation of cellular activities without disturbing the natural cellular environment \cite{Schnell:2012aa}. 
Fluorescence microscopy, a core technique in live-cell imaging, has undergone major advances, delivering unprecedented spatial and temporal resolution. Using fluorescent probes and dyes, researchers can selectively visualize specific cellular components, track molecular interactions, and explore dynamics within cells (at the single-cell level) or in-between cells in a population. 


The above remarkable progress in fluorescence microscopy based techniques for live-cell imaging can with appropriate choice of fluorophore be used to image the dynamically changing Calcium (Ca${}^{2+}$) levels within a living cell population. Oscillations in these play fundamental roles in various cell signalling processes and have been the subject of numerous modelling studies, like in \cite{Kowalewski:2006aa} that attempts to model these oscillations with differential equations based on live cell data from light-sheet fluorescence microscopy.
We will showcase how our data driven method outlined in \cref{sec:Outline} can be used to address the scientific discovery task of modelling the oscillations of these Calcium levels without introducing any a priori differential equations.  
\Cref{sec:DescExp} gives a more detailed account of the particular cell population and how this data is collected.

\subsection{Related methods}\label{sec:OtherMethods}
Recovering a dynamical system from time series data that represent finitely many noisy observations of its states is a central task in scientific discovery.
Scientific progress has for centuries relied on manually deriving such models from empirical observations  \cite{Kepler1609, Newton1687, Fourier1822, Ohm1827, Fick1855}.

The above is also a long-standing research area within the computational sciences \cite{HoKalman1966, BellmanAstrom1970, Langley1981, Langley1987, NarendraParthasarathy1990, BongardLipson2007, Schmidt:2009aa, Simpkins2012}, where most attention has been directed at recovering time-continuous dynamical systems modelled by a differential equation whose flow map is sampled in time. 
Time-continuous dynamical systems are typically defined as an ordinary differential equation governed by a vector field $f \colon \StateSpace \to \Real^n$ on a state space $\StateSpace\subseteq\Real^n$ with associated flow $\varphi_t \colon \StateSpace \to \StateSpace$.
The flow relates to the vector field through an initial value problem whose solution defines a flow $t\mapsto\varphi_t(z)$.
The observed time series data $x_k = h\bigl(\varphi_{t_k}(z_0)\bigr)+e_k \in \Real^n$ for $k=1,\ldots, m$ with $t_1< t_2 < \ldots < t_m$ are observations of this flow with $e_k$ denoting the (random) observation error and $h \colon \StateSpace \to \Real^n$ is the observation map. 
This map is the identity when the full state is observed, but it can be non-injective when only part of the state is observed.

The task of reconstructing a dynamical system from the time series data $x_1, \ldots, x_m \in \Real^n$ is traditionally phrased as the task of recovering the vector field $f$, or equivalently the flow $\varphi_t$. 
Existing methods for addressing this task are typically divided into \emph{parametric} and \emph{equation-free} methods.
Parametric methods fix a priori a family $\{f_\theta\}_{\theta}$ of admissible vector fields and the task is now to recover one member of this family.
Equation-free methods assume no such functional form and instead return a surrogate that summarises the dynamics, such as 
\begin{inparaenum}[(a)]
\item invariant sets along with their stability type and the orbits connecting them, 
\item Markov approximations,\label{item:OurMethod}
\item the Koopman or transfer operator and its spectral data, and
\item a topological or order-theoretic skeleton of the global dynamics.
\end{inparaenum}  
\emph{The method developed in this paper (outlined in \cref{sec:Outline,subsec:LinkToIP}) falls into category \cref{item:OurMethod} above.} In the following, we here survey some of the other approaches.

\subsubsection{Parametric methods}
Classical approaches fix the governing equations, so reconstruction reduces to parameter identification.
This inverse problem is often ill-posed \cite{BellmanAstrom1970, CobelliDistefano1980, DistefanoCobelli1980, NguyenWood1982, Miao2011}, and early work regularized it by restricting the function class to a narrow family described by a few physical parameters \cite{Alessandrini1986, Acar1993, Knowles2001}. 
However, such a strong prior limits applicability of the method.

\emph{Sparse regression} widens the scope by replacing the narrow function class with a parsimony criterion encoded by sparsity-promoting regularization.
This idea of exploiting sparsity in solving ill-posed inverse problems has a rich history in mathematics.
Parts of it can be traced back to Beurling's 1938 work on ``minimal extrapolation'' \cite{Beurling:1938aa} and it has since then evolved into a rich theory for solving ill-posed inverse problems \cite{Santosa:1986aa, Donoho:1992aa, Candes:2006aa, Scherzer:2009aa, Foucart:2013aa}.
Its application to parameter identification in dynamical systems from data $y_i\approx\dot x(t_i)$ was promoted with the introduction of SINDy \cite{Brunton:2016aa} where sparsity implied seeking a model that is expressible as a linear combination of as few elements as possible from a fixed family of functions (dictionary).
Stated formally, one assumes there is a dictionary $\{\phi_k\}_{k=1}^{s}$ and the task is to recover a representation $f\approx\sum_k\alpha_k\phi_k$ in $\dot x=f(x)$ with as few active terms as possible, i.e., $\hat{\boldsymbol\alpha}=\arg\min_{\boldsymbol\alpha}\lVert\Phi\boldsymbol\alpha-\mathbf{y}\rVert_2^2+\lambda\lVert\boldsymbol\alpha\rVert_0$ with $\Phi_{ik}=\phi_k(x_i)$.
However, the $\ell_0$-problem is non-convex and NP-hard \cite{Natarajan1995}, so a common approach is convex relaxation that replaces the problematic $\ell_0$-term with a convex $\ell_1$-term, which in turn can be solved with proximal gradient methods.
SINDy's sequentially thresholded least squares approach is hard thresholding with an exact solve on the active support.
The $\ell_0$-problem is however only equivalent to its convex relaxation under coherence conditions that are hard to verify \cite{DonohoElad2003, CohenDahmenDeVore2009}. 
Extensions of SINDy cover PDEs \cite{Rudy:2017aa}, richer dictionaries \cite{Kaheman:2020aa, Purnomo2023}, ensembling \cite{Fasel:2022aa}, control, implicit and non-smooth dynamics \cite{Fasel2021, Quade2018}, and weak formulations that avoid differentiating noisy data (to generate the $y_i$ values)  \cite{Messenger2021WeakSINDyODE, Messenger:2021aa, Qian2022}, with applications ranging from plasma physics to epidemiology \cite{Kaptanoglu:2021aa, Lagergren:2020aa, Horrocks:2020aa}.
A key aspect of sparse regression is the need for a dictionary that sparsifies the true dynamics, which is also what carries the interpretability, and a parameter selection rule for setting the value of the regularization parameter $\lambda>0$ given an estimate of the noise level in data. See \cite{Fung:2025aa} that selects dictionary atoms through marginal likelihood and thereby sidesteps the need for having an explicit parameter selection rule.

\emph{Symbolic regression} drops the dictionary and searches over expression trees, thereby recovering structure and parameters jointly \cite{North2023, YuWang2024, Makke:2024aa}.
The search is NP-hard \cite{Virgolin:2022aa}, so beyond short expressions \cite{Bartlett:2023aa} one resorts to heuristics: genetic programming \cite{Koza1992, Schmidt:2009aa, Cranmer2023}, continuous relaxations over compositional bases \cite{MartiusLampert2017, Sahoo2018, Scholl2025}, or reinforcement learning \cite{Petersen2021}, often combined with pre-trained transformers and language models \cite{Biggio2021, Kamienny2022, Grayeli2024, Shojaee2025}.
Benchmark studies are given in \cite{LaCava:2021aa} and applications to the discovery of ``natural laws'' in \cite{Schmidt:2009aa, Hillar:2018aa, Liu:2021aa, Ahmadi:2023aa}.
SINDy can be seen as symbolic regression restricted to a linear ansatz over a fixed library, which buys convexity and speed at the price of expressiveness.

\emph{Neural parametrisations} keep this setting but forfeit interpretability. 
Most of these methods use the observed time series data $x_k$, which is unsupervised, to manufacture supervision from the time ordering by forming input/output pairs $(x_k,y_k)$ with $y_k=x_{k+1}$.
One may represent the vector field $f$ with a deep neural network \cite{Raissi2018, Chen2018, Greydanus2019, Cranmer2020} that is trained by a trajectory rollout loss that admits irregular sampling and suits time-continuous data.
Another option is to represent the flow map with a deep neural network \cite{Bilos2021, Canizares2024} that can be trained by a cheaper one-step loss whose errors compound over the rollout.
Operator learning lifts the learning to function space \cite{Lu2021, Kovachki2023}, where Fourier neural operator architectures \cite{Li2021} are a popular choice.
One can also set up learning that accounts for side information, which can be any knowledge about the dynamical system one seeks to learn, besides trajectory data \cite{Ahmadi:2023aa}.
Furthermore, much attention has been devoted to architectures that account that embed handcrafted physical laws \cite{Cornelio2023, Yu:2023aa, Srivastava2025} or symmetries \cite{UdrescuTegmark2020} improve generalization, and symbolic distillation partly restores interpretability \cite{Cranmer:2020aa, LaCava:2021aa}, but extrapolation and hallucination still remain as barriers to adoption \cite{Vafa2025}.
Other approaches use deep learning to resolve the differences between the true dynamics of the system and the dynamics given by a model of the system that is either inaccurately or inadequately described \cite{Qraitem:2020aa}.
Applications include weather forecasting, fluid dynamics, and plasma modelling for nuclear fusion \cite{Kurth2023, Gopakumar2024}.

On a final note, an inverse problems perspective of dynamical system recovery is given in \cite{Engl2009}, which also treats \emph{qualitative inverse problems} where the goal is to identify regions in parameter space that correspond to dynamical systems exhibiting some prescribed qualitative dynamics, like bistability, limit cycles, or a given bifurcation, rather than fitting a trajectory.
The dual question of which data determine which qualitative dynamics regime is addressed in \cite{Duan2023}.

\subsubsection{Equation-free methods}\label{sec:eqfree}
Here the prior question is which features of a dynamical system are (stably) identifiable from time series data. 
A survey that organises the answer around the transformation law, invariant sets, the invariant measure, the Koopman or transfer operator, and Markov approximations, together with convergence results in different topologies, is provided in \cite{BerryDas2025}.

Formal identifiability is settled for linear systems \cite{Stanhope2014, Qiu2022, Casolo2025} and, through necessary and sufficient conditions, for broad non-linear classes of ODEs and PDEs \cite{Scholl2022, Scholl2023, Hauger2024}.
A recent paper also shows somewhat counter-intuitively that chaos, which is a property typically associated with unpredictability, is a crucial ingredient for ensuring that a dynamical system is identifiable \cite{Shumaylov2026}.
However, none of these works address the issue of instability in recovering the dynamical system.

\emph{Delay embeddings} recover a state space from a scalar record.
Takens' theorem \cite{Takens1981, SauerYorkeCasdagli1991, Takens2010} provides the theoretical basis stating that for a suitably chosen delay $\tau$ and dimension $m$, the map $q(t)=\bigl(x(t),\ldots,x(t-(m-1)\tau)\bigr)$ embeds the attractor of the dynamical system.
The embedding is, however, expressed in coordinates other than the states, so quantities that are not invariant under smooth conjugacy, like components of the vector field $f$, are not possible to recover within this framework.
Moreover, it may be arbitrarily ill-conditioned, so Takens' theorem yields a form of injectivity but not well-posedness.

\emph{Koopman learning} lifts the dynamics to observables, where $[\Koop_t h](z)=h(\varphi_t(z))$ that acts linearly even when $f$ does not \cite{Mezic2005, BudisicMohrMezic2012, Brunton2022, Colbrook:2026aa}.
The price for linearity is that the Koopman operator is infinite-dimensional, and since $\Koop_t$ is neither compact nor normal, its finite compressions need not converge in norm.
Dynamic mode decomposition (DMD) is the special case that fits a linear map to snapshot pairs \cite{Schmid2010, Kutz2016}, whereas extended dynamic mode decomposition (EDMD) is a family of data-driven operator approximation methods that extend the standard DMD framework by lifting finite-dimensional state data into a finite-dimensional subspace of observables (dictionary functions) and seeking a finite-dimensional approximation of the (infinite-dimensional) Koopman operator that governs the evolution of observables along the flow.
One can show that EDMD converges strongly as the number of data points and number of dictionary atoms grow \cite{KordaMezic2018}.
EDMD with Galerkin regression over a dictionary with empirical quadrature is provided in 
\cite{WilliamsKevrekidisRowley2015}.
ResDMD replaces the unstable recovery of eigenvalues of the finite compressions of the Koopman operator by the stably recoverable residual \cite{ColbrookTownsend2024}, and under partial observation the dictionary comes from delay coordinates \cite{ArbabiMezic2017}.
Other variants on Koopman learning is Koopman autoencoders \cite{Lusch2018, OttoRowley2019}, which amount to EDMD with a learned dictionary whose linearity penalty is the EDMD residual.

Using pointwise losses in recovering a dynamical system beyond the Lyapunov time-horizon is meaningless. 
Here, the only reasonable target left to recover is the invariant measure, which is what distributional training against optimal-transport or MMD discrepancies aims at in \cite{BotvinickGreenhouse2023, Schiff2024}, while score-based, diffusion, and latent-SDE formulations model the observation process instead of assuming it away \cite{Gilpin2024}.
The approach outlined in this paper is also an instance for building a Markov model from the invariant measure that is recovered from data.

A notable deep learning based approach is \cite{Floryan:2022aa} which uses a neural network to learn a manifold representing a dynamical system's intrinsic state variables directly from time-series data.
Another deep learning based approach for non-parametric learning of the governing structures of collective dynamics is given in \cite{Zhong:2020aa}. 
Finally, \cite{Kacprzyk:2025aa} uses large language models for direct semantic modeling, which is the task to predict the semantic representation of the dynamical system (i.e., a natural language description of its qualitative behavior) directly from data.

\emph{Conley index theory} \cite{MischaikowMrozek2002} confronts the discontinuous dependence of invariant sets on $f$ by replacing the invariant set with an isolating neighbourhood that is stable under perturbation.
In particular, one can on such a neighbourhood certify equilibria, periodic and connecting orbits, and semiconjugacy onto a subshift with an entropy bound.
Data enter through an outer approximation on a cubical grid, so every consistent system is a selector and each conclusion is a theorem about the unknown system \cite{Mischaikow1999, Batko2020,MischaikowActa2002,KaczynskiMischaikowMrozek2004}.
An accurate piecewise-linear network determines the Morse decomposition and nothing finer \cite{GameiroGelbMischaikow2025}, and learning a qualitative descriptor directly \cite{KacprzykVanDerSchaar2025} is the learning-theoretic counterpart of this observation.

On a final note, most of the above listed equation-free methods rely on numerical algorithms for analysis of dynamical systems and their bifurcations \cite{Guckenheimer2002, DellnitzJunge2002}.

\section{Detailed description of the method}
\label{sec:method}
The goal is to organize the data. This organization is in the form of Markov chains for consecutive geometrical scales and the model is one of these chains.
The Guiding Principle allows to construct the Markov chains as long as the validation criteria are satisfied up to a fixed \emph{uncertainty threshold}. The requirements of the application determine what should be regarded as uncertainty threshold.
\begin{example}[Cell signalling dynamics]
The uncertainty threshold is set to $\delta=0.25$ for data from the cell signalling dynamics (\cref{sec:BiologicalUseCase}).
\end{example}
\subsection{Time-series} 
Consider numerical measurements of an observable of a physical process. 
Usually there are many relevant observables, however one needs instruments to be able to make measurements. 
In particular it is often impossible to give a complete description of the process, since we have a limited set of instruments. Let $D$ be the number of accessible observables. 

Suppose that there is a set of observables such that our physical process allows to make repeated measurements of these observables. 
In principle the measurements might be made continuously, but in practice at the moment that the measurements will be processed they will consists of finite number of time points.  Hence, for all practical purposes we may assume that  measurements are obtained at discrete time intervals of length $\Delta T$.
This gives rise to time series, say of length $T$:
\[ 
\varphi \colon \bigl\{1,\dots,T\bigr\}\mapsto \Real^{D}.
\]
\begin{example}[Cell signalling dynamics]
Observations (data) is a sequence of $512\times 512$ grey-scale images, i.e., a time series of digitised 2D images (movie), see \cref{sec:BiologicalUseCase}. 
More specifically there are $4\,318$ such images with $\Delta T=5$~sec, i.e., $T=4\,318$ and $D=512\times 512$.
\end{example}

\subsection{Projected time-series}\label{proj}
 In general, the higher the dimension of the data, the more data points are needed for statistical inference to succeed.
It may therefore happen that the dimension $D$ of the time series $\varphi$ is too large relative to its length.
We present here a method that applies when each data point represents a continuous function: it reduces the dimension to some $d < D$ without discarding measurements. 
Instead of throwing data away, we rearrange it, and in doing so simultaneously increase the number of time series, as explained in the next paragraph. 
Conceptually, the procedure therefore trades ``resolution in state space'' for a larger number of time-series samples. 
Note that this step can be skipped whenever the dimension $D$ is already manageable.

Choose a sequence of physically relevant projections $\pi_k \colon \Real^{D}\to \Real^{d}$ for $k=1,\dots,K $. These projections create new time series (called \emph{projected time series}) by $\varphi_{k}=\pi_k\circ\varphi$.
The set of projected time series will henceforth be denoted by $\Phi=\{\varphi_k\}_{k=1,\dots,K}$.
There is no a general rule for how to choose `informative' projections, but the specific physical process might provide an indication of natural choices. 

The use of projected time series has two advantages. It increases the number of time series while reducing the resolution insignificantly and the second advantage is that it allows to incorporate chain recurrence in the model. In particular, it gives the possibility to use Markov chains as models. Chain recurrence is discussed in details in \cref{RevAttr}.

\begin{example}[Cell signalling dynamics]
The data in \cref{sec:BiologicalUseCase}, i.e., a time series of $512\times 512$ pixel images, is projected to a coarser image. 
To do so we divide each image into $16\times 16$ smaller squares and then choose a single pixel from each square. Each cell corresponds roughly speaking with a square.
We call this selection of pixels a \emph{grid}, and each grid defines a projection. 
In our specific example we chose $1\,000$ random grids, i.e., $K=1\,000$. 
In particular, there are projections 
\[ 
\pi_k \colon \Real^{512\times 512}\to\Real^{16\times 16} \quad\text{for $k=1,\dots, K$.}
\]
Depending on the specific needs one might need to design more elaborate projection methods, i.e. where the projections are time dependent. For example one could take the pixels of the grid inside pre-described cells. 
The projections will be time dependent because the cells are moving. To illustrate the method we use the presented projections.
\begin{figure}[!htb]
\centering
\includegraphics[width=0.6\textwidth]{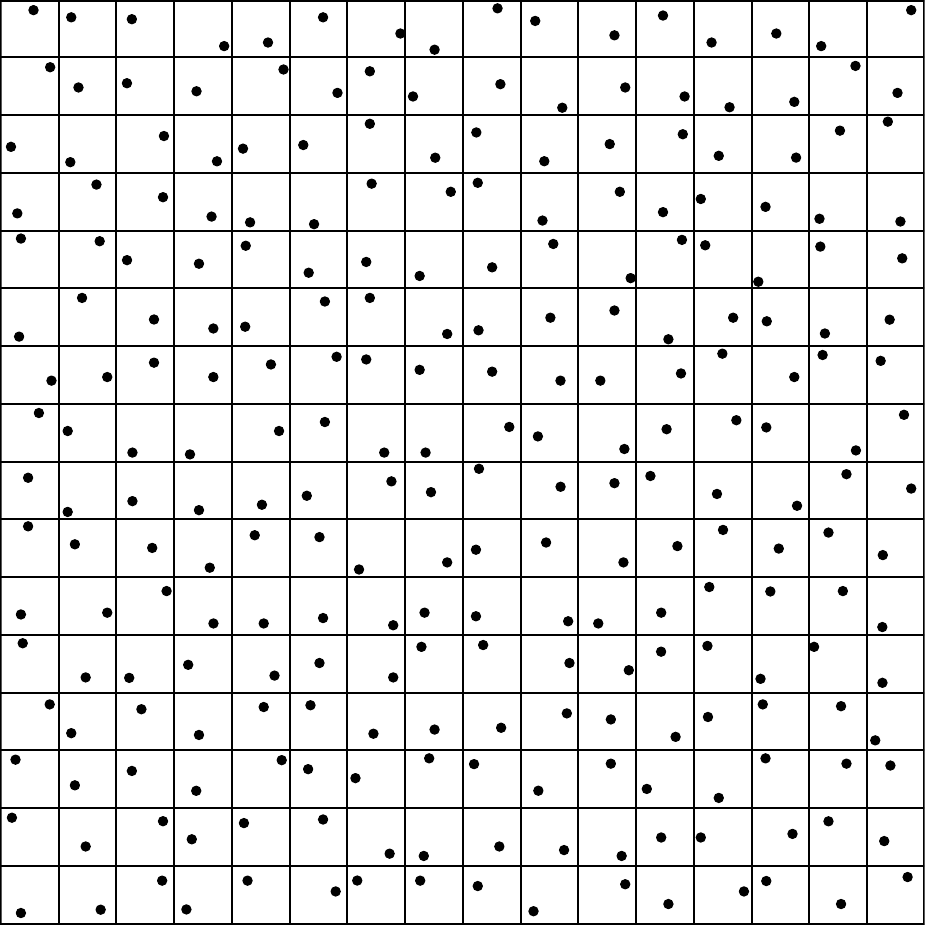}
\caption{Example of a $16\times 16$ grid generated by randomly selecting pixels (dots), one in each square, in a $512\times 512$ pixel image.}
\label{Fig2}
\end{figure}
\end{example}

\subsection{The discrete differential equation}\label{DDE}
The first task  is to construct a discrete dynamical system that presents the projected time-series $\Phi$ above. The state-space $X$ for such a dynamical system is as follows:
\[
X=\bigcup_{k}\varphi_{k}\bigl(\{1,\dots,T\}\bigr)\subset\Real^d.
\]
Observe that the size of the state space $X\subset\Real^d$ is $K\times T$. The evolution of the process is described by a map
$f \colon X\to X$, where given a point $x\in X$, the image $f(x)$ is the successor of it. 
\begin{description}
\item[Successor Assumption:]
Every point $x\in X$ has a successor in $X$ which is denoted by $f(x)$.  
\end{description}

We next describe how to construct the successor. Each point in $X$ is of the form $x=\varphi_k(t)$ with $t\leq T$. Hence, if $t<T$ then the only natural choice for the successor is
\[
f(x)=\varphi_k(t+1)\in X.
\]
If $x=\varphi_{k}(T)$ it  is at the end of a time series and there is no natural successor as in the case when $t<T$. However, Successor Assumption assures that a successor always exists. How do we choose it explicitly? 
The successor $f(x)$ of a point $x=\varphi_{k}(T)$ is defined as the closest point $y\in X_s$ where
\[
X_s=\bigcup_{i,k}\varphi_{k}(\left\{1,\dots,s\right\})\subset X\subset \Real^d
\]
and $s\le 0.75 \cdot T$. The motivation for the choice of $f(\varphi_{k}(T))\in X_s$ is the following. Most likely the closest point in $X$ to $\varphi_k(T)$ is $\varphi_{k}(T-1)$. However, this point is really the predecessor of $\varphi_k(T)$ and we need the closest point in the future. To address this issue we chose the closest point from the first sections of the time series defined by $s$. Formally we define the dynamical system $f \colon X\to X$ as follows. Choose $s\le 0.75 T$. Then 
\[
f(x)=\begin{cases}
\varphi_{k}(t+1) &\text{ if $x=\varphi_{k}(t)$ with $t<T$}
\\
\varphi_{\hat{k}}(\hat{s}) & \text{ if $x=\varphi_{k}(T)$,}
\end{cases}
\]
where $\varphi_{\hat{k}}(\hat{s})\in X_s$ is the closest point to $x$, namely, 
\[
d\bigl(\varphi_{\hat{k}}(\hat{s}) ,x\bigr)=\min\bigl\{d(z,x)\bigl|\bigr.z=\varphi_{m}(t), t\leq s,m\leq K\bigr\}.
\]
Observe that the point $\varphi_{\hat{k}}(\hat{s})\in X_s$ is not uniquely defined. In case of multiple choice one has to specify the definition of $\varphi_{\hat{k}}(\hat{s})$. 

To gain the proper intuition for what a discrete dynamical system is, one can reformulate its dynamics as a discrete differential Equation
\[
\Delta x= V(x) \Delta t,
\]
where the discrete vector field is given by 
\[
V(x)=\frac{f(x)-x}{\Delta t}.
\]
Another example of a commonly used discrete dynamical system occurs when one simulates numerically the solutions of a differential equation. Basic examples are the Runge-Kutta method or the Euler method. 

The next subsections discuss tools to analyse the dynamics of the 
discrete dynamical system $f \colon X\to X$.  

\begin{example}[Cell signalling dynamics]
We choose $s=0.5 \cdot T$ for the data in \cref{sec:BiologicalUseCase}, so the number of points in the state space $X$ is $4\,318\,000$. 
A two dimensional projection of $X$ was obtained by projecting to coordinates 113 and 97, see \cref{Fig:4b}. Moreover there were no multiple choices for the points $\varphi_{\hat{k}}(\hat{s})$ at minimal distance to $x$. The distance used on $\Real^d$, $d=256$, is the max distance.  
\end{example}
\begin{figure}[h]
\centering
\includegraphics[width=0.6\textwidth]{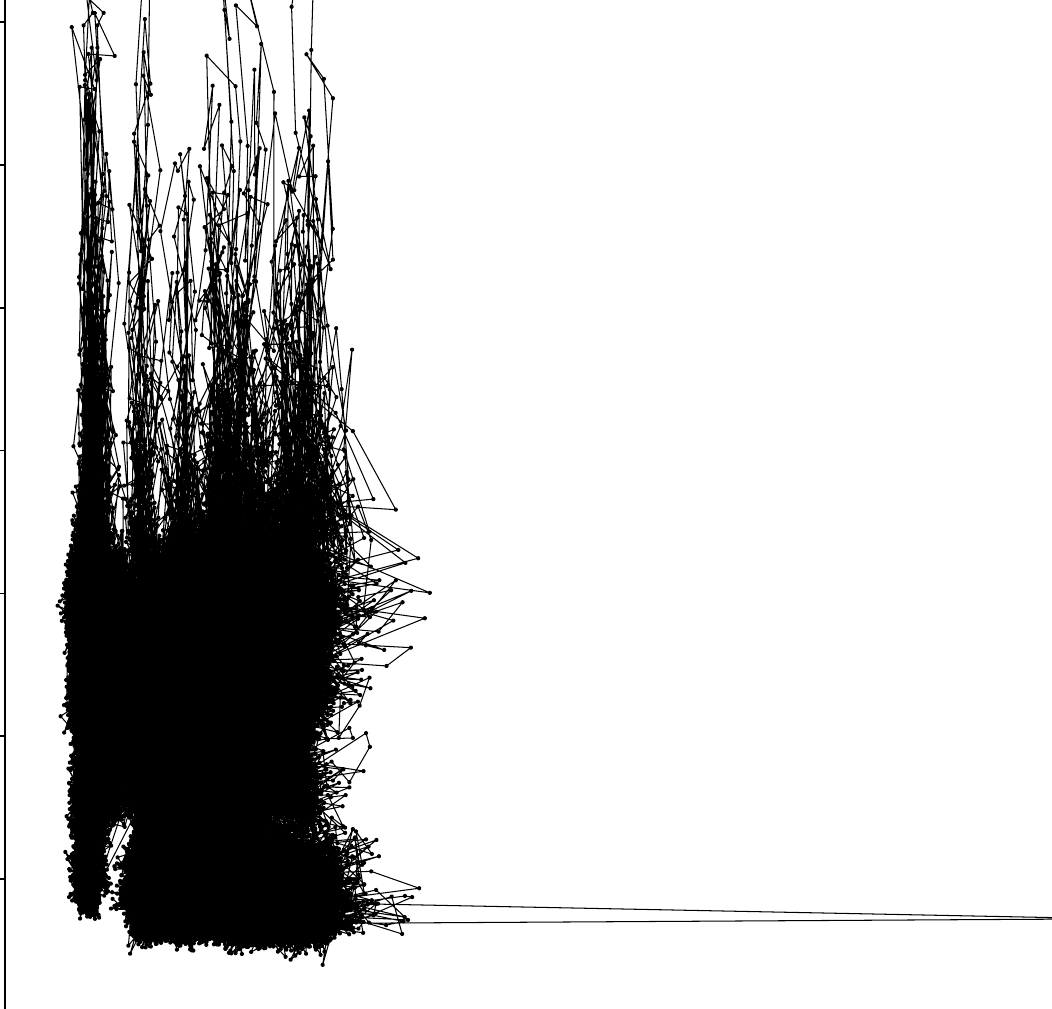}
\caption{A Two-Dimensional Projection of the state space $X$.}
\label{Fig:4b}
\end{figure}

\begin{figure}[h]
\centering
\includegraphics[width=0.6\textwidth]{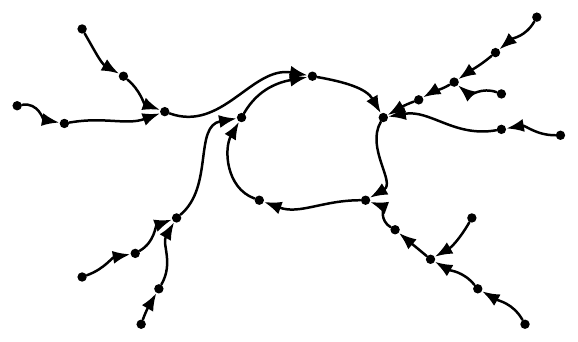}
\caption{A Cycle and its Basin}
\label{Fig3}
\end{figure}

\subsection{The attractor}\label{attr}
Let $f \colon X\to X$ be a discrete dynamical system. The \emph{orbit} of a point $x\in X$ is defined as
\[
\operatorname{Orb}(x) :=
\bigl\{f^n(x)\bigl|\bigr. n\geq 0 \bigr\}.
\]
The orbit of $x$ can be considered as the solution of the discrete differential equation with $x$ as initial point.

Next, a point $x\in X$ is a \emph{periodic point} of $f$ is there exists $p>0$ such that $f^p(x)=x$. The \emph{period} of $x$ is the smallest of such $p$. 
Similarly, $x\in X$ is called \emph{eventually periodic} if there exists $n\geq 0$ such that $f^n(x)$ is periodic. Denote the orbit of a periodic point $x$ by $P=\operatorname{Orb}(x)$. Such an orbit is called a \emph{cycle}. 
The \emph{basin} of the cycle $P$ is defined as 
\[
\operatorname{Basin}\{P\} :=
\bigl\{
  x\in X \bigl|\bigr.
  \text{there exists $n$ such that $f^n(x)\in P$}
\bigr\}.
\]
Observe that if $P$ and $Q$ are periodic orbits then either $P$ and $Q$ are disjoint or are equal. Moreover the restriction of the map $f$ to a cycle is a bijection. Finally, the set of all periodic points of $f$ is in particular the union of cycles. 
\begin{lemma}[Attractor Structure Lemma] 
The set $A$ of all periodic points of $f$ is invariant under $f$, i.e. $f(A)= A$. In particular, $f \colon A\to A$ is a bijection. Moreover, 
\[
X=\bigcup_{P\subset A}\operatorname{Basin}(P)
\quad\text{where $P=\operatorname{Orb}(x)$ for $x \in A$.}
\]
\end{lemma}
The set $A$ is called the \emph{attractor} of the system. In particular, all essential aspects of the dynamics occur in the attractor $ A$. The Structure Lemma holds since $X$ has only finitely many points. In the general case already to determine the existence of an attractor can be a very difficult task. 

\begin{example}[Cell signalling dynamics]
The size of the attractor is $|A|=69\,846$ for system arising in \cref{sec:CellProcess}, which is less than $2\%$ of the size of the whole state space $|X|=431\,8000$. 
The attractor $A$ consists of $30$ cycles, see \cref{tab:period_orbits} and a
two dimensional projection of $A$ was obtained by projecting to the coordinates $113$ and $97$, see \cref{Fig10}.
\begin{figure}[h]
\centering
\includegraphics[width=0.6\textwidth]{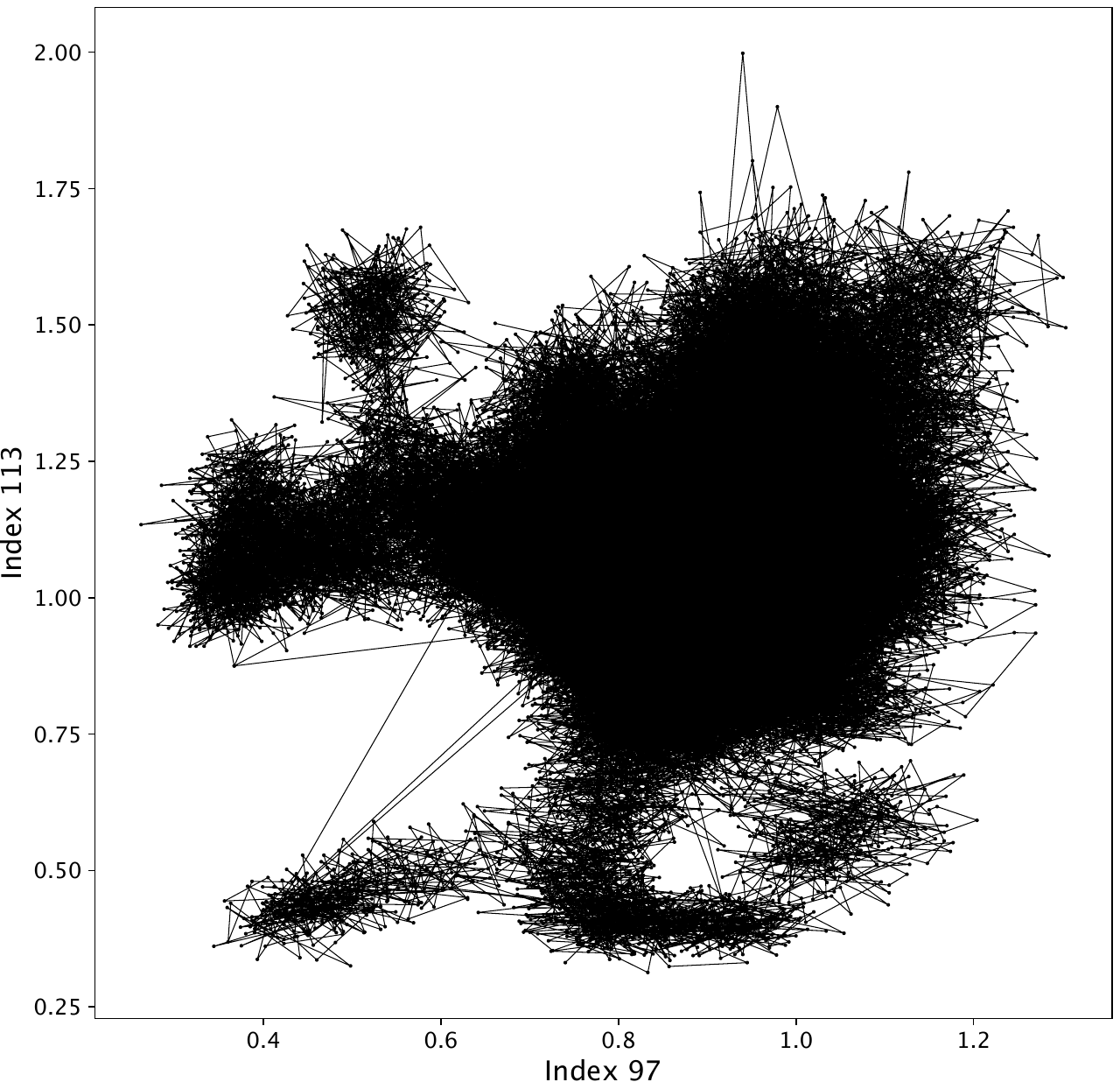}
\caption{A Two-Dimensional Projection of the Attractor}
\label{Fig10}
\end{figure}
The other $98\%$ of $X$ that is not the attractor is in the basin, so it is not relevant for the recurrence of the dynamics. 
This is the first result concerning the physical system in \cref{sec:BiologicalUseCase}.
\emph{In particular, this reduces the data set to a much smaller set, i.e. $2\%$ of the total.} 

To find a small attractor is not surprising, since this is in fact what one expects to see in systems governed by dissipative processes. Indeed, our model confirms the dissipative aspect of the cell dynamics and encourages us to develop it further into a predictive tool. Moreover, one should compare \cref{Fig:4b} with \cref{Fig10} to see that the attractor has clearly defined feature, unlike the whole set $X$. 
\begin{table}[]
\centering
\caption{Distribution of Cycle Periods}
\label{tab:period_orbits}
\begin{tabular}{cl}
\toprule
\textbf{\# Orbits} & \textbf{Period} \\
\midrule
2 & 2162, 2205, 2234 \\
\addlinespace
1 & 2166, 2172, 2177, 2178, 2185, 2187, 2188, 2193, 2194, \\
 & 2210, 2216, 2232, 2241, 2244, 2253, 2264, 2277, 2280, \\
 & 2303, 2351, 2416, 2512, 2829, 4376 \\
\bottomrule
\end{tabular}
\end{table}
\end{example}


\subsection{Renormalization}\label{reno}
The attractor $A$ has been discussed so far only in its topological properties in terms of cycles and their periods. We introduce now a tool which allows to study its geometrical properties as well, i.e., how close the cycles are to each other and how they accumulate at each other or at themselves. This tool is known as renormalization and describes the attractor from scale to scale. 

The attractor $A$ inherits the metric from $\Real^d$. A rectangular box is a set of the form 
\[
\prod_{i=1}^d[a_i,b_i].
\]
Let $U_0$ be the smallest rectangular box which contains $ A$. Inductively we define a nested sequence of \emph{renormalization domains}, which are rectangular boxes $U_0\supset U_1\supset \dots\supset U_m$ as follows. Assume that $U_k$ is a rectangular box and let $B_k=U_k\cap A$. Assume that $U_k$ is the smallest rectangular box containing $B_k$. The set $B_{k+1}$ is defined as
\[
B_{k+1}= B_k\setminus (\partial U_k\cap B_k)
\]
where $\partial U_k$ denotes the boundary of $U_k$. Finally let $U_{k+1}$ be the smallest rectangular box containing $B_{k+1}$. For each $U_k$ we define the corresponding renormalization $R^kf \colon U_k\to U_k$ as \[ R^kf(x) := f^s(x)
   \quad\text{where $s>0$ is minimal such that $f^s(x)\in U_k$.}
\]
In other words the $k$-th renormalization of $f$ is the first return map to $U_k$. The set $U_k$ are getting smaller and smaller and define the $k$-th scale of $f$. The $k$-th renormalization describes the dynamics at the scale $k$, see \cref{Fig4}.
\begin{figure}[h]
\centering
\includegraphics[width=0.6\textwidth]{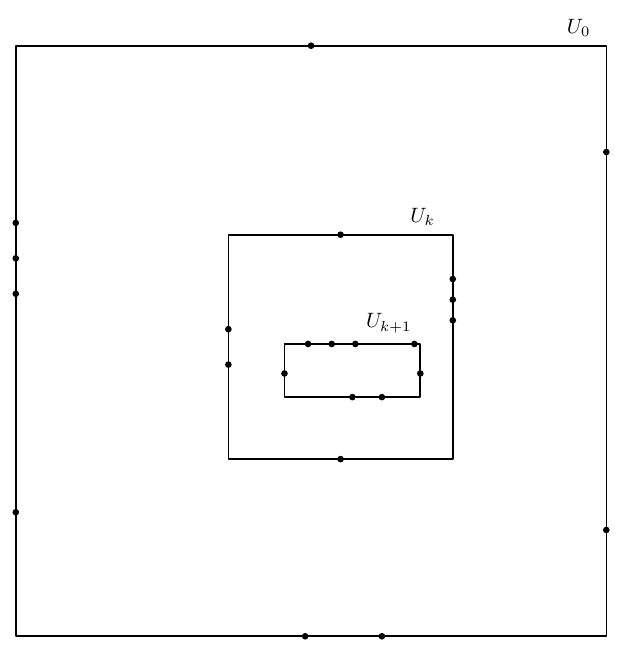}
\caption{Consecutive Renormalization Domains}
\label{Fig4}
\end{figure}

Now we describe the structure of the $k$-th renormalization. Observe that all points of $R^kf$ are periodic. We will describe the dynamics by the number of points with a given return time. In particular, we have the \emph{return time distribution} which is a mapping $\psi_k \colon \Natural\to\Natural$ defined as 
\[
\psi_k(r) := \#\bigl\{x\in U_k \bigl|\bigr. R^kf(x)=f^r(x)\bigr\}.
\]

\begin{example}[Cell signalling dynamics]\label{example:returntimedistributions}
$\psi_k(r)$ is either $0$ or $1$ for all $k$ and all $r$ for the system in \cref{sec:BiologicalUseCase}. This means that each point in $U_k$ returns with its own time. 
\begin{figure}
  \centering
  \includegraphics[width=0.9\linewidth]{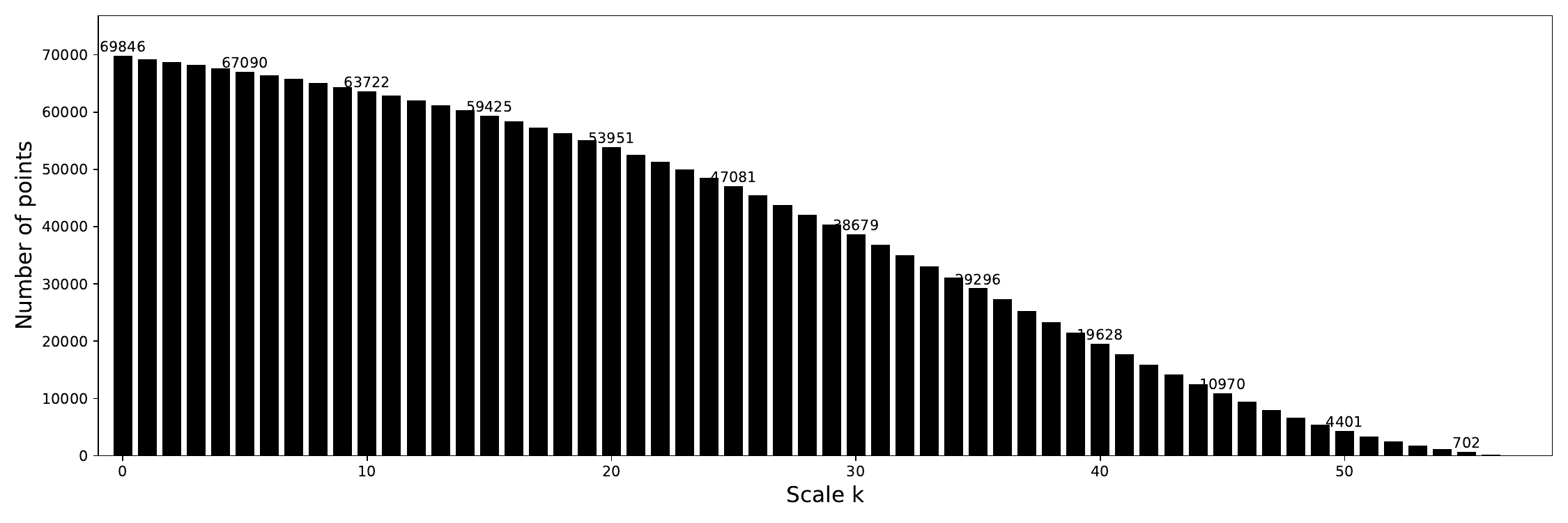}
  \caption{Number of points in each $U_k$ used in renormalization.}
  \label{fig:Bkpoints}
\end{figure}
\end{example}

\subsection{Dynamical partitions and renormalization sub-shifts}\label{DynPart}
Each renormalization domain $U_k$ determines a cover of open sets of the attractor. This cover is called the \emph{$k$-th dynamical partition} which corresponds to a geometrical scale of the dynamics. 
We start by describing each dynamical partition in a combinatorial way. 

Associated with each renormalization domain $U_k$ there is a directed graph $G_k$ that is defined as follows. A cycle of period $r$ can be represented as a directed graph consisting of $r$ vertices, and each vertex $v$ has a unique incoming and a unique outgoing edge. Such a directed graph is called a \emph{loop} of \emph{length} $r$. 

Given the return time distribution $\psi_k$ of the $k$-th renormalization, consider a collection of loops such that, for each $r$ with $\psi_n(r)\neq 0$ there are exactly $\psi_n(r)$ loops of length $r$. 
Choose one vertex in each loop and denote it by $u_k$. The directed graph $G_k$ is the union of those loops in which all vertices $u_k$ are identified. The vertex $u_k$ is called the \emph{root} of $G_k$. 
A graph with this structure is called a \emph{renormalization graph}. Then $G_k$ defines the corresponding \emph{sub-shift of finite type} as follows. 
A function $\gamma\colon \Integer\to G_k$ is called a path if $\gamma (t+1)$ is a successor of $\gamma(t)$. Let 
\[
\Sigma_k := \bigl\{\gamma \colon \Integer\to G_k \bigl|\bigr. \text{$\gamma$ is a path in $G_k$} \bigr\}
\]
and let $\sigma_k \colon \Sigma_k\to\Sigma_k$ be such $\sigma_k(\gamma)(t)=\gamma(t+1)$.
The pair $(\Sigma_k,\sigma_k)$ is the sub-shift of finite type associated to $G_k$. The topology on $\Sigma_k$ is defined as follows: A sequence $\gamma_n\in \Sigma_k$ converges to $\gamma\in \Sigma_k$ if for each finite set $Z\subset \Integer$ there exists $n_0\in \Natural$ such that the restrictions to the set $Z$ become eventually constant, namely  ${\gamma_{n}}_{|Z}=\gamma_{|Z}$ for all $n\ge n_0$.
\begin{figure}[h]
\centering
\includegraphics[width=0.6\textwidth]{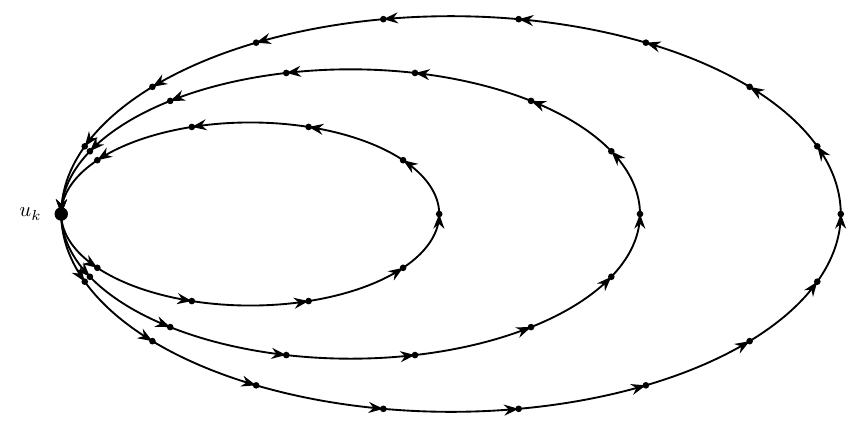}
\caption{The renormalization graph $G_k$}
\label{Fig5}
\end{figure}

Observe that every vertex in $G_k\setminus \left\{u_k\right\}$ corresponds to a single point in $ A$. The vertex $u_k$ corresponds to the set $U_k\subset A$, which in turn defines a projection $p_k \colon A\to G_k$ by
\begin{equation*}
p_k(x) := 
\begin{cases}
    x & \text{ for $x\notin U_k$,}
    \\
    u_k & \text{ for $x\in U_k$,}
\end{cases}
\end{equation*}
and the $k$-th dynamical partition 
\[
\mathcal D_k =
  \{U_k\} \cup
  \bigl\{\{v\}\bigl|\bigr. 
  v\in G_k \setminus\{u_k\}
  \bigr\}.
\]
Observe that $\mathcal D_{k+1}$ is a refinement of $\mathcal D_k$ and that the return time distribution $\psi_k$ completely determines the directed graphs $G_k$ and the corresponding sub-shifts of finite type. The directed graphs offer another point of view for the consecutive renormalizations. The dynamical partitions equip the attractor $A$ with an extra geometrical structure that will be explored in the next section.

\subsection{Revisiting the attractor: chain recurrence}\label{RevAttr}
Observe that our discrete dynamical system $f \colon A\to A$ cannot be treated as a regular map sending points to points. This is due to measurement errors and to the choice of the last successor, i.e. $\varphi_k(T)\mapsto f(\varphi_k(T))$. In particular, given a point $x$ in $A$, because of the errors we actually have a lot of choices, close to $f(x)$ but not necessarily equal. This tells us that the discrete map $f \colon A\to A$ might obscure the presence of other physically relevant orbits.

In this section, we address the above issue and incorporate those orbits that are left behind by the formal discrete map $f$ but that might be relevant to the physical process. As an example, consider the situation in \cref{Fig:chainrec} where a point $x$ in a cycle $\gamma$ is sent to $f(x)$ with uncertainty due to errors described by a neighbourhood $E(f(x))$ of $f(x)$. It could be that another cycle $\tilde\gamma$ crosses the error neighbourhood $E(f(x))$ in a point $y$. This gives rise to a choice for the physically relevant successor of $x$, namely $f(x)$ or $y$. This ambivalence occurs at each step and gives rise to many more physically relevant orbits than the only cycles in $A$. The precise description of this ambivalence and its consequences are presented in the following.

\begin{figure}[h]\label{Fig:chainrec}
\centering
\includegraphics[width=0.6\textwidth]{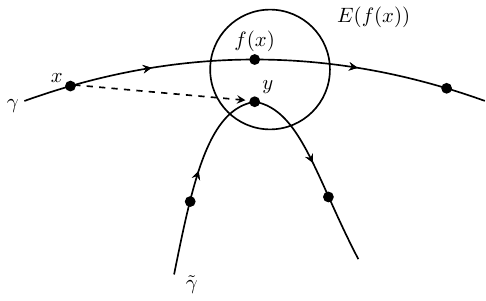}
\caption{Chain Recurrence}
\label{chainrecurrence}
\end{figure}

Each point $x\in A$ has a small neighbourhood $A\supset E(x)\ni x$ that represents the uncertainty of the point $x$.
Moreover, each point in $E(x)$ is mapped under $f$ to a point in the neighbourhood $E(f(x))$, but we have no precise information on this assignment. Hence, a more realistic presentation of the attractor is given by some map
\[
F \colon E(x) \mapsto E(f(x)).
\]
Since these error boxes might overlap, this might not give a well defined global map on 
\[ 
E=\bigcup_{x\in A} E(x).
\]
We only know that $F(E(x))\subset E(f(x))$, and we are going to describe the dynamics incorporating this so-called \emph{chain recurrence}. Consider the largest set $\mathcal{O}$ of possible orbits. An orbit $\gamma\in\mathcal{O}$ in this sense is a function $\gamma \colon \Natural \to A$ such that $ \gamma(t+1)\in E(f(\gamma(t)))$ holds.
The original map $f \colon A\to A$ induces the shift $\phi \colon \mathcal{O}\to \mathcal{O}$ by
$\phi(\gamma)(t)=\gamma(t+1)$.
Observe that the set $\mathcal{O}$ consists of all possible orbits incorporating all possible choices due to errors, as described above. The set of \emph{physical relevant orbits} is a subset
$\mathcal{O}_{\text{phys}}\subset \mathcal{O}$ and the next task is to describe $\mathcal{O}_{\text{phys}}$. Unfortunately, in practice, we do not know the error sets $E(x)$. However, the renormalization scheme tells us how to prescribe the ambivalence, namely, in each renormalization scale $k$, the ambivalence is restricted to $U_k$, i.e.
\[ x\mapsto f(x) \quad\text{for all $x\notin U_k$.}
\]
Define $E_k:=\bigl\{y=f(x)\bigl|\bigr. x\in U_k\bigr\}$ and let $\mathcal{O}_k\subset \mathcal{O}$ be the set of all functions $\gamma\colon \Natural\to A$ such that
\[
\gamma(t+1)=\begin{cases}
f\bigl(\gamma(t)\bigr) & \text{ if $\gamma(t)\notin U_k$,} 
\\
v & \text{ if $\gamma(t)\in U_k$ and $v\in E_k$.}
\end{cases}
\]
In particular, the orbit $\gamma$ has a successor freedom only when $\gamma(t)\in U_k$. 
Then projection $p_k \colon A\to G_k$ induces the bijection $\pi_k \colon \mathcal{O}_k\to\Sigma_k$
such that 
\[
\pi_k(\gamma)(t)=p_k\bigl(\gamma(t)\bigr)
\quad\text{whenever $\gamma\in O_k$.}
\]
Moreover, $\sigma_k\circ\pi_k=\pi_k\circ \phi$, so the map $\pi_k$ is called \emph{conjugation}. 

We are now ready to state the first Hyperbolicity Consequence~and Validation Criterion, see \cref{Sec:HypCons}.
 
\begin{description}
\item[Hyperbolicity Consequence~I.] For every hyperbolic system there exists a sequence of nested renormalization boxes $U_k$ with $0\leq k< m$ such that the system is equivalent to the corresponding sub-shift $\Sigma_k$ for large enough $k$. Moreover, given a finite selection of renormalization boxes $U_k$ with $0\leq k< m$, there is a non empty collection of cycles that intersect all renormalization boxes $U_k$ with $k<m$. 

\item[Validation Criterion~I.] 
  Each cycle $\gamma$ of $A$ intersects all renormalization boxes $U_k$ with $k<m$.     
\end{description}
If Validation Criterion~I holds, then the cycles in $A$ play the role of the cycles that were ensured to exist in a hyperbolic context; see Hyperbolicity Consequence~I.

\begin{remark}\label{Rem:first}
    Given a hyperbolic system with a chosen nested sequence of renormalization boxes $U_k$  with $k=0,\dots, k_{\text{max}}$, it is possible to select periodic orbits $\gamma_k$  such that 
\[ 
\gamma_k\cap U_k\neq\emptyset 
\quad\text{and}\quad
\gamma_k\cap U_{k+1}=\emptyset.
\]
Indeed, a single $\gamma_k$ does  not describe all the details of the system. A periodic orbit $\gamma$ is called \emph{deep} if $\gamma\cap U_{k_{\text{max}}}\neq\emptyset$. A single deep periodic orbit may describe the global dynamics in more detail. In particular, for all $k$, $\gamma_k$ is not deep. From this perspective, one may interpret the Validation Criterion~I as confirming that the observed cycles share a common qualitative property, i.e. they are all deep. In other words, the physical process only uses deep cycles. 
\end{remark}

Recall that the renormalization scheme says that, on each scale $k$, the ambivalence only occurs on the small set $U_k$.  The Validation Criterion~I implies that the attractor $A\subset \mathcal{O}_k$ and we may reasonably assume that $\mathcal{O}_\text{phys}\subset \mathcal{O}_k$ for some $k$, i.e., the physical orbits only sense ambivalence while being in $U_k$. 

\begin{example}[Cell signalling dynamics] 
The system in \cref{sec:BiologicalUseCase} satisfies the Validation Criterion~I. From \cref{fig:sgraph} one can observe that there are no outliers. 
\end{example}
\subsection{The physical measure}\label{sec:PhysMeasure} 
A physical measure for a dynamical system is a measure $\mu^*$ that captures the statistical distribution of the orbits. In particular, it assigns to each open set the frequency of a typical orbit visiting this set, namely 
\[ 
\mu^*(U) := \lim_{n\to\infty}\frac{1}{n}\#\bigl\{i<n \bigl|\bigr. f^i(x)\in U \bigr\}
\quad\text{for any open set $U$.}
\]
The essence of a physical measure is that it does not depend on the starting point $x$, which is a typical point in state space. 
Hence, from a statistical point of view, all orbits behave the same, which is referred to as the ergodicity of the system. Many physical systems are often ergodic and have a physical measure. 

Moreover, physical measures are used to make predictions of observables. Let $\varphi$ be an observable of a chaotic dynamical system having a physical measure $\mu^*$. Then the time series $i\mapsto\varphi(f^i(x))$ will be a chaotic sequence of numbers and we will not be able to predict the value at a specified moment in time $i$. However, we can make a prediction on its time average and consider 
\[
\Expect[\varphi]=\lim_{n\to\infty}\frac{1}{n}\sum_i\varphi(f^i(x)).
\]
In addition, we get from Birkhoff ergodic theorem that 
\[
\Expect[\varphi]=\int\varphi d\mu^*.
\]
Our goal is to construct the physical measure $\mu^*$ that will allow us to make meaningful predictions. 

\subsection{Cycle measures}
The task at hand is to characterize the physically relevant orbits, $\mathcal{O}_\text{phys} $. Indeed, the description of $\mathcal{O}_\text{phys} $ constitutes the model itself. This characterization is carried out by statistical means. We now introduce the necessary ingredients.

So far we built a sequence of sub-shifts of finite type $\left(\Sigma_k, \sigma_k\right)$ associated to the graphs $G_k$. As explained in the introduction, our model will be a Markov process, so we need to equip these purely topological objects 
with an invariant measure. Observe that every cycle carries a unique invariant measure which we will use to construct the physical measure of the model. 
\begin{figure}
  \centering
\includegraphics[width=0.9\linewidth]{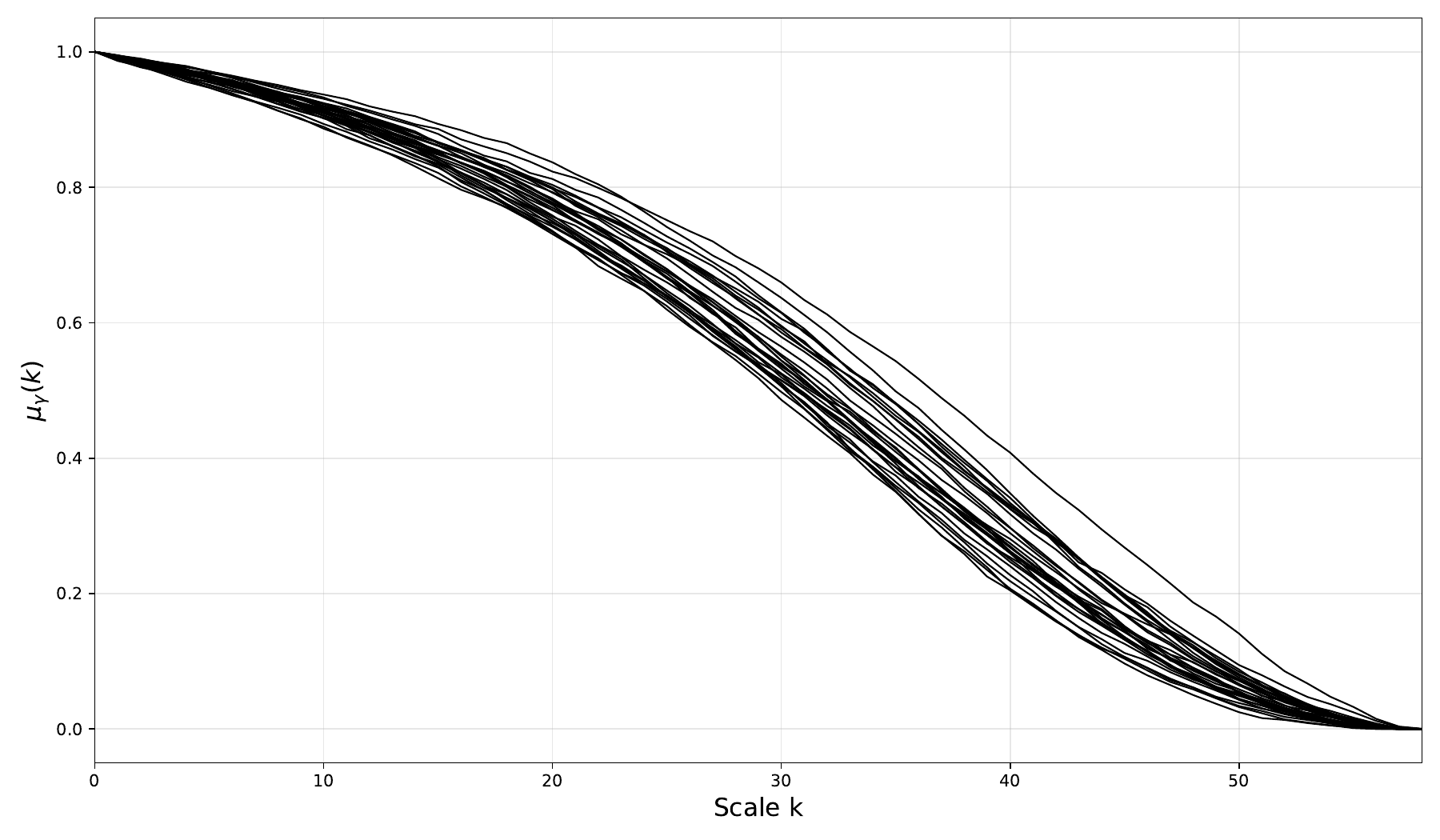}
  \caption{Validation of Criteria I and II}
  \label{fig:sgraph}
\end{figure}

The collection of all cycles in $A$ is denoted by $\Gamma$. For each cycle $\gamma\in\Gamma$, say of period $T$, consider the unique invariant probability measure on $\gamma$
\[
\mu_{\gamma} := \frac{1}{T}\sum_{a\in\gamma}\delta_a
\quad\text{where $\delta_a$ is the point mass in $a$.}
\]
This is called a \emph{cycle measure} of $\gamma$ and it will be used to construct the physical measure of our model:
\[ 
\mu\in\operatorname{Span}\bigl(\{\mu_{\gamma}\}_{\gamma\in\Gamma }\bigr).
\]

\begin{description}
\item[Hyperbolicity Consequence~II.]
A hyperbolic system with a physical measure $\mu^*$ has the property that for every $\delta>0$ and for any finite selection of renormalization boxes $U_k$ with $0\leq k\leq m-1$, there exists a non empty finite collection of cycles $\Gamma^*=\{\gamma\}$ such that 
\[
\bigl|\mu_{\gamma}(U_k)-\mu^*(U_k)\bigr| < \delta
\quad\text{holds for every $\gamma\in\Gamma^*$ and $0\leq k\leq m-1$.}
\]

\item[Validation Criterion~II.] 
Let $\delta_{\scriptscriptstyle \mathrm{II}}$ be the smallest number such that the following holds. For every pair of cycles $\gamma,\tilde\gamma\in\Gamma$ we have that
\begin{equation}\label{bundelingS}
\bigl|\mu_{\gamma}(U_k)-\mu_{\tilde\gamma}(U_k)\bigr|< 2\delta_{\scriptscriptstyle \mathrm{II}}
\quad\text{for every $0\leq k\leq m-1$.}
\end{equation}
\end{description}
If Validation Criterion~II holds, then the cycles in $A$ play the role of the cycles ensured to exist in a hyperbolic context; see Hyperbolicity Consequence~II.

\begin{remark}\label{Rem:second}
Similarly to \cref{Rem:first}, given a hyperbolic system with a chosen nested sequence of renormalization boxes $U_k$  with $k=0,\dots, k_{\text{max}}$, it is possible to select periodic orbits $\gamma_k$ that do not satisfy Validation Criterion~II. From this perspective, one may interpret Validation Criterion~II as confirming that the observed cycles arising from the physical process satisfy an even stronger quantitative criterion than merely being deep cycles.

Observe that Validation Criterion~I guarantees that the observed cycles satisfy a qualitative criterion, namely that they cross all renormalization boxes. Validation Criterion~II shows that they moreover share an extra common quantitative property: they pass through the renormalization boxes with approximately equal frequency. In particular, the observed cycles are not arbitrary cycles. 
They appear to exhibit a statistical organization.
\end{remark}
\begin{example}[Cell signalling dynamics]
The system in \cref{sec:BiologicalUseCase} satisfies Validation Criterion~II with $\delta_{\scriptscriptstyle \mathrm{II}}=0.1040$. 
The functions $k\mapsto\mu_{\gamma}(U_k)$, $\gamma\subset\Gamma$, are plotted in \cref{fig:sgraph}. The bundling one observes in \cref{fig:sgraph} is the first indication of the statistical organization of the cycles. 
\end{example}

\subsection{A Markov chain as model}
A hyperbolic dynamical system with physical measure $\mu^*$ can be represented by many different Markov chains. Indeed, each Markov partition induces a symbolic representation of the dynamics and hence a corresponding Markov chain. The dynamical partitions created by the renormalization process can be used as covers to construct a Markov chain model for the given hyperbolic system.

Motivated by this observation, we use the dynamical partitions $\mathcal{D}_k$ generated by the renormalization process (see \cref{DynPart}) to construct a Markov chain model for the physical process under investigation. At each scale $k$, the partition $\mathcal{D}_k$ defines a topological subshift of finite type $(\Sigma_k,\sigma_k)$. The key step is to enrich this symbolic system with probabilistic information. To this end, we use the cycles in $\Gamma$ together with their associated cycle measures to define a probability measure $\mu_k$ on $\Sigma_k$, thereby obtaining a probabilistic Markov chain
$
(\Sigma_k,\sigma_k,\mu_k)$.
The construction of $\mu_k$ and the resulting Markov chain model are described in detail below. 

Consider the following measure on the attractor $A$, 
\[ 
\quad
\mu=\frac{1}{\#\Gamma}\sum_{\gamma\in\Gamma}\mu_{\gamma},
\]
where $\#\Gamma$ is the number of cycles in $\Gamma$. 
In order to turn each $\Sigma_{k}$ into a Markov chain, we need to define an invariant measure $\mu_k$. In particular, given a finite path $\underline v=\left(v_1,v_2\dots,v_r\right)$ in the graph $G_k$ we have to describe the probability 
$
\mu_k\left[v_1,v_2\dots,v_r\right],
$
that the path $\underline v$ occurs. For a path $\underline v=(v)$ of length $1$, the measure $\mu$ describes this probability as
\[
\mu_k\bigl([v]\bigr)=\mu(v).
\]
For a general path $\underline v=(v_1,v_2\dots,v_r)$ we get,
\[
\mu_k\bigl([v_1,v_2\dots,v_r]\bigr)=\mu(v_1)\prod_{i=1}^{n-1}p_{v_i,v_{i+1}},
\]
where $p_{v,v'}$ denotes the \emph{transition probability} from vertex $v$ to vertex $v'$ and they are defined as follows.

If there is no edge in $G_k$ from $v$ to $v'$, then $p_{v,v'}=0$. If there is an edge and $v\neq u_k$ then $p_{v,v'}=1.$ Next, let $V_k\subset G_k\setminus{u_k}$ be the vertices such that $u_kv$ is an edge of $G_k$. For $v'\in V_k$, the probabilities $p_{u_k,v'}$ are defined such that
\[
\mu(u_k)p_{u_k,v'}=\mu(v').
\]
In particular, if $v'\in\gamma'$, $\gamma'\in\Gamma$ with period $T_{\gamma'}$ then
\[
p_{u_k,v'}=\frac{1}{T_{\gamma'}}\frac{1}{\sum_{\gamma\in\Gamma}\mu_{\gamma}(u_k)}
\quad\text{and}\quad
p_{u_k u_k}=1-\sum_{v\in V_k} p_{u_k v}.
\]

The definition of the Markov chain $\left(\Sigma_k,\mu_k, \sigma_k\right)$ for $k=0,1,\dots,m-1$ is complete:
Define the \emph{transition matrix} $P_k=\left(p_{vv'}\right)$ and the probability row vector $m_k=\left(\mu(v)\right)_{v\in G_k}$. 
This definition implies the following lemma, see \cite[Chapter~5]{Petersen} for more details on Markov chains.
\begin{lemma}\label{transitionmatrix}
Each row in the matrix $P_k$ is a probability vector, i.e., $\sum_{v'}p_{vv'}=1$ holds for every $v$.
Moreover, $m_k$ is the left eigenvector of $P_k$, i.e.
$m_k P_k=m_k$. In particular, the powers of the transition matrix $P^s_{k}$ converge exponentially to the matrix in which each row equals $m_k$,
\[
\lim_{s\to\infty} xP^s_k=\left({\sum x_i}\right)m_k,
\quad\text{for each vector $x=(x_i)$.}
\]
\end{lemma}
The \emph{entropy} of the Markov chain $(\Sigma_k,\mu_k, \sigma_k)$ is defined as 
\begin{equation}\label{eq:EntropyExpr}
h_k:=-\sum_{v,v'}\mu_k(v)p_{v,v'}\log p_{v,v'}.
\end{equation}
Generally speaking, the entropy measures the degree of chaos in the process. The higher the entropy, the more chaotic it is, see \cite{Petersen} for more information.
\begin{figure}
   \centering
   \includegraphics[width=0.9\linewidth]{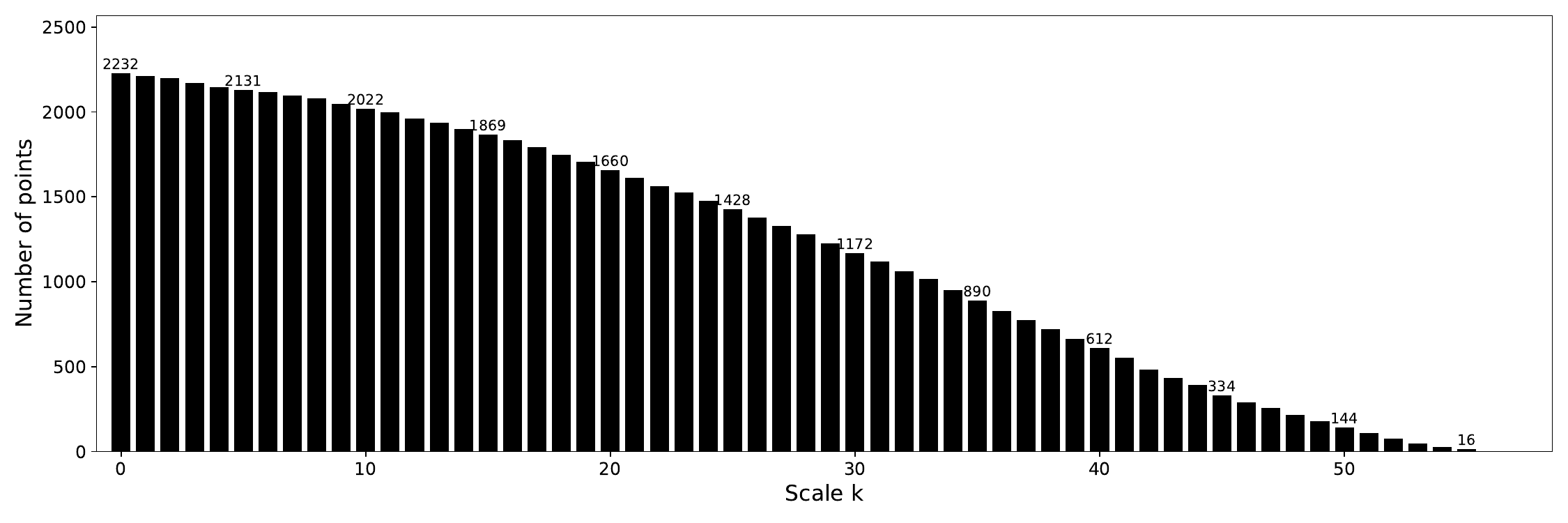}
   \caption{Values of $N_{\gamma}(k)$ over sets $U_k$ used in renormalization.}
   \label{fig:N_gammakoveruk}
\end{figure}
  
Next, let $\gamma\in\Gamma$ and recall that $\gamma \colon \Integer\to A$ is a periodic function with period $T_\gamma$. Then $N_\gamma(k)=E_\gamma(k)+F_\gamma(k)$ where 
\begin{align*}
 N_\gamma(k)&:= \#\bigl\{0\leq t\leq T_{\gamma}\bigl|\bigr.\gamma(t)\in U_k\bigr\},\\
 E_\gamma(k)&:= \#\bigl\{0\leq t\leq T_{\gamma}\bigl|\bigr.\gamma(t)\in U_k\text{ and }\gamma(t+1)\notin U_k\bigr\},\\
 F_\gamma(k)&:= \#\bigl\{0\leq t\leq T_{\gamma}\bigl|\bigr.\gamma(t),\gamma(t+1)\in U_k\bigr\}.
\end{align*}
The following lemma follows directly from the definition of the terms arising in the claim.
\begin{lemma}\label{lem:TransProb}
The transition probabilities and the physical measure for the Markov chain $(\Sigma_k,\mu_k, \sigma_k)$ can be expressed as
\[
   p_{u_k,u_k }=\frac{\mu_k(u_k)-\sum_{v\in V_k}\mu_k(v)}{\mu_k(u_k)}
 \quad\text{and}\quad
   \mu_k(u_k)-\sum_{v\in V_k}\mu_k(v)=\frac{1}{\#\Gamma}\sum_{\gamma\in\Gamma}\frac{F_{\gamma}(k)}{T_{\gamma}}.
\]
\end{lemma}
\begin{proposition} The entropy profile $k\mapsto h_k$ in \cref{eq:EntropyExpr} satisfies the following:
\[
h_k = \frac{1}{\#\Gamma} \sum_{\gamma\in \Gamma} 
\frac{N_\gamma(k)}{T_{\gamma}}\log \biggl(T_{\gamma} \cdot \sum_{\gamma'\in \Gamma} \frac{N_{\gamma'}(k)}{T_{\gamma'}}\biggr)
-\frac{1}{\#\Gamma} \sum_{\gamma\in \Gamma} 
\frac{F_\gamma(k)}{T_{\gamma}}\log \biggl(T_{\gamma} \cdot \sum_{\gamma'\in \Gamma} \frac{F_{\gamma'}(k)}{T_{\gamma'}}\biggr).
\]
\end{proposition}
\begin{proof} For each $v\in V_k$, $\gamma_v\in \Gamma$ is the cycle which contains $v$. 
Then,
{\begingroup
\allowdisplaybreaks
\begin{align*}
 h_k &:= 
 -\sum_{v,v'\in G_k}\mu_k(v)p_{v,v'}\log p_{v,v'} \\ 
 &= -\sum_{v\in V_k} \mu_k(u_k)p_{u_k v} \log p_{u_k v}-\mu_k(u_k)p_{u_k u_k} \log p_{u_k u_k}
 \\
 &= -\sum_{v\in V_k} \mu_k(v)\log \frac{\mu_k(v)}{\mu_k(u_k)}-\biggl( \mu_k(u_k)-\sum_{v\in V_k}\mu_k(v)\biggr)\log\biggl(\frac{\mu_k(u_k)-\sum_{v\in V_k}\mu_k(v)}{\mu_k(u_k)}\biggr)
 \\
 &= \mu_k(u_k)\log \mu_k(u_k)-\sum_{v\in V_k} \mu_k(v)\log \mu_k(v)
 \\
 &\qquad\qquad -\biggl(\mu_k(u_k)-\sum_{v\in V_k}\mu_k(v)\biggr)
 \log\biggl(\mu_k(u_k)-\sum_{v\in V_k}\mu_k(v)\biggr)
 \\
 &= \frac{1}{\#\Gamma}\sum_{\gamma\in\Gamma}\frac{N_{\gamma}(k)}{T_{\gamma}}\log\biggl(\frac{1}{\#\Gamma}\sum_{\gamma'\in\Gamma}\frac{N_{\gamma'}(k)}{T_{\gamma'}}\biggr)-\sum_{v\in V_k}\frac{1}{\#\Gamma T_{\gamma_v}}\log\biggl(\frac{1}{\#\Gamma T_{\gamma_v}}\biggr)
 \\
 &\qquad\qquad -\frac{1}{\#\Gamma}\sum_{\gamma\in\Gamma}\frac{F_{\gamma}(k)}{T_{\gamma}}\log\biggl(\frac{1}{\#\Gamma}\sum_{\gamma'\in\Gamma}\frac{F_{\gamma'}(k)}{T_{\gamma'}}\biggr)
 \\
 &= \frac{1}{\#\Gamma}\sum_{\gamma\in\Gamma}\frac{N_{\gamma}(k)}{T_{\gamma}}\log\biggl(\frac{1}{\#\Gamma}\sum_{\gamma'\in\Gamma}\frac{N_{\gamma'}(k)}{T_{\gamma'}}\biggr)-\frac{1}{\#\Gamma}\sum_{\gamma\in\Gamma}\frac{E_{\gamma}(k)}{ T_{\gamma}}\log\biggl(\frac{1}{\#\Gamma T_{\gamma}}\biggr)
 \\
 &\qquad\qquad -\frac{1}{\#\Gamma}\sum_{\gamma\in\Gamma}\frac{F_{\gamma}(k)}{T_{\gamma}}\log\biggl(\frac{1}{\#\Gamma}\sum_{\gamma'\in\Gamma}\frac{F_{\gamma'}(k)}{T_{\gamma'}}\biggr)
 \\
 &= \frac{1}{\#\Gamma}\sum_{\gamma\in\Gamma}\frac{N_{\gamma}(k)}{T_{\gamma}}\log\Bigl(\sum_{\gamma'\in\Gamma}\frac{N_{\gamma'}(k)}{T_{\gamma'}}\Bigr)-\frac{1}{\#\Gamma}\sum_{\gamma\in\Gamma}\frac{E_{\gamma}(k)}{ T_{\gamma}}\log\biggl(\frac{1}{ T_{\gamma}}\biggr)
 \\
 &\qquad\qquad -\frac{1}{\#\Gamma}\sum_{\gamma\in\Gamma}\frac{F_{\gamma}(k)}{T_{\gamma}}\log\biggl(\sum_{\gamma'\in\Gamma}\frac{F_{\gamma'}(k)}{T_{\gamma'}}\biggr)
 \\
 &= \frac{1}{\#\Gamma}\sum_{\gamma\in\Gamma}\frac{N_{\gamma}(k)}{T_{\gamma}}\log\biggl(T_{\gamma} \cdot\sum_{\gamma'\in\Gamma}\frac{N_{\gamma'}(k)}{T_{\gamma'}}\biggr)
 -\frac{1}{\#\Gamma}\sum_{\gamma\in\Gamma}\frac{F_{\gamma}(k)}{T_{\gamma}}\log\biggl(T_{\gamma}\cdot\sum_{\gamma'\in\Gamma}\frac{F_{\gamma'}(k)}{T_{\gamma'}}\biggr).
\end{align*}
\endgroup}
In the above, we used \cref{lem:TransProb} and the identity $N_{\gamma}(k)=F_{\gamma}(k)+E_{\gamma}(k)$.
\end{proof}

\begin{description}
\item[Hyperbolicity Consequence~III.] 
Given a hyperbolic system with physical measure $\mu^*$ which has entropy $h^*$. Then for every $\delta>0$ there exists a non empty finite collection of cycles $\Gamma^*=\{\gamma\}$ such that for every $\gamma\in\Gamma^*$, the corresponding Markov chains $\bigl(\Sigma_k(\gamma),\mu_k(\gamma), \sigma_k(\gamma)\bigr)$, $k=0,\dots,m$, have entropy function $h_{\gamma}(k):=h_k(\gamma)$ satisfying the following:
\begin{itemize}
\item[-] $h_{\gamma}(0)=0$ and $h_{\gamma}(m)=0$,
\item[-] there exists an interval $S_{\gamma} := \bigl[k(0),k(1)\bigr]$ such that $h_{\gamma}$ is increasing for all $k\leq k(0)$ and it is decreasing for all $k\geq k(1)$,
\item[-] $\bigl|h_{\gamma}(k)-h^*\bigr|<\delta$, for all $k\in S_{\gamma}$,
\item[-] $\bigcap_{\gamma\in\Gamma^*}S_{\gamma}\neq\emptyset$.
\end{itemize} 

\item[Validation Criterion~III.] Let $\delta_{\scriptscriptstyle \mathrm{III}}$ be the smallest number such that the following holds.
The following holds for every $\gamma\in\Gamma $, the collection of cycles in $A$:
\begin{enumerate}
\item $h_{\gamma}(0)=0$ and $h_{\gamma}(m)=0$, and 
\item there exists an interval $S_{\gamma}=\left[k(0),k(1)\right]$, called the plateau of $\gamma$, that is defined as
\[ S_{\gamma} :=
    \bigl\{k\bigl|\bigr. h_{\gamma}(k)>\max h_{\gamma}-\delta_{\scriptscriptstyle \mathrm{III}}\bigr\},
\]
such that $h_{\gamma}$ is increasing for all $k\leq k(0)$ and it is decreasing for all $k\geq k(1)$.
\end{enumerate} 
Moreover, $S=\bigcap_{\gamma\in\Gamma}S_{\gamma}\neq\emptyset,$ and  
\begin{equation}\label{bundelingofentropy}
     \bigl|h_{\gamma}(k)-h_{\tilde\gamma}(k)\bigr| < \delta_{\scriptscriptstyle \mathrm{III}}
     \quad\text{for every $\gamma,\tilde\gamma\in\Gamma$ and $k\in S$.}
\end{equation}
\end{description}

\begin{remark}\label{Rem:fourth}
   Similarly to \cref{Rem:first,Rem:second} given a hyperbolic system with physical measure $\mu$, it is possible to select an arbitrary large number of periodic orbits $\gamma_k$ that do not satisfy Validation Criterion~III. From this perspective, one may interpret Validation Criterion~III as confirming that the observed cycles arising from the physical process satisfy an even stronger statistical criterion than the ones satisfying only Validation Criterion~I and II. In particular, the observed cycles in $\Gamma$ are not arbitrary cycles and they share an extra common property. Indeed, they appear to have even more statistical structure, namely, at the scales $k\in S$, they have essentially the same entropy. 
\end{remark}

If the cycles in $A$ satisfy Validation Criteria I, II, and III, with sufficiently small\footnote{The requirements of the application determine what should be regarded as sufficiently small.} $\delta_{\scriptscriptstyle \mathrm{II}}$ and $\delta_{\scriptscriptstyle \mathrm{III}}$, then they share the statistical properties expected of physically relevant cycles. Consequently, it is reasonable to take the corresponding measure $\mu$ as the measure describing the model. 

\par\medskip
\fbox{\parbox{14cm}{%
Choose any $k^*\in S$, then the Markov chain $\left(\Sigma_{k^*},\mu_{k^*}, \sigma_{k^*}\right)$ based on $\Gamma$ is a model for the physical process. The model has entropy $h_{k^*}$. The value ${k^*}$ is called the \emph{scale} of the model.}}
\par\medskip
In practice, one would choose $k^*$ to be in the middle of $S$. 
\begin{example}[Cell signalling dynamics]
All cycles in $A$ for the system in \cref{sec:BiologicalUseCase} satisfy conditions~1 and 2 of Validation Criterion~III. The corresponding entropy bundling condition of \cref{bundelingofentropy} is illustrated in \cref{fig:entropies_eachcycle}. The aggregate entropy profile when considering all cycles in $\Gamma$ is in \cref{fig:entropiesallgoodcycles_a}.
The plateau considering all cycles in $\Gamma$, is $S=[25,36]$ for $\delta_{\scriptscriptstyle \mathrm{III}}=0.4234$. 



\begin{table}[]
\centering
\caption{Aggregate entropy $h(k)$ over the plateau $S=[25,36]$ for all 30 cycles, $X_{0-6,1000}$.}
\label{tab:entropy_aggregate_gr0_r1000}
\begin{tabular}{rr | rr | rr}
$k$ & $h(k)$ & $k$ & $h(k)$ & $k$ & $h(k)$ \\
\hline
25 & 2.1558 & 29 & 2.3361 & 33 & 2.3539 \\
26 & 2.2069 & 30 & 2.3577 & 34 & 2.3200 \\
27 & 2.2570 & 31 & 2.3609 & 35 & 2.2754 \\
28 & 2.2980 & 32 & 2.3635 & 36 & 2.2145 \\
\end{tabular}
\end{table}

\end{example}


\begin{figure}
  \centering
  \begin{subfigure}[t]{0.48\linewidth}
    \centering
    \includegraphics[width=\linewidth]{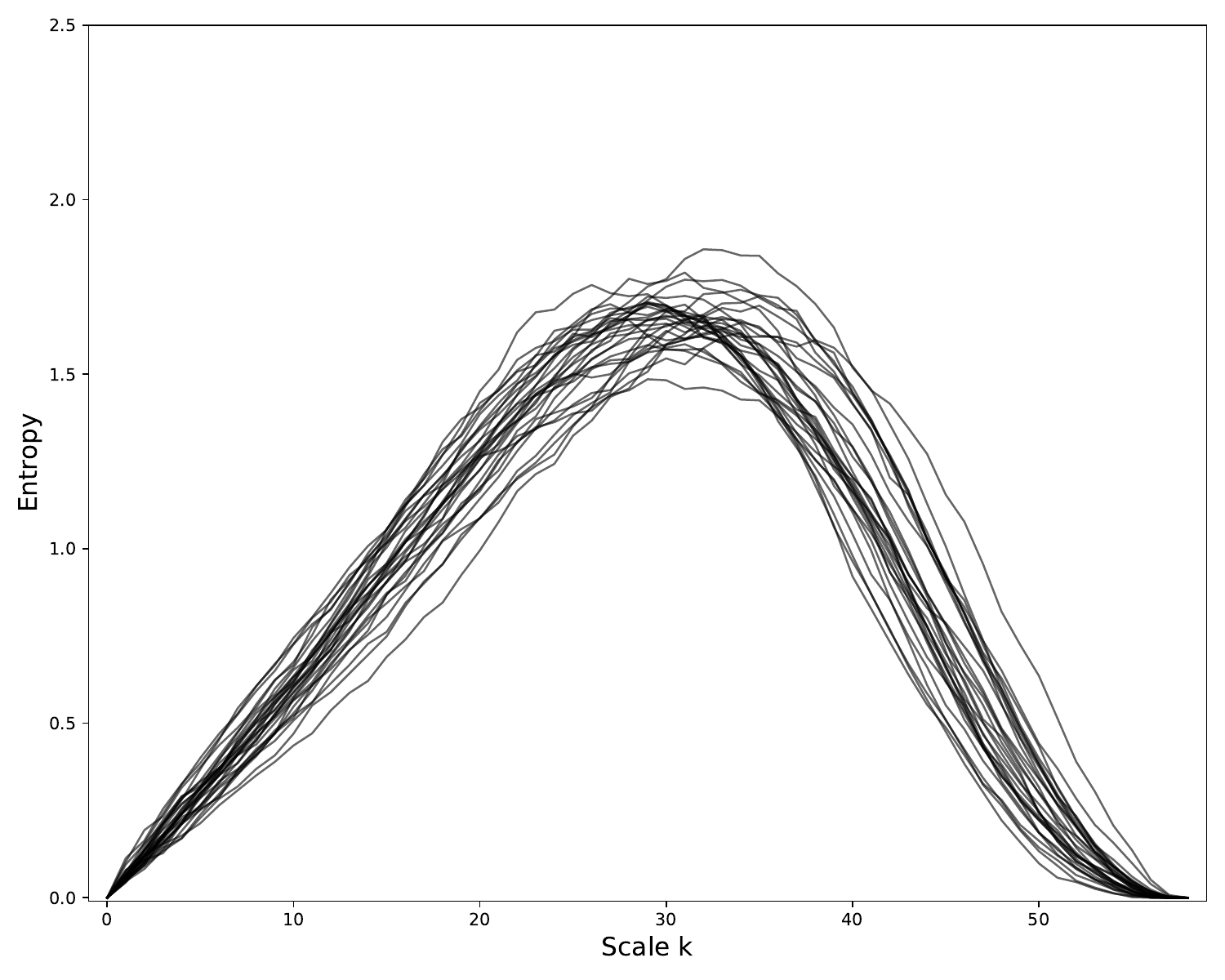}
    \caption{Individual entropy profile for each cycle.}
    \label{fig:entropies_eachcycle}
  \end{subfigure}
  \hfill
  \begin{subfigure}[t]{0.48\linewidth}
    \centering
    \includegraphics[width=\linewidth]{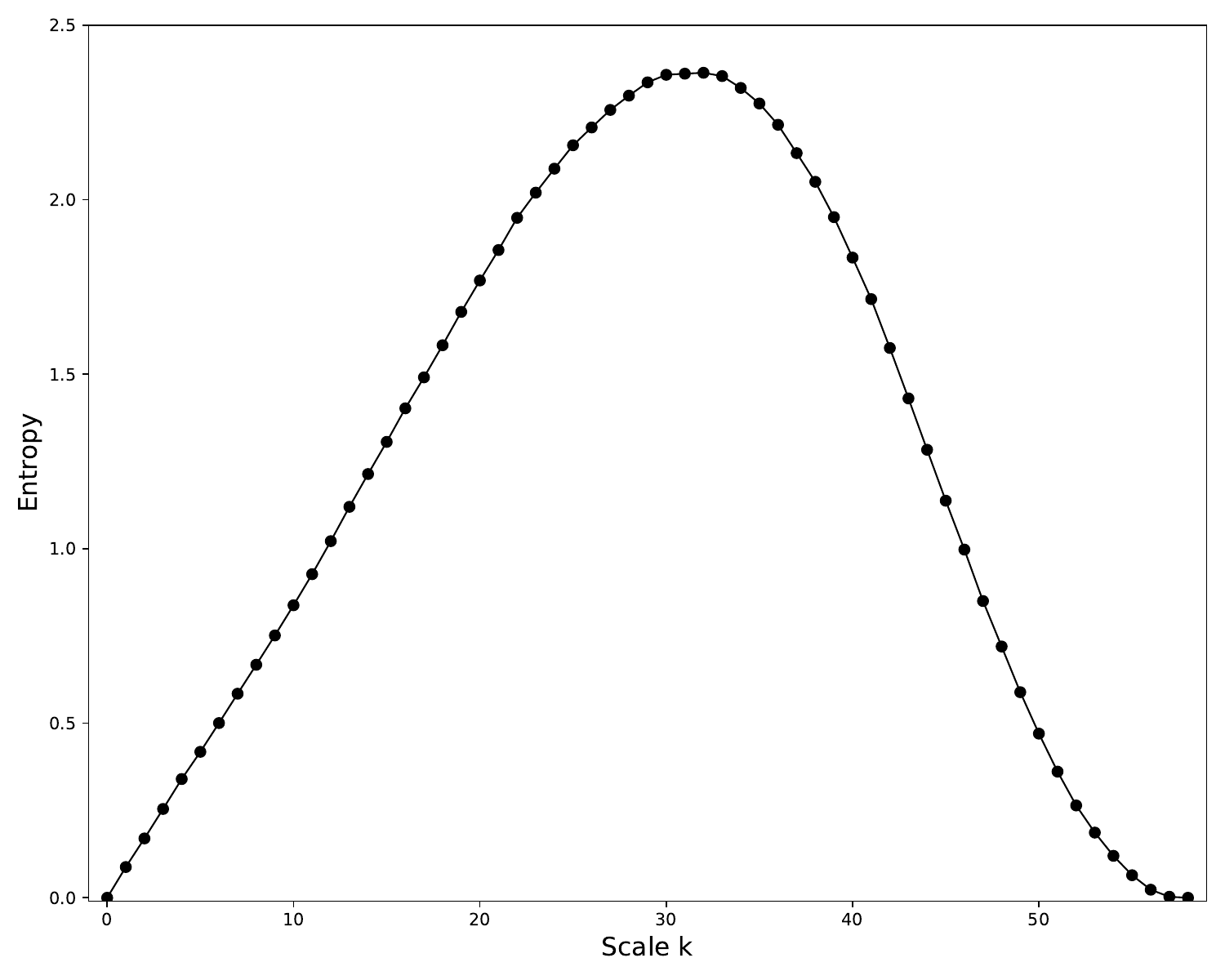}
    \caption{Aggregate entropy profile for all cycles in $\Gamma$.}
    \label{fig:entropiesallgoodcycles_a}
  \end{subfigure}
  \caption{Entropy profiles for $X_{0-6,1000}$.}
  \label{fig:entropies_gr0_r1000}
\end{figure}






\subsection{Further possible criteria for hyperbolicity}
In this subsection, we derive another consequence of hyperbolicity and introduce an additional validation criterion. The choice of criteria is not unique and allows for a certain degree of flexibility. In general, the more independent criteria are satisfied by the data, the stronger the evidence supporting the validity of the model. As example in this subsection we present and test another criterion. 
The new criterion is motivated by the following. 

Given a hyperbolic system with a physical measure $\mu$, a natural question is how well this measure can be approximated by periodic orbits. For any prescribed uncertainty threshold, there exists a periodic orbit whose invariant measure approximates $\mu$ within that threshold. In such generality, however, the period of the approximating orbit may be very large. On the other hand, obtaining very long periodic cycles from actual measurements is typically difficult in practice.
There are two possible ways to address this issue. The first is to reduce the dimension of the state space, since the approximation of a measure in lower dimension is easier to realize. Moreover, in many applications, only a few coordinates are relevant to study the phenomenon under investigation. From a measure-theoretic perspective, one may therefore project the physical measure onto the subspace determined by these coordinates and study the resulting marginal measure.
The second approach is to approximate the physical measure not by a single periodic orbit, but by a finite collection of periodic orbits, each of which may have a substantially shorter period. Hyperbolicity Consequence~IV and the corresponding criterion presented below are based on this idea.

Recall that the state space $X$ is contained in $\Real^d$. An observable is a smooth function $\varphi\colon \Real^d\to\Real$. However, in practical situations, the observables often depend only on small numbers, say $p$, of coordinates. 
From now on, we fix an observable on the reduced space $\Real^p$, $\varphi\colon \Real^p\to\Real$. 
Consider the projection, $\pi\colon \Real^d\to\Real^p$, which can be used to transfer by push forward any measure $\mu$ on $\Real^d$ to a (marginal) measure $\pi_*(\mu)$ on $\Real^p$ given as
\[ 
\pi_*(\mu)(U) := \mu\bigl(\pi^{-1}(U)\bigr)
\quad\text{where $U\subset \Real^p$ is any open set.}
\]
Observe now that
\[
\Expect[\varphi] =\int\varphi d\pi_*(\mu),
\]
so, it is enough to consider only the marginal physical measure in $\Real^p$.
A further simplification is made by observing that to estimate $\Expect[\varphi]$ within  given threshold one can restrict the measure to a finite collection of cubes. Namely, consider a finite grid that coves $\pi(A)$ by cubes $B_{\underline{i}}$ each of width $\delta$ (the uncertainty threshold):
\[
\pi(A)\subset\bigcup_{\underline{i}}B_{\underline{i}}. 
\]
The measure $\mu$ has a \emph{weight} which describes the measure restricted to these boxes, i.e.,
\[
w_{\mu} \colon \underline{i}\mapsto w_{\mu}(\underline{i})
= \pi_*(\mu)\bigl(B_{\underline{i}}\bigr).
\]
We will use the weight of $\mu$ to approximate integrals using Riemann sums:
\[
S_{\mu}\varphi := \sum_{\underline{i}}\varphi(c_{\underline{i}})w_{\mu}(\underline{i})
\quad\text{where $c_{\underline{i}}\in B_{\underline{i}}$ is the centre of the box.}
\]
We will use $S_{\mu}\varphi$ as a prediction of $\Expect[\varphi]$ with uncertainty bounded by $\delta |\varphi|_{C^1}$. This is formalized in the following lemma.
\begin{lemma}
$\bigl|\Expect[\varphi]-S_{\mu}\varphi \bigr|\leq \delta |\varphi|_{C^1}$.
\end{lemma}
Moreover, we also have two additional lemmas that will be useful:
\begin{lemma}
$| S_{\mu}\varphi-S_{\tilde\mu}\varphi | 
\leq |\varphi|_{C^0} |\mu-\tilde\mu|_{p}.
$
where the semi-norm on the measures is defined as
\[
|\mu-\tilde\mu|_{p}=\sum_{\underline{i}}\bigl|w_{\mu}(\underline{i})-w_{\tilde\mu}(\underline{i})\bigr|.
\]
\end{lemma}
\begin{lemma} 
Let $\mu_j$ be a measure for all $j\leq M$ where 
$\bigl|\mu_{i}-\mu_{j}\bigr|_p < \delta$ holds for all $i\neq j$.
Then
\[
\bigl|\mu_j-\mu\bigr|_{p}\leq\frac{M-1}{M}\delta
\quad\text{holds for all $j\leq M$ where}\quad
\mu := \frac{1}{M}\sum_{j=1}^{M}\mu_j.
\]
\end{lemma}
\begin{proof}
By definition we get that 
\begin{align*}
|\mu_j-\mu|_{p} &= \sum_{\underline{i}}\bigl|w_{\mu_j}(\underline{i})-w_{\mu}(\underline{i})\bigr| 
= \sum_{\underline{i}}\Bigl|w_{\mu_j}(\underline{i})-\frac{1}{M}\sum_{k=1}^{M}w_{\mu_k}(\underline{i})\Bigr|
\\
&\leq \frac{1}{M}\sum_{\underline{i}}\sum_{k=1}^{M}\bigl|w_{\mu_j}(\underline{i})-w_{\mu_k}(\underline{i})\bigr|
= \frac{1}{M}\sum_{k=1}^{M}|\mu_j-\mu_k|_{p}
< \frac{M-1}{M}\delta.
\end{align*}
This concludes the proof.
\end{proof}
Observe that every cycle $\gamma$ carries a cycle measure which has the corresponding weight, which we henceforth denote by $w_{\gamma}$.
\begin{description}
\item[Hyperbolicity Consequence~IV.] 
Given a hyperbolic system with physical measure $\mu^*$. Then for every $\delta>0$ and $p\in\Natural$ there exists a non empty finite collection of cycles $\Gamma^*=\left\{{\gamma}\right\}$ such that 
\[ |\mu_{\gamma}-\mu^*|_p < \delta
\quad\text{for every $\gamma\in\Gamma^*$.}
\] 

\item[Validation Criterion~IV.] Let $\delta^{\scriptscriptstyle M}_{\scriptscriptstyle \mathrm{IV}}$ be the smallest number such that the following holds.
There exist cycles $\Gamma'=\{\gamma_1,\dots,\gamma_M\}\subset\Gamma$ with $M\geq 2$ such that 
\begin{equation}\label{eq:diffmuslessdelta}
  \left|\mu_{\gamma_i}-\mu_{\gamma_j}\right|_p<\frac{M}{M-1}\delta^{\scriptscriptstyle M}_{\scriptscriptstyle \mathrm{IV}}
  \quad\text{for all $i\neq j$}
\end{equation}
and 
\begin{equation}\label{eq:diffmuslessdelta1}
\bigl|\mu_{\gamma_i}(U_k)-\mu_{\gamma_j}(U_k)\bigr|< 2\delta^{\scriptscriptstyle M}_{\scriptscriptstyle \mathrm{IV}}
\quad\text{for every $0\leq k\leq m-1$.}
\end{equation}
\end{description}
Let $M$ be maximal such that $\delta^{\scriptscriptstyle M}_{\scriptscriptstyle \mathrm{IV}}$ is sufficiently small, then these $M$ cycles in $\Gamma'$ play the role of the cycles ensured to exist in a hyperbolic context; see Hyperbolicity Consequence~IV.
\begin{remark}\label{Rem:third}
   Similarly to \cref{Rem:first,Rem:second,Rem:fourth}, given a hyperbolic system with a chosen finite grid of cubes $B_{\underline{i}}$ covering $\pi(A)$, it is possible to select an arbitrary large number of periodic orbits $\gamma_k$ that do not satisfy Validation Criterion~IV. From this perspective, one may interpret Validation Criterion~IV as confirming that the observed cycles arising from the physical process satisfy an even stronger statistical  criterion than the ones satisfying only Validation Criterion~I, II and III. In particular, the observed cycles in $\Gamma'$ are not arbitrary cycles and have an extra common property. Indeed, they appear to have even more statistical organization, namely they assign the same weight to the boxes in the grid. 
\end{remark}
\begin{remark}
The highest we can choose $M$ such that $\delta^{\scriptscriptstyle M}_{\scriptscriptstyle \mathrm{IV}}$ is small enough in the Validation Criterion~IV the more reliable the final model will be. 
\end{remark}

\begin{example}[Cell signalling process] 
For data in \cref{sec:BiologicalUseCase}, consider the projection $p=2$ to the coordinates corresponding to the grid squares at position $152$ and $153$ and also the projection $p=2$ to the coordinates corresponding to the grid squares at position $97$ and $113$.
 
There are two cycles, $M=2$, that satisfy Validation Criterion~IV with $\delta^{\scriptscriptstyle 2}_{\scriptscriptstyle \mathrm{IV}}=0.1335$ and three cycles, $M=3$, with $\delta^{\scriptscriptstyle 3}_{\scriptscriptstyle \mathrm{IV}}=0.3114$, for grid squares at position $152$ and $153$. For the projection to to the grid squares at position $97$ and $113$ there are two cycles that satisfy Validation Criterion~IV for with $\delta^{\scriptscriptstyle 2}_{\scriptscriptstyle \mathrm{IV}}=0.1817$ and three cycles with $\delta^{\scriptscriptstyle 3}_{\scriptscriptstyle \mathrm{IV}}=0.4143$.


These are found in the following way. First observe that \cref{eq:diffmuslessdelta1} in Validation Criterion~IV is satisfied for all pair of cycles. We will concentrate on the condition in \cref{eq:diffmuslessdelta}. Define the matrix $30\times 30$, $D_2$, as 
\[
D_2(i,j)=\left|\mu_{\gamma_i}-\mu_{\gamma_j}\right|_p.
\]
Start with $M=2$ and choose the pair $\gamma_{i_0},\gamma_{j_0}$ such that $
\delta^{\scriptscriptstyle 2}_{\scriptscriptstyle \mathrm{IV}}=D_2({i_0}, {j_0})$ is the smallest entry of the matrix. 
Take $M=3$ and define the array $D_3$ as 
\[ 
D_3({i}, {j}, {k}) = 
\max\bigl\{D_2(i,j), D_2(i,k), D_2(j,k)\bigr\}.
\]
Choose the triple $\gamma_{i_0},\gamma_{j_0}, \gamma_{k_0}$ such that $\delta^{\scriptscriptstyle 3}_{\scriptscriptstyle \mathrm{IV}}=D_3({i_0}, {j_0}, {k_0})$ is a smallest entry of the matrix $D_3$.  And so on.
\end{example} 

The previous example shows that, in the case of the Cell Signalling Process, the cycle measures differ significantly from one another. This confirms the concerns raised at the beginning of this subsection and is consistent with what one would expect: the available data do not support the existence of a sufficiently long cycle that could provide an accurate approximation of the physical measure of the process. Consequently, the statistical behaviour can only be captured by considering the collection of cycle measures as a whole. Alternatively, a more accurate approximation would require a larger amount of data.

\section{Cell signalling process for a kidney cell population}\label{sec:BiologicalUseCase}
We here provide a more detailed description of the scientific discovery task related to  the cell signalling process that we have used in \cref{sec:OverviewOfApproach,sec:method} for illustrating our approach.
It relates to understanding how cells synchronise against each other within a cell population. This is part of a larger quest in identifying emerging properties of cell populations that relate to directional/positional awareness and how individual cells contribute to form a tissue. 

\subsection{Description of the experiment and the data}\label{sec:DescExp}
The specific cell population consists of monoclonal Madin–Darby canine kidney (MDCK) cells, a model mammalian epithelial cell line widely used in studies of cell polarity, cell–cell adhesion, collective cell motility, toxicity, and responses to growth factors \cite{Hurley:2003aa,Dukes:2011aa}, and a standard system for studying collective cellular behaviour.

The experimental data consist of long time-lapse live-cell fluorescence microscopy recordings of intracellular calcium activity in such a population. Cells are plated at low density and grow into polarised monolayers, forming a spatially structured tissue with distinct interior and boundary regions. The monolayer is imaged by wide-field epifluorescence microscopy
\cite{Webb:2012aa} using a 63×/1.4 NA oil-immersion objective. Calcium signaling is monitored with GCaMP6m \cite{Chen:2013aa,Zhang:2023aa}, a genetically encoded calcium indicator stably expressed in cells, consisting of a circularly permuted green fluorescent protein fused to the calcium-binding protein calmodulin and the calmodulin-binding peptide M13. Binding of Ca²⁺ triggers a conformational change that strongly increases the green fluorescence (peak excitation 480 nm, peak emission 510 nm), so that the fluorescence intensity in the 2D microscopy images reports the intracellular Ca²⁺ concentration. The fluorescence measurements are converted into normalised calcium activity traces for each pixel.

To summarise, the data set is a time series of 512 × 512 images capturing the spatial distribution of Ca²⁺ activity across the population, equivalently, a high-dimensional collection of coupled time series describing the evolution of signalling activity in a dynamically interacting cellular population. The scientific discovery task is to investigate whether these data reveal emergent properties belonging to the cell population as a whole, making it a natural test case for the construction of data-driven dynamical models.

\subsection{Application of the model to the Cell Signalling Process}
\label{Sec:application_corr}
According to the model built above, the first conclusion we can deduce is that the cell signalling process is chaotic and can be described by a Markov chain with entropy approximately $2.3$. This Markov process allows to use statistical tools. 

As an example we will use our model to indicate how one could study the possible interaction between two not necessarily neighbouring cells. Each cell is approximately contained in a square of the grid in  \cref{Fig2}. The two squares containing the cells are labelled by $\underline i$ and $\underline j$. 

The question we would like to address is if and how long it takes to two cells to react to each other. To answer this question we are going to look at the delayed correlation between the signals of the two cells. 

We start to define properly what a delay correlation is in the context of our Markov process. Let $\xi:\Real^d\to\Real$ where $d=256$ is the dimension of the space where the attractor lives. For example $\xi$ can be the projection to the $\underline{i}^{th}$ coordinate. Let $k^*$ be the scale of the system and consider the graph $G_{k^*}$. The function $\xi$ induces a function $\xi_{k^*} \colon G_{k^*}\to\Real$ which is defined as follows: Notice that if $v$ is a vertex in $G_{k^*}$, $v\neq u_{k^*}$, then $v\in A\subset\Real^d$ and set $\xi_{k^*}(v):=\xi(v)$. Otherwise, $\xi_{k^*}$ is set to the average value over $U_{k^*}$:
\[ 
\xi_{k^*}( u_{k^*})
:=\frac{\int_{U_{k^*}}\xi d\mu}{\mu(U_{k^*})}
=\frac{\sum_{\gamma\in\Gamma}\sum_{x\in U_{k^*}\cap\gamma}{\varphi(x)}/{T_{\gamma}}}{\sum_{\gamma\in\Gamma}{N_{\gamma}(k^*)}/{T_{\gamma}}}.
\]
Observe that 
\begin{equation}\label{eq:intxi}
\int\xi_{k^*} d\mu_{k^*}=\sum_{v\in G_{k^*}}\xi_{k^*}(v)\mu_{k^*}(v)
\end{equation}
and a more careful calculation yields the following result.
\begin{lemma}
  \begin{equation*}
    \int\xi_{k^*} d\mu_{k^*} = \frac{1}{\#\Gamma} \sum_{\gamma \in \Gamma} \frac{1}{T_{\gamma}}\sum_{x \in \gamma} \xi(x)
  \end{equation*}
\end{lemma} 

Every $\gamma\in\Sigma_{k^*}$ gives rise to the corresponding time series 
\[
\Xi(t;\gamma)=\xi_{k^*}\bigl((\sigma^t\gamma)(0)\bigr) 
\quad\text{with $t\in\Integer$.}
\]

Now choose two normalized functions $\xi,\tilde\xi:\Real^d\to\Real$ with the corresponding time series $\Xi$ and $\tilde\Xi$. The normalization requires that
\[ 
\Expect[\xi_{k^*}]=\int\xi_{k^*}d\mu_{k^*}=\int\xi d\mu=0.
\]
If needed replace $\xi$ by $\xi-\int\xi d\mu$.
The correlation of $\xi$ and $\tilde\xi$ is defined as the cosine of the angle between $\Xi$ and $\tilde\Xi$. More precisely, let $\gamma\in\Sigma^*_{k^*}$ be a $\mu^*_{k^*}$-typical point, then 
\begin{align*}
\operatorname{Corr}(\xi,\tilde\xi)
&= \lim_{T\to\infty}
   \dfrac{\frac{1}{T} \sum_{0\leq t<T}\Xi(t;\gamma)\tilde\Xi(t;\gamma)}
   {\Bigl(\frac{1}{T}\sum_{0\leq t<T}\Xi^2(t;\gamma)\Bigr)^{1/2}
    \Bigl(\frac{1}{T}\sum_{0\leq t<T}\tilde\Xi^2(t;\gamma)\Bigr)^{1/2}
   }
\\[0.5em]
&= \dfrac{\int\xi_{k^*}\tilde\xi_{k^*}d\mu_{k^*}}{\left(\int\xi^2_{k^*}d\mu_{k^*}\right)^{1/2}\left(\int\tilde\xi^2_{k^*}d\mu_{k^*}\right)^{1/2}}.
\end{align*}
This follows from the Birkhoff Ergodic Theorem, In particular, the independence of the correlation on $\gamma$. The integrals in the previous formula can be calculated using \cref{eq:intxi}. 

The \emph{$s$-delayed} version of $\xi_k \colon \Sigma_{k}\to \Real$ is 
\[
\xi_k^{(s)}=\xi_k\circ \sigma_{k}^s 
\quad\text{with $s\in \Integer$.}
\]
Observe that if $\xi$ is normalized, then each delayed $\xi^{(s)}_k$ is normalized.

The \emph{$s$-delayed correlation} of the normalized functions
$\xi,\tilde\xi:\Real^d\to\Real$ is $\operatorname{Corr}\left(\xi,\tilde\xi^{(s)}\right)$. Namely,
\[
\operatorname{Corr}^{(s)}(\xi,\tilde\xi)
=
\dfrac{\int\xi_{k^*}\tilde\xi_{k^*}^{(s)}d\mu_{k^*}}{\Bigl(\int\xi^2_{k^*}d\mu_{k^*}\Bigr)^{1/2}\Bigl(\int\tilde\xi^2_{k^*}d\mu_{k^*}\Bigr)^{1/2}}.
\]
The \emph{correlation function} of $\xi,\tilde\xi$ is
\[
s\mapsto \operatorname{Corr}^{(s)}\left(\xi,\tilde\xi\right)
\]
Recall that $\underline i$ and $\underline j$ are the locations of the squares containing the cells. The integrals needed to calculate the delayed correlation can be determined using \cref{eq:intxi} and

\begin{lemma} Let $\xi, \tilde\xi:\Real^d\to \Real$ be defined by $\xi(\underline x)=x_{\underline i}$ and $\tilde \xi(\underline x)=x_{\underline j}$.
Then
\[
\int\xi_{k^*}\tilde\xi_{k^*}^{(s)}d\mu_{k^*}
 =\sum_{v_0 \in G_{k^*}} \xi_{k^*}(v_0) \mu_{k^*}(v_0)\left[
 \sum_{v_s \in G_{k^*}} \tilde \xi_{k^*}(v_s) P_{k^*}^s(v_0, v_s))\right]
\]
where $P_{k^*}$ is the transition matrix of the Markov chain $(\Sigma_{k^*}, \mu_{k^*}, \sigma_{k^*})$.
Moreover, the delayed correlations decay exponentially,
\[
\lim_{s\to\infty}\int\xi_{k^*}\tilde\xi_{k^*}^{(s)}d\mu_{k^*}=0.
\]
\end{lemma}

\begin{proof} 
To start with, observe
\begin{align*}
 \int\xi_{k^*}\tilde\xi_{k^*}^{(s)}d\mu_{k^*}
 &=
 \!\!\!\!\!\sum_{[v_0, v_1, \cdots, v_s]} \xi_{k^*}(v_0) \tilde \xi_{k^*}(v_s) \mu_{k^*}([v_0, v_1, \cdots, v_s])
 \\ 
 &=
 \sum_{v_0,v_s} \xi_{k^*}(v_0) \tilde \xi_{k^*}(v_s) 
 \!\!\!\!\!\sum_{[v_0, x_1, \dots, x_{s-1}, v_s]}\mu_{k^*}([v_0, x_1,\cdots, x_{s-1}, v_s])
 \\ 
 &=
 \sum_{v_0} \xi_{k^*}(v_0) \mu_{k^*}(v_0)\left[
 \sum_{v_s} \tilde \xi_{k^*}(v_s) P_{k^*}^s(v_0, v_s))\right]. 
\end{align*} 
The converges as $s \to \infty$ is then 
\begin{multline*}
\lim_{s\to\infty}\int\xi_{k^*}\tilde\xi_{k^*}^{(s)}d\mu_{k^*}
 =
 \lim_{s\to\infty}\sum_{v_0} \xi_{k^*}(v_0) \mu_{k^*}(v_0)\left[
 \sum_{v_s} \tilde \xi_{k^*}(v_s) P_{k^*}^s(v_0, v_s))\right]
 \\ 
 =
 \left[\sum_{v_0} \xi_{k^*}(v_0) \mu_{k^*}(v_0)\right]\cdot\left[
 \sum_{v_s} \tilde \xi_{k^*}(v_s) \mu_{k^*}(v_s)\right]
 =
 \int\xi_{k^*} d\mu_{k^*}\cdot \int\xi_{k^*} d\mu_{k^*}=0.
\end{multline*} 
Here we used that $\xi$ and $\tilde \xi$ are normalized and \cref{transitionmatrix}.
\end{proof}

\begin{figure}[ht]
    \centering
    \begin{subfigure}{0.3\linewidth}
        \centering
        \includegraphics[width=\linewidth]{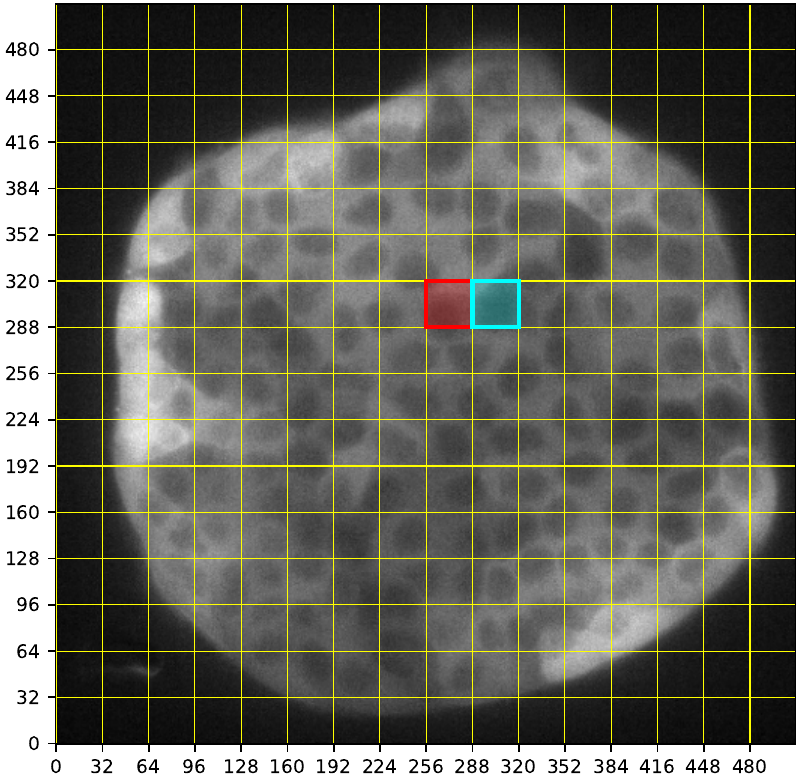}
        \caption{Positions 152 and 153.}
        \label{fig:positions15215}
    \end{subfigure}
    \hfill
    \begin{subfigure}{0.3\linewidth}
        \centering
        \includegraphics[width=\linewidth]{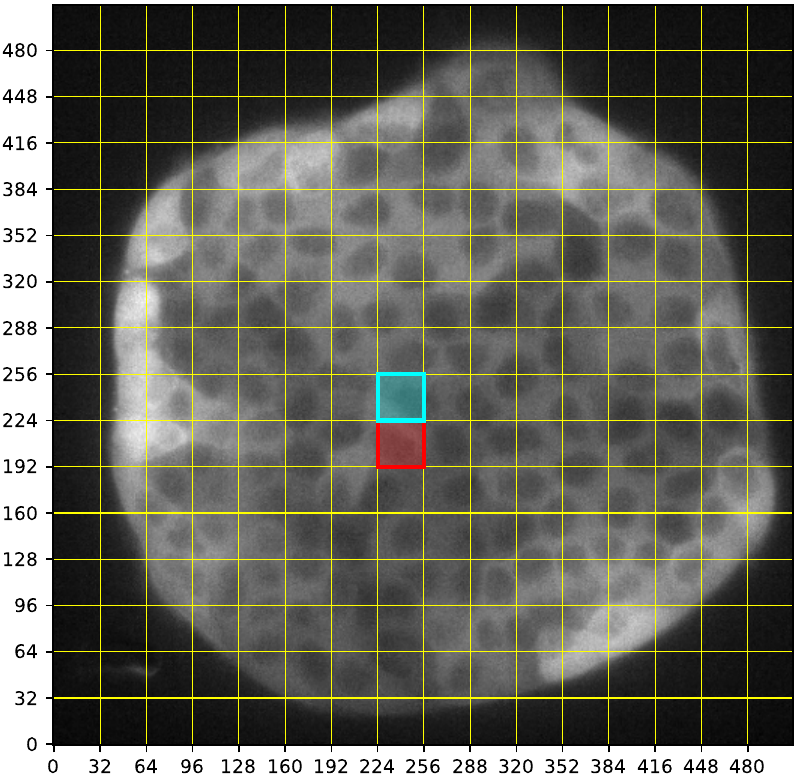}
        \caption{Positions 103 and 119.}
        \label{fig:positions103119}
    \end{subfigure}
    \hfill
    \begin{subfigure}{0.3\linewidth}
        \centering
        \includegraphics[width=\linewidth]{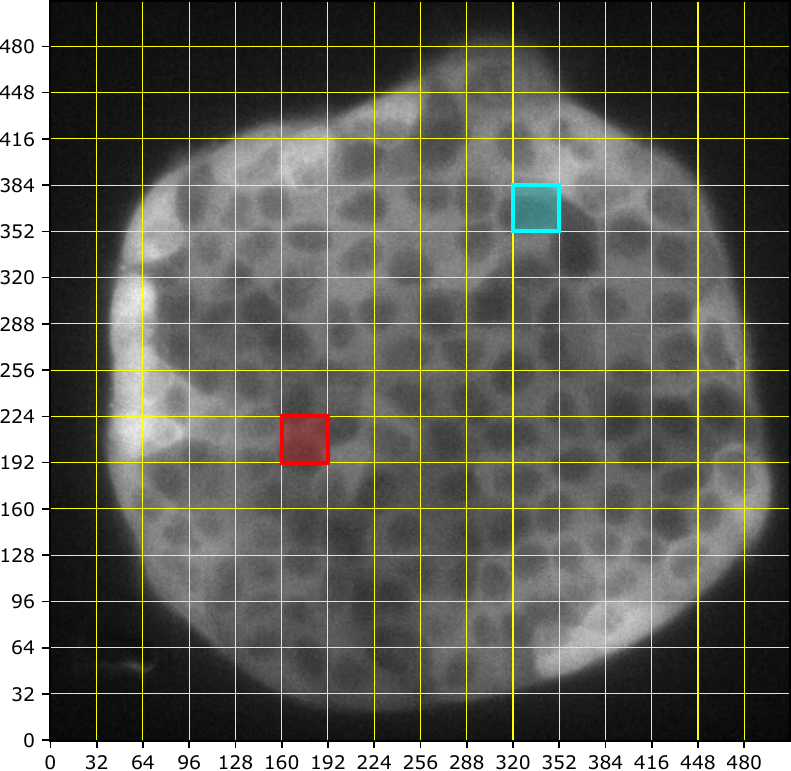}
        \caption{Positions 101 and 186.}
        \label{fig:positions101186}
    \end{subfigure}
    \caption{Position of cells under analysis.}
    \label{fig:positions}
\end{figure}

\begin{figure}[ht]
\centering
\begin{subfigure}{0.48\linewidth}
    \centering
    \includegraphics[width=\linewidth]{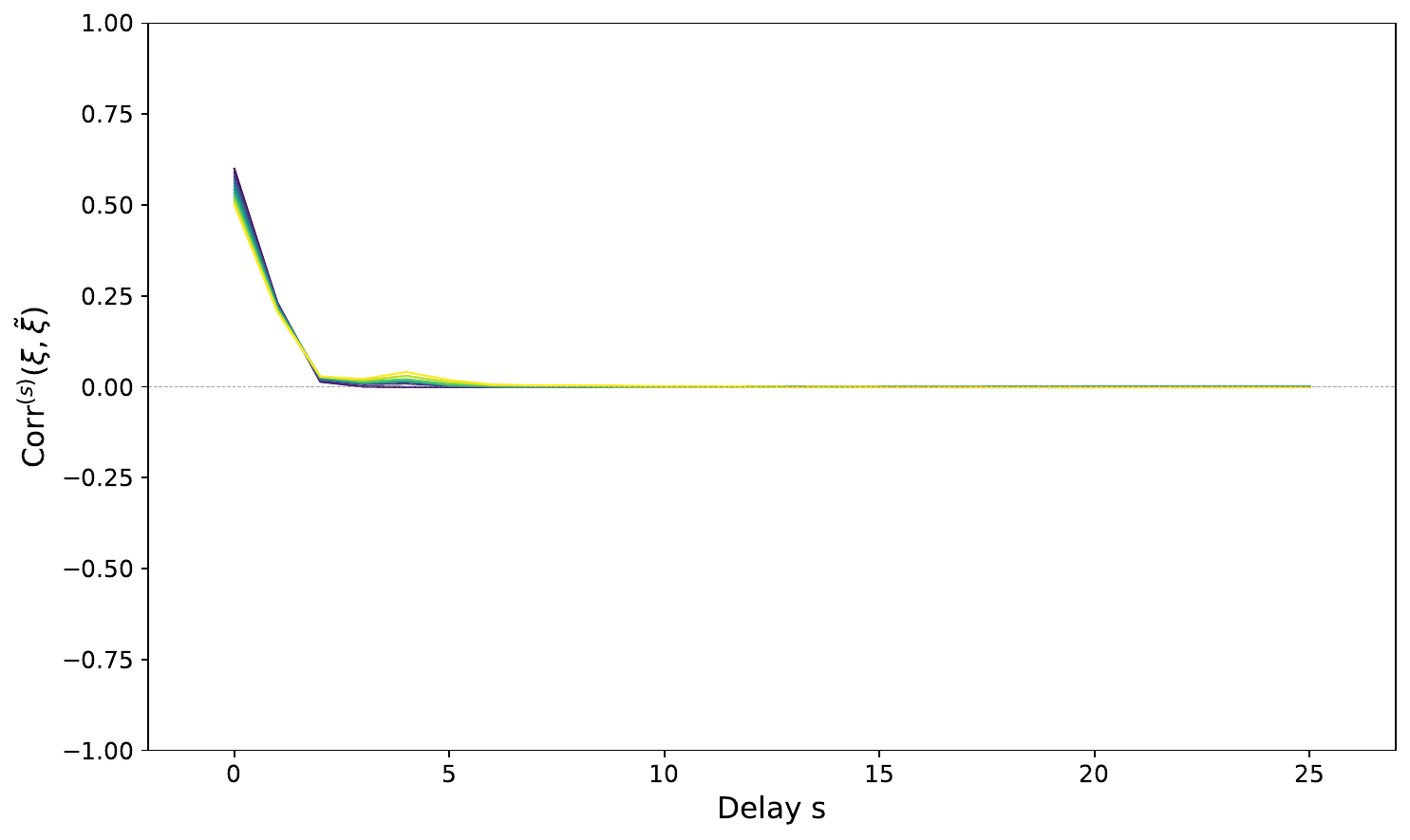}
    \caption{Position 152 and 153.}
    \label{fig:allcorr_allcycles}
\end{subfigure}
\hfill
\begin{subfigure}{0.48\linewidth}
    \centering
    \includegraphics[width=\linewidth]{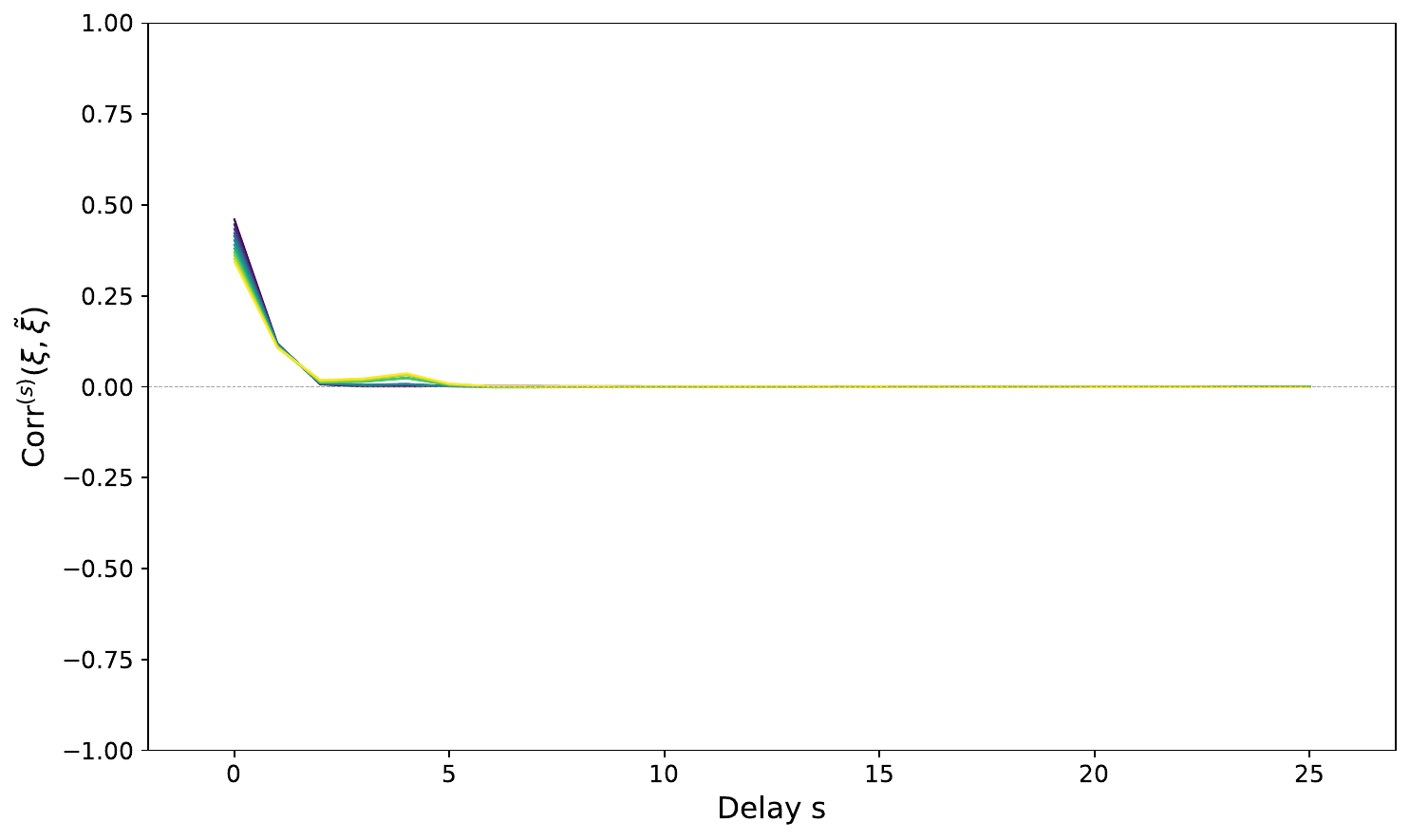}
    \caption{Position 103 and 119.}
    \label{fig:allcorr_allcycles1}
\end{subfigure}
\hfill
\begin{subfigure}{0.48\linewidth}
    \centering
    \includegraphics[width=\linewidth]{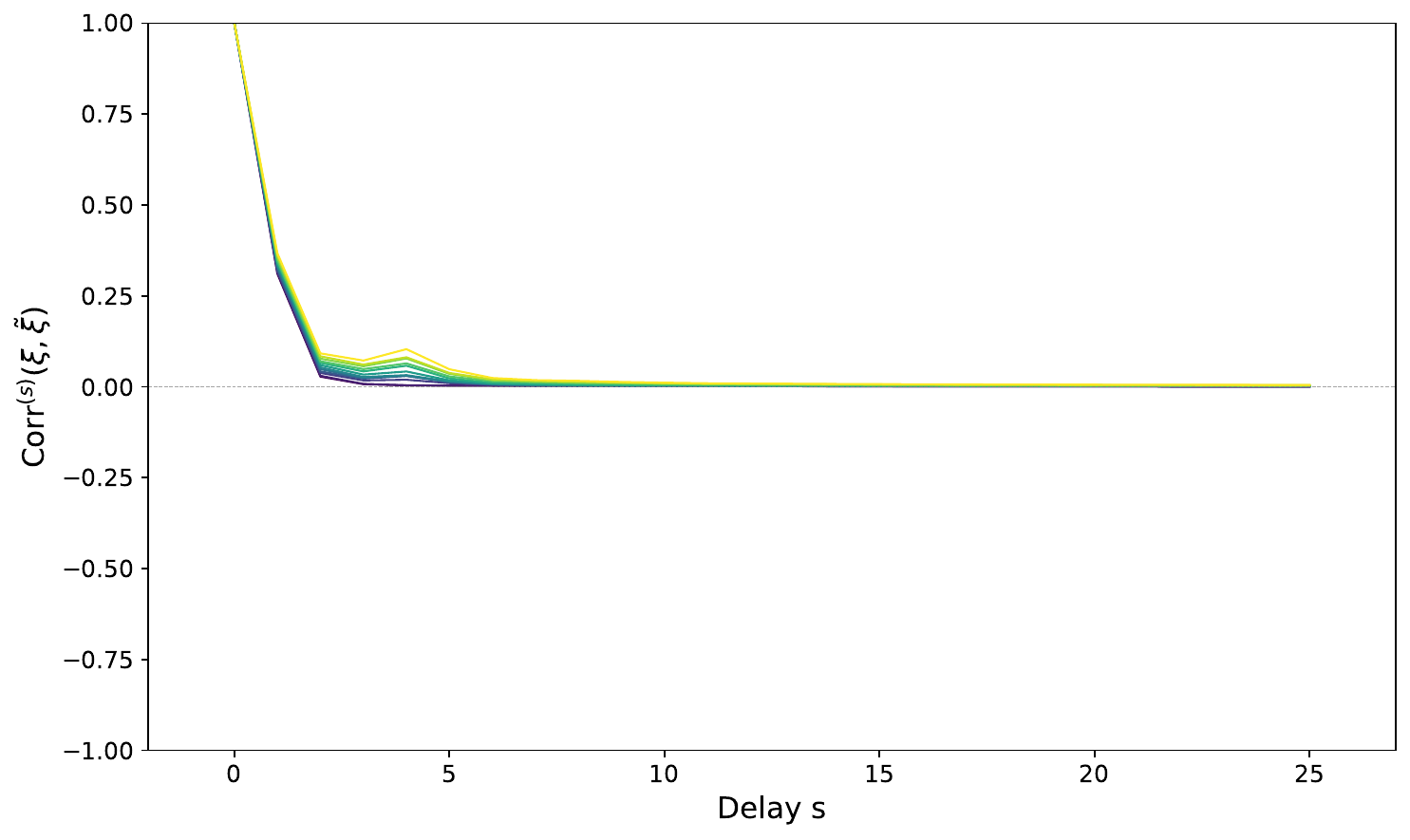}
    \caption{Position 152 and 152.}
\label{fig:allcorr_allcycles_gr0_152_152}
\end{subfigure}
\hfill
\begin{subfigure}{0.48\linewidth}
    \centering
    \includegraphics[width=\linewidth]{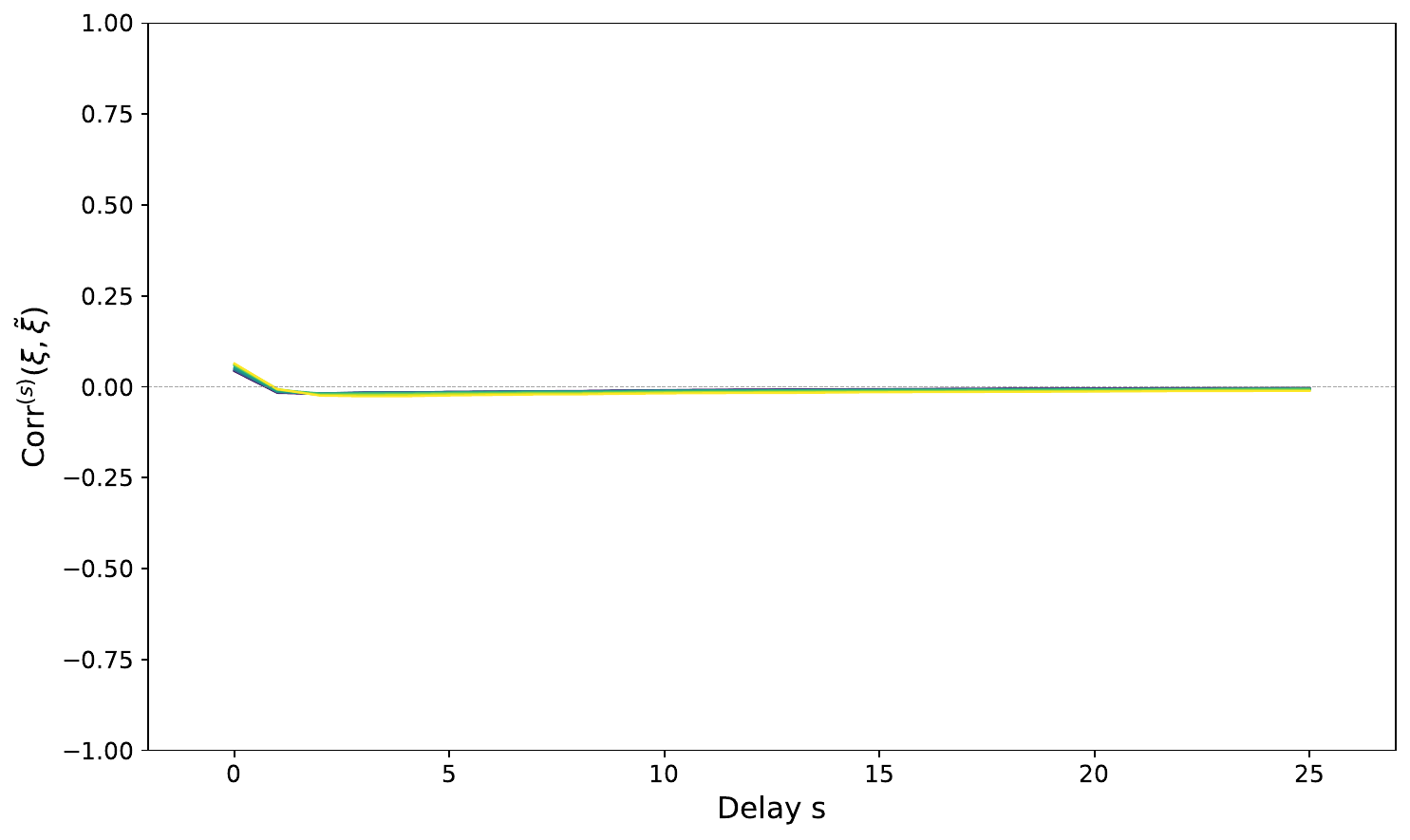}
    \caption{Position 101 and 186.}
\label{fig:allcorr_allcycles_gr0_101_186}
\end{subfigure}
\hfill
\begin{subfigure}{0.25\linewidth}
    \centering
    \raisebox{1.3cm}{\includegraphics[width=\linewidth]{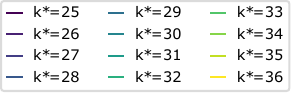}}
\end{subfigure}
\caption{Delayed correlation profiles for three representative pairs of cells for $X_{0-6,1000}$, shown for all $k^*\in S$.}
\label{fig:allcorr_subfigures}
\end{figure}

\

\begin{example}[Cell signalling process] 
Consider the system in \cref{sec:CellProcess}. 
\Cref{fig:allcorr_allcycles,fig:allcorr_allcycles1,fig:allcorr_allcycles_gr0_152_152}
show non trivial correlation between certain pairs of cells. First of all observe that the correlation graphs are essentially the same for any choice of model $k^*\in S$. This indicates consistency of the method.

Observe that all pairs of cells considered have a non zero initial correlation. This indicates a global instantaneous organization. This instantaneous organization is weaker for cells at positions $101$ and $186$  which are very far apart, see  \cref{fig:allcorr_allcycles_gr0_101_186}  and stronger for the cells at positions $152$ and $153$ and the ones at position $103$ and $119$ which are neighbours, see 
\cref{fig:allcorr_allcycles,fig:allcorr_allcycles1} (of course the strongest instantaneous organization occurs when one consider the correlation of the cell at position $152$ with itself, see \cref{fig:allcorr_allcycles_gr0_152_152}). 

Moreover, in the cases when the instantaneous organization is stronger, we also see a second spike. In all cases the spike occurs after $20$ seconds. This might indicate a response of one cell to the other. In particular, the cell at position $153$ reacts to a signal sent by the cell at position $152$, see \cref{fig:allcorr_allcycles}. Similarly for the cells at positions $103$ and $119$, see \cref{fig:allcorr_allcycles1}. Observe that the pairs of cells considered are not neighbouring, see \cref{fig:positions15215,fig:positions103119}, nevertheless the response time is the same. 

Surprisingly, at the same time of $20$ second there is a spike in the correlation graph of the cell at position $152$ with itself, see \cref{fig:allcorr_allcycles_gr0_152_152}. Of course one cannot talk about cell response in this case, maybe one should start to consider the possibility of a local rhythm of frequency $1/20$ Hz which is visible when the initial instantaneous organization is stronger. 

One might consider whether the global organization of the whole cell colony has a rhythm of frequency $1/20$ Hz. All these needs of course further investigation.

\end{example}

 \section{Robustness of the model construction}

The aim of this section is to investigate the robustness of the construction of the model. A fundamental requirement for robustness is that different datasets, with the same characteristics, obtained from the same physical process, should lead to essentially the same model. Also in this case, the theory of hyperbolic systems serves as a guideline for assessing robustness.

First observe that, in our setting, the characteristics of the data are: the time $T$, the interval of time $\Delta T$, the dimension of the instantaneous observations $D$, the dimension of the projected data $d$ and the number of grids $K$.

Since different choices of data may lead to different Markov chains, a criterion is needed to determine when the resulting models should be regarded as equivalent. Motivated by the classification theory of Markov chains, we compare models through the entropy. The Markov chains arising in our construction are measure-theoretically isomorphic if and only if they have the same entropy, see \cite{Ornstein1970, FriedmanOrnstein1970}. We therefore say that two models are essentially the same if their entropies differ by at most a prescribed threshold.

\begin{description}
\item[Hyperbolicity robustness consequence.] 
Given a hyperbolic system with physical measure $\mu^*$ and entropy $h^*$. Then for every $\delta>0$, there are data characteristics such that a typical data set with this characteristics creates a model $\left(\Sigma_{k^*},\mu_{k^*}, \sigma_{k^*}\right)$ whose entropy $h_{k^*}$ satisfies,
\[ |h_{k^*}-h^*|< \delta.
\]  

\item[Validation robustness criterion.] 
Let $X_1,\dots, X_N$ be data sets obtained from the same physical process and let $h_{k^*_1},\dots, h_{k^*_N}$ be the entropies of the corresponding models. Let $\delta_{\scriptscriptstyle R}$ be the smallest number such that the following holds.
\[ |h_{k^*_i}-h_{k^*_j}| < \delta_{\scriptscriptstyle R}, \text{ for all }i\neq j.
\] 
\end{description}
If the Validation robustness criterion is satisfied for sufficiently small $\delta_{\scriptscriptstyle R}$, then the models obtained from different data sets are essentially the same and, in the case under consideration, the construction of the model is robust. 
Moreover, robustness of the construction of the method would be another confirmation that the obtained Markov chains models are reasonable descriptions of the physical process. 

We discuss now the robustness of the use case, i.e. the Cell Signalling Process. In the previous examples we gave the results obtained for data taken by the first $6$ hour movie of the total $24$ hours we were given. We denote this data set by $X_{0-6,1000}$. In this case, $T=4318$, $\Delta T=5$, $D=512\times 512$, $d=256$ and $K=1000$. The results are given in the previous sections. 
To check robustness we also build a model based on the following data sets. First, we consider the next $6$ hours of the initial $24$ hour movie we were given. We denote this data set by $X_{6-12,1000}$. These two data sets have exactly the  same characteristics.  The results are presented in \cref{sec:next6hour}.

Next we return to the initial $6$ hour movie and we consider more projected time series, namely we run experiments for $K=1500$ and $K=5000$. The corresponding data sets are denoted by $X_{0-6,1500}$ and $X_{0-6,5000}$. The other characteristic are kept the same. The results are discussed in \cref{sec:longerdata}. 

Below is a list of characteristics which will be systematically presented and discussed in each model constructed. 
\begin{description}
\item[Data Characteristics:]
 $\Delta T, T, D, K, d $.

\item[Model Characteristics:]
Size of $X$, size of $A$, Distribution of Cycle Periods, Validation of Criteria I and II (\cref{fig:sgraph}), $\delta_{\scriptscriptstyle \mathrm{II}}$, Individual entropy profile for each cycle (\cref{fig:entropies_eachcycle}), The aggregate entropy profile when considering all cycles (\cref{fig:entropiesallgoodcycles_a}), $S$, $\delta_{\scriptscriptstyle \mathrm{III}}$,  $\delta^{\scriptscriptstyle 2}_{\scriptscriptstyle \mathrm{IV}}$, $\delta^{\scriptscriptstyle 3}_{\scriptscriptstyle \mathrm{IV}}$. 

\item[Robustness Characteristics:]
$\delta_{\scriptscriptstyle R}$.
\end{description}

\subsection{Different data sets}\label{sec:next6hour}
We collect here the results derived from the same movie using the data corresponding to the second block of 6 hours. Notice, the data obtained from the second block of 6 hours can be considered as obtained from a second movie of the cell colony, a movie independent from the movie obtained during the first 6 hours. 

 \begin{table}[h!]
\centering
\caption{Summary of model characteristics over two datasets: $X_{0-6, 1000}$ and $X_{6-12, 1000}$.}
\label{tab:summary_datasets}
\resizebox{\textwidth}{!}{%
\begin{tabular}{lrrcccclll}
\toprule
\multicolumn{1}{c}{Dataset} & \multicolumn{1}{c}{$|X|$} & \multicolumn{1}{c}{$|A|$} & Cycles & \multicolumn{1}{c}{Periods} & $\delta_{\mathrm{II}}$ & $S$ & $\delta_{\mathrm{III}}$  & \multicolumn{1}{c}{Entropy}  & \multicolumn{1}{c}{Entropy Profile}  \\
\midrule
$X_{0-6,1000}$ & $4\,318\,000$  & $69\,846$ & 30 
& \cref{tab:period_orbits} & $0.10$ & $[25,36]$ 
& $0.42$ 
&2.36
   & \cref{fig:sgraph}, \cref{fig:entropies_gr0_r1000} 
     \\
$X_{6-12,1000}$ & $4\,318\,000$  & $139\,835$  & 63 
& \cref{tab:period_orbits_gr1} & $0.19$ & $[29,47]$ 
& $0.91$ 
&2.51
    & \cref{fig:sgraph_gr1}, \cref{fig:entropies_gr1} 
   \\
\bottomrule
\end{tabular}%
}
\end{table}

 \subsection{Different number of projected time series}\label{sec:longerdata}
 We add data by taking more grids, $K=1500$ and $K=5000$.
\begin{table}[h!]
\centering
\caption{Summary of model characteristics across grid densities for $X_{0-6}$.}
\label{tab:summary_r}
\resizebox{\textwidth}{!}{%
\begin{tabular}{lrrcccclll}
\toprule
\multicolumn{1}{c}{Dataset} & \multicolumn{1}{c}{$|X|$} & \multicolumn{1}{c}{$|A|$} & Cycles & \multicolumn{1}{c}{Periods} & $\delta_{\mathrm{II}}$ & $S$ & $\delta_{\mathrm{III}}$  & \multicolumn{1}{c}{Entropy} & \multicolumn{1}{c}{Entropy Profile}  \\
\midrule
$X_{0-6,1000}$ & $4\,318\,000$  & $69\,846$ & 30 
& \cref{tab:period_orbits} 
& $0.10$ & $[25,36]$ & $0.42$ & 2.36
    & \cref{fig:sgraph}, \cref{fig:entropies_gr0_r1000} 
     \\
$X_{0-6,1500}$ & $6\,480\,000$  & $84\,138$  & 37 
 & \cref{tab:period_orbits_gr0_1500} 
& $0.11$ & $[27,38]$ & $0.49$ & 2.45
    & \cref{fig:sgraph_gr0_1500}, \cref{fig:entropies_gr0_r1500} 
    \\
$X_{0-6,5000}$ & $21\,590\,000$ & $146\,087$ & 64 
& \cref{tab:period_orbits_gr0_5000} 
& $0.13$ & $[32,43]$ & $0.58$ & 2.59
    & \cref{fig:sgraph_gr0_5000}, \cref{fig:entropies_gr0_r5000} 
     \\
\bottomrule
\end{tabular}%
}
\end{table}

\subsection{Model reliability}
Of course we can not claim that the Cell signalling process is equivalent to the obtained Markov chain models. The least we can ask is that different data sets result in comparable models. We use robustness as reliability criterium. 

We have four data sets obtained from the Cell Signalling Process. The corresponding robustness characteristics is $$\delta_R=0.2361.$$ The specific questions of the study of the cell signalling process have to determine whether this is small enough. 

We finish with some remarks about the obtained models.
\begin{itemize}
\item[1)] Given the complexity of the Cell Signalling Process it is remarkable how consistent the results are, e.g. the entropies are essentially the same,  the aggregate entropy profiles are very similar. Even the individual cycles share similar entropy profiles. 
\item[2)] Observe the the periods of the cycle are all, except for three cycles, between 2000 and 2500. One might have expected that in the periods coming from the consecutive data sets $X_{0-6,1000}$, $X_{0-6,1500}$, and $X_{0-6,5000}$ would have been increasing. The tables of periods show that that did not happen. In particular, the cycles obtained from $X_{0-6,5000}$ are not longer than the cycles from $X_{0-6,1000}$ and are not able to describe information on smaller scale.
This is probably related to the observation that the bundling widths $\delta_{\mathrm{II}}$ and  $\delta_{\mathrm{III}}$ did not decrease neither. If one needs to tighten the bundling constants one might need to refine the grid, see \cref{Fig2}.
\item[3)] In \cref{fig:entropies_eachcycle_gr1}, the entropy profiles of the cycles obtained form $X_{6-12,1000}$, shows clearly two outlying cycles, see also \cref{fig:sgraph_gr1}. One might consider to eliminate these outlying cycles. Observe, that these two cycles might be responsible for the large $\delta_{\mathrm{III}}$  of $X_{6-12,1000}$.
\item[4)] There are cells which changed location from one during the first 6 hours to another in the second 6 hours. Careful studies should incorporate this motion of cells. Observe that all models based on the data from the first 6 hours do show the correlation spike after 20 seconds.
\end{itemize}


\newpage
\FloatBarrier
\clearpage

\appendix
\section*{Appendix I}

The figures and tables corresponding to the different sets of data are collected in Appendix I.

 \begin{table}[!ht]
\centering
\caption{Distribution of Cycle Periods for $X_{6-12,1000}$}
\label{tab:period_orbits_gr1}
\begin{tabular}{cl}
\toprule
\textbf{\# Orbits} & \textbf{Period} \\
\midrule
\addlinespace
3 & 2162, 2175 \\
\addlinespace
2 & 2163, 2169, 2172, 2174, 2176, 2178, 2181, 2182, 2185, 2188, \\
  & 2190, 2191, 2201, 2202, 2208, 2278 \\
\addlinespace
1 & 2170, 2173, 2177, 2183, 2187, 2189, 2196, 2198, 2227, 2233, \\
  & 2237, 2243, 2250, 2258, 2259, 2268, 2269, 2282, 2287, 2337, \\
  & 2354, 2377, 2413, 2431, 2550 \\
\bottomrule
\end{tabular}
\end{table}
     \begin{figure}[!ht]
  \centering
\includegraphics[width=0.6\linewidth]{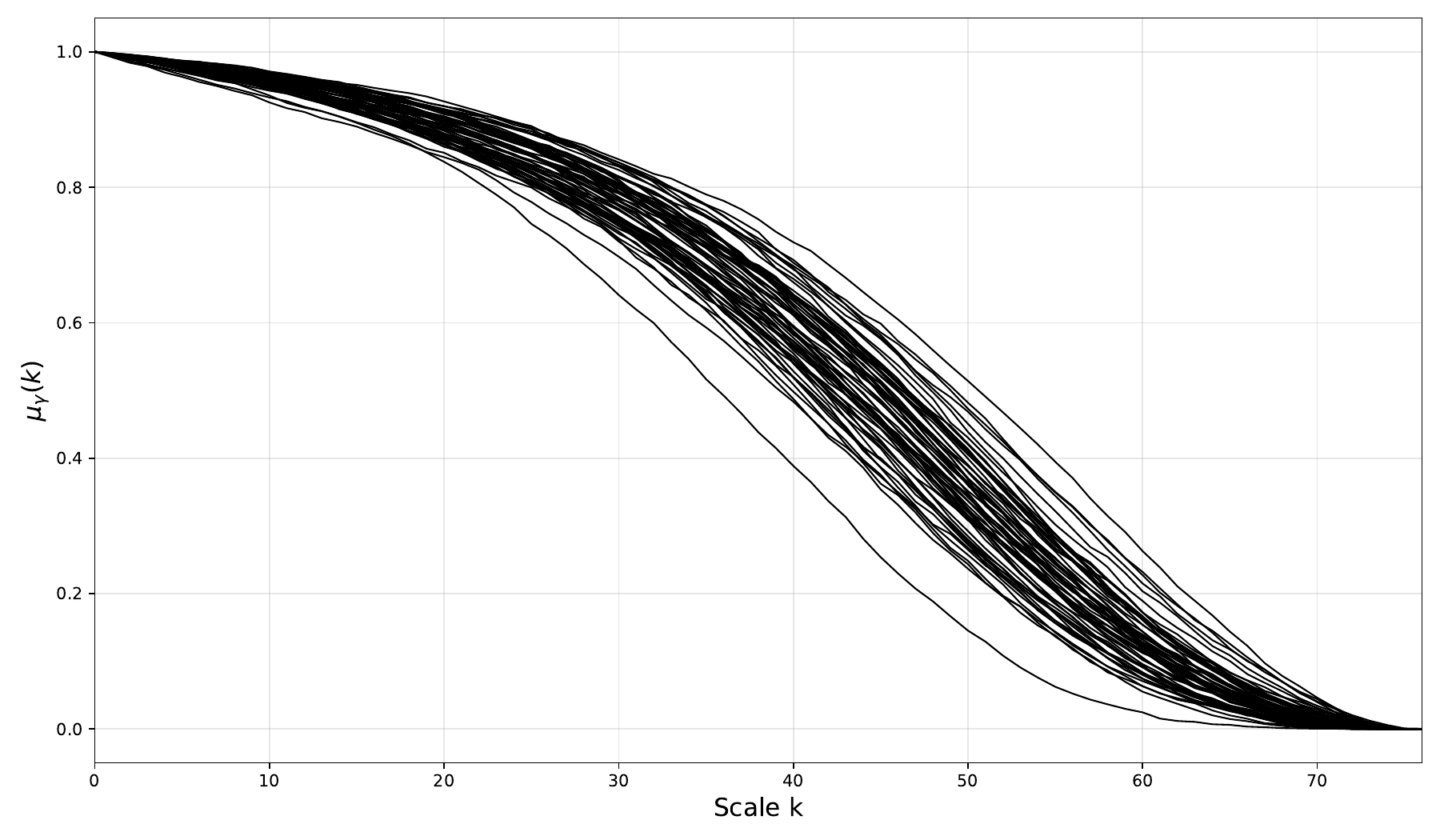}
  \caption{Validation of Criteria I and II for $X_{6-12,1000}$}
  \label{fig:sgraph_gr1}
\end{figure}
   
 \begin{figure}[!ht]
  \centering
  \begin{subfigure}[t]{0.48\linewidth}
    \centering
    \includegraphics[width=\linewidth]{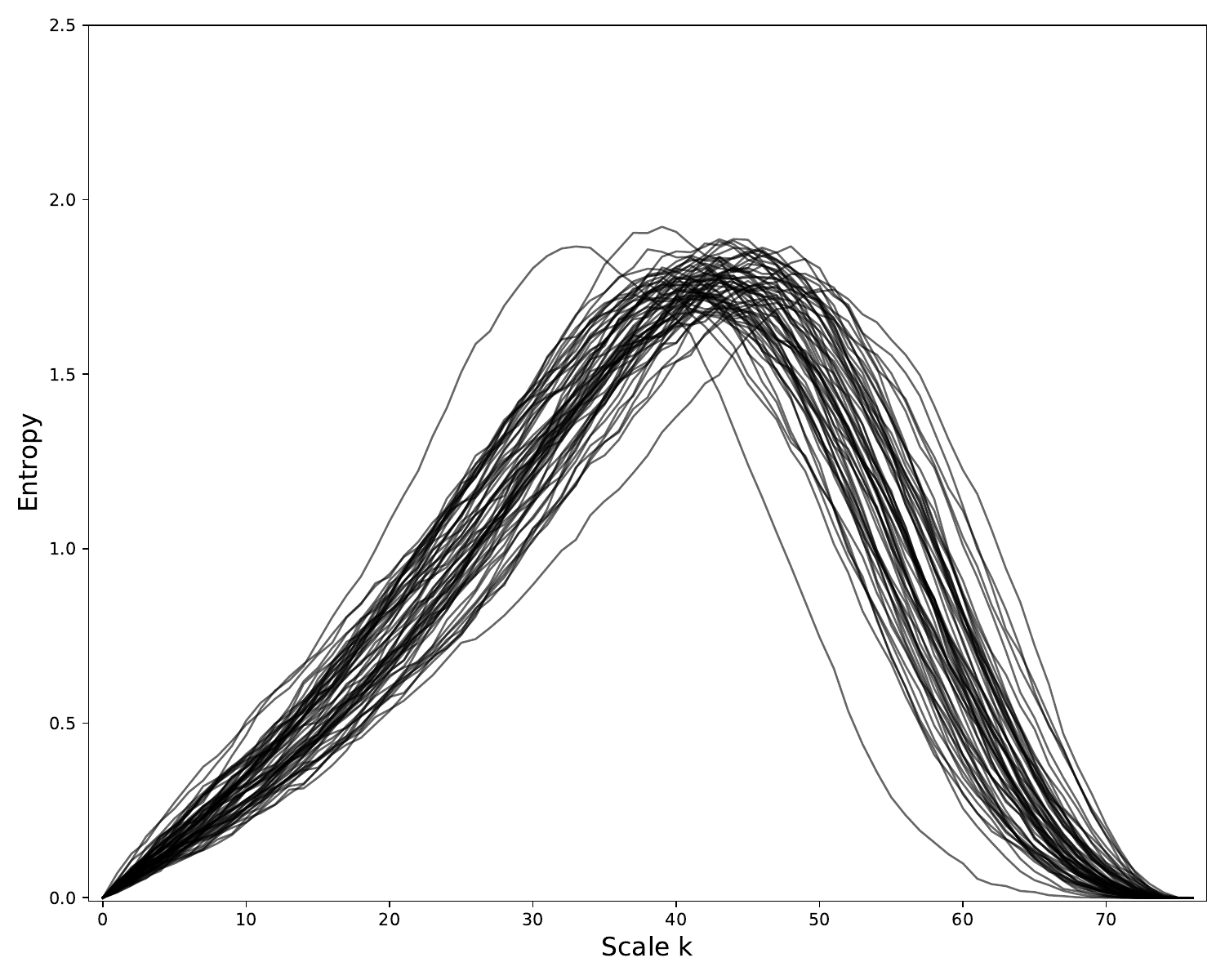}
    \caption{Individual entropy profile for each cycle.}
    \label{fig:entropies_eachcycle_gr1}
  \end{subfigure}
  \hfill
  \begin{subfigure}[t]{0.48\linewidth}
    \centering
    \includegraphics[width=\linewidth]{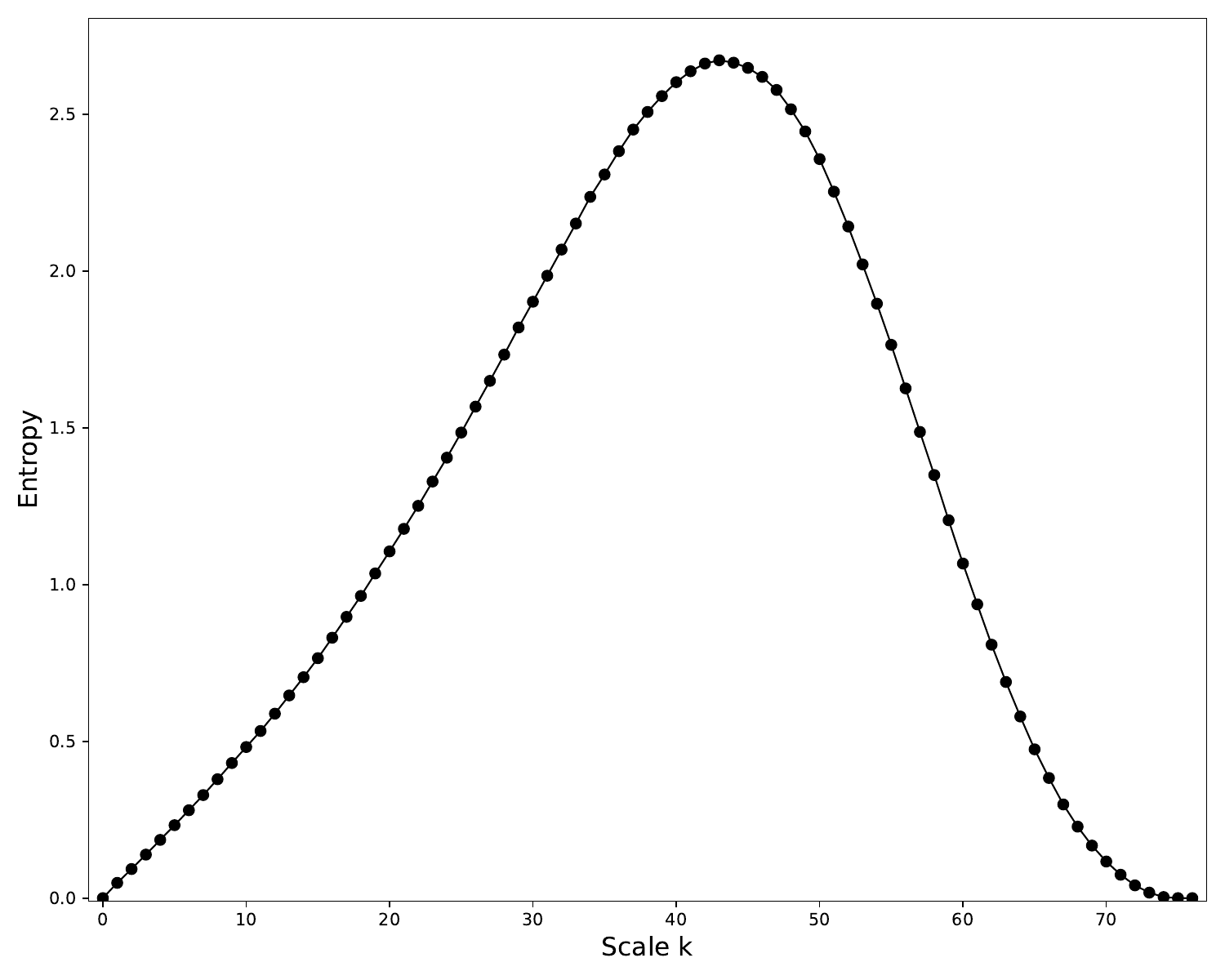}
    \caption{Aggregate entropy profile for all cycles in $\Gamma$.}
    \label{fig:entropiesallgoodcycles_a_gr1}
  \end{subfigure}
  \caption{Entropy profiles for $X_{6-12,1000}$.}
  \label{fig:entropies_gr1}
\end{figure}

\FloatBarrier
\clearpage

\begin{table}\centering\caption{Distribution of Cycle Periods for $X_{0-6,1500}$.}\label{tab:period_orbits_gr0_1500}
\begin{tabular}{cl}
\toprule
\textbf{\# Orbits} & \textbf{Period} \\
\midrule
2 & 2173, 2175, 2197, 2240 \\
\addlinespace
1 & 2162, 2163, 2165, 2166, 2168, 2171, 2172, 2174, 2178, \\
  & 2182, 2183, 2188, 2191, 2198, 2203, 2207, 2213, 2216, \\
  & 2222, 2236, 2246, 2252, 2257, 2263, 2274, 2275, 2362, \\
  & 2436, 4445 \\
\bottomrule
\end{tabular}
\end{table}

     \begin{figure}
  \centering
\includegraphics[width=0.6\linewidth]{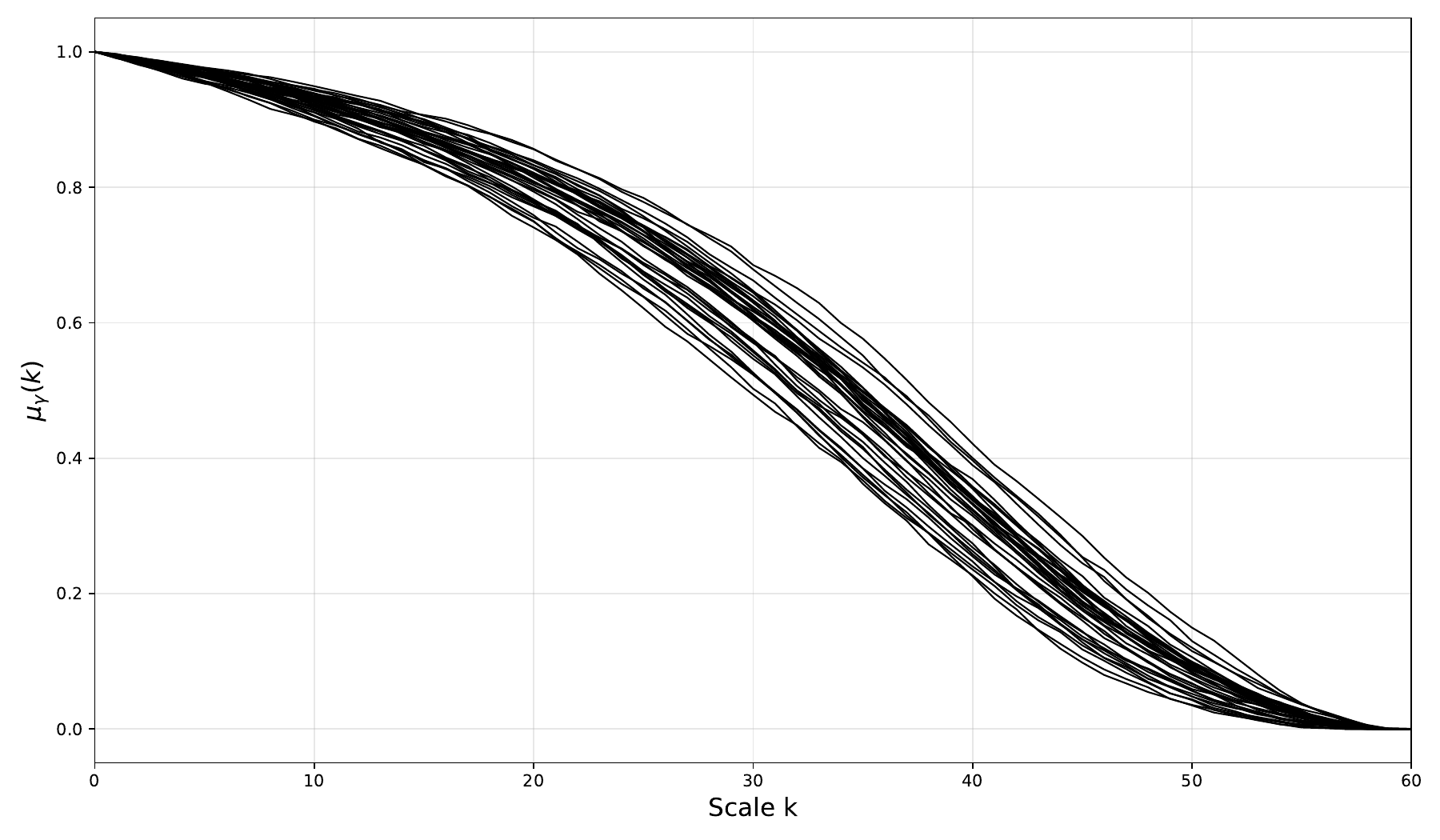}
  \caption{Validation of Criteria I and II for $X_{0-6,1500}$}
\label{fig:sgraph_gr0_1500}
\end{figure}

   \begin{figure}
  \centering
  \begin{subfigure}[t]{0.48\linewidth}
    \centering
    \includegraphics[width=\linewidth]{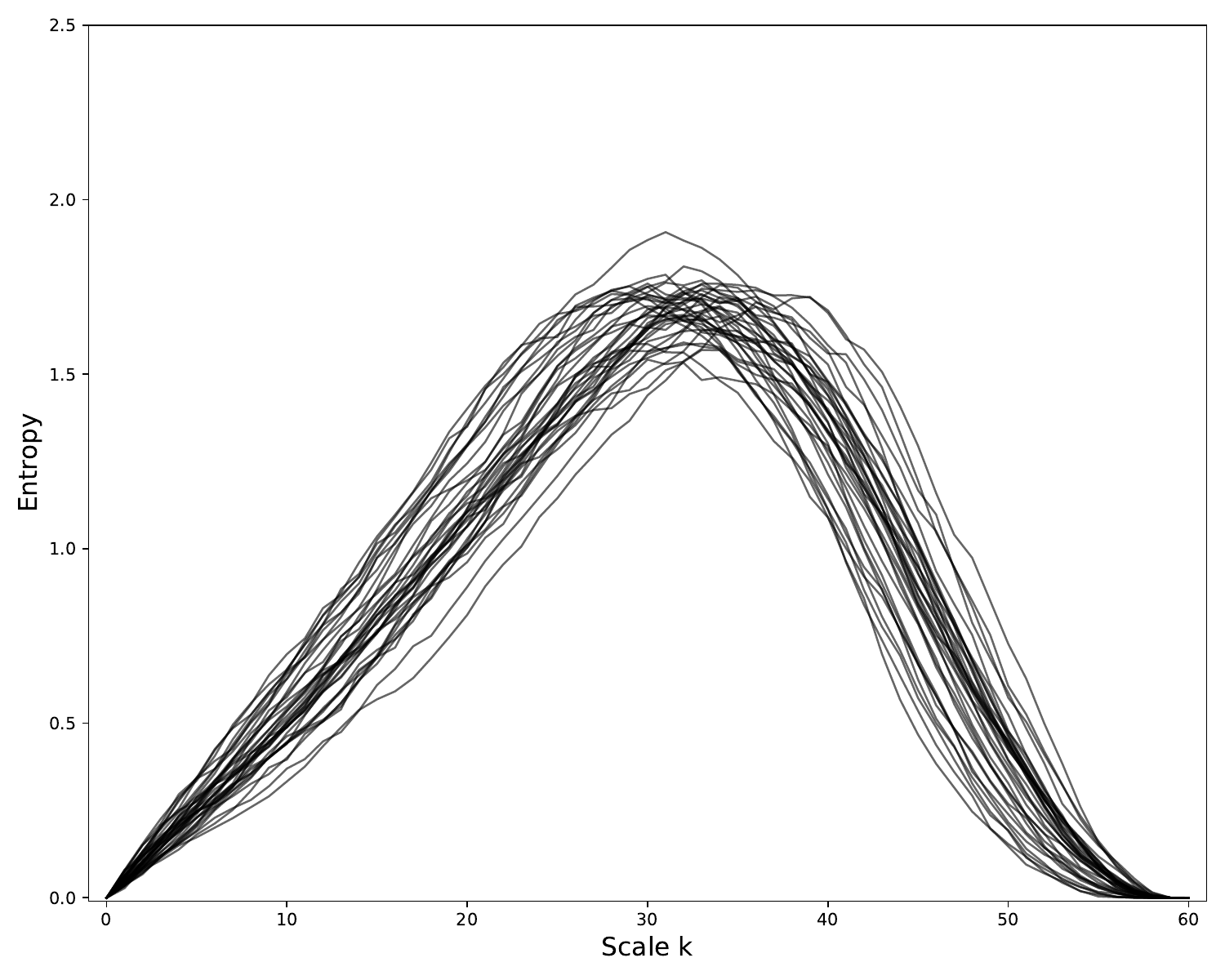}
    \caption{Individual entropy profile for each cycle.}
    \label{fig:entropies_eachcycle_gr0_1500}
  \end{subfigure}
  \hfill
  \begin{subfigure}[t]{0.48\linewidth}
    \centering
    \includegraphics[width=\linewidth]{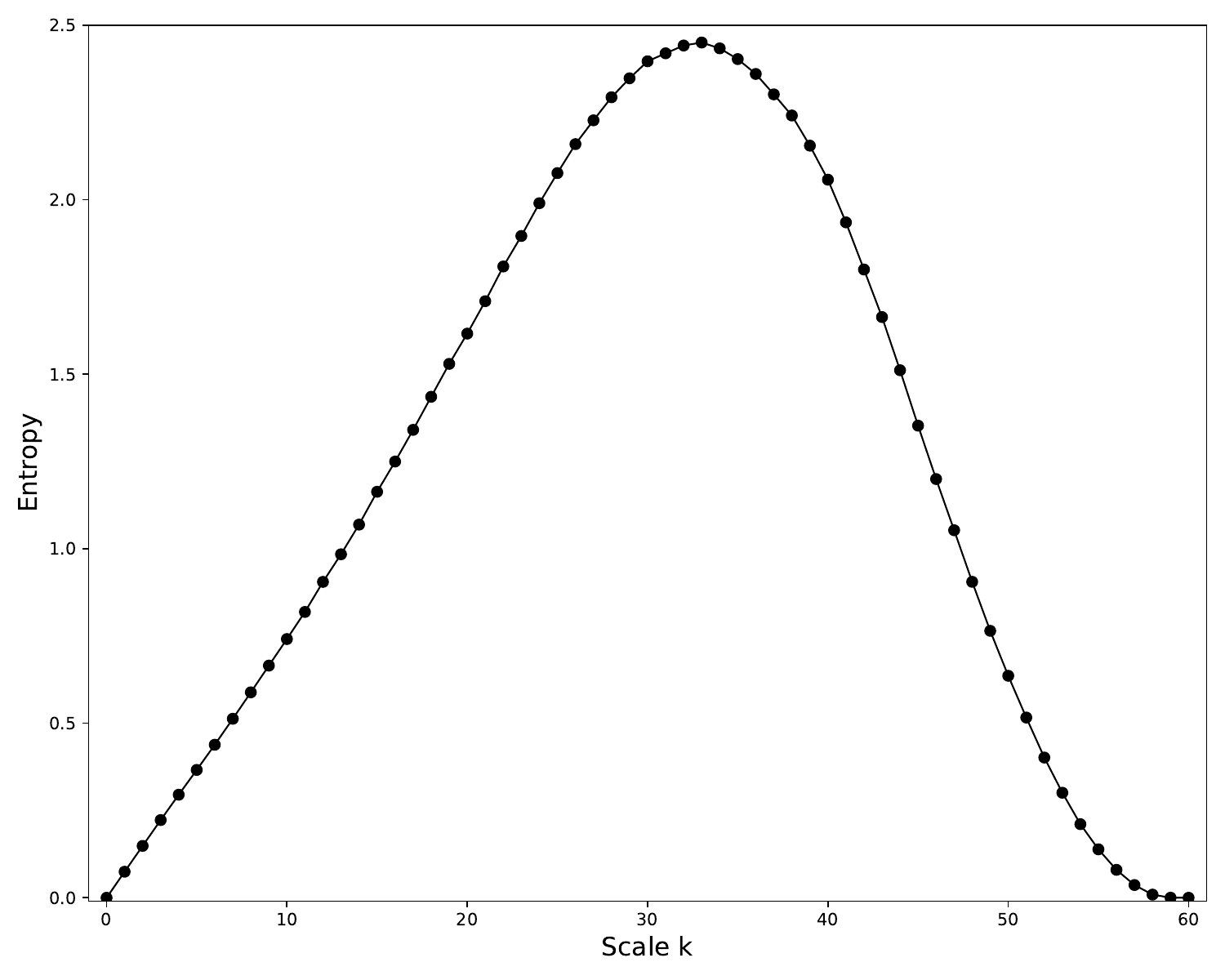}
    \caption{Aggregate entropy profile for all cycles in $\Gamma$.}
    \label{fig:entropiesallgoodcycles_a_gr0_1500}
  \end{subfigure}
  \caption{Entropy profiles for $X_{0-6,1500}$.}
  \label{fig:entropies_gr0_r1500}
\end{figure}

 \begin{figure}
\centering
\begin{subfigure}{0.24\linewidth}
    \centering
    \includegraphics[width=\linewidth]{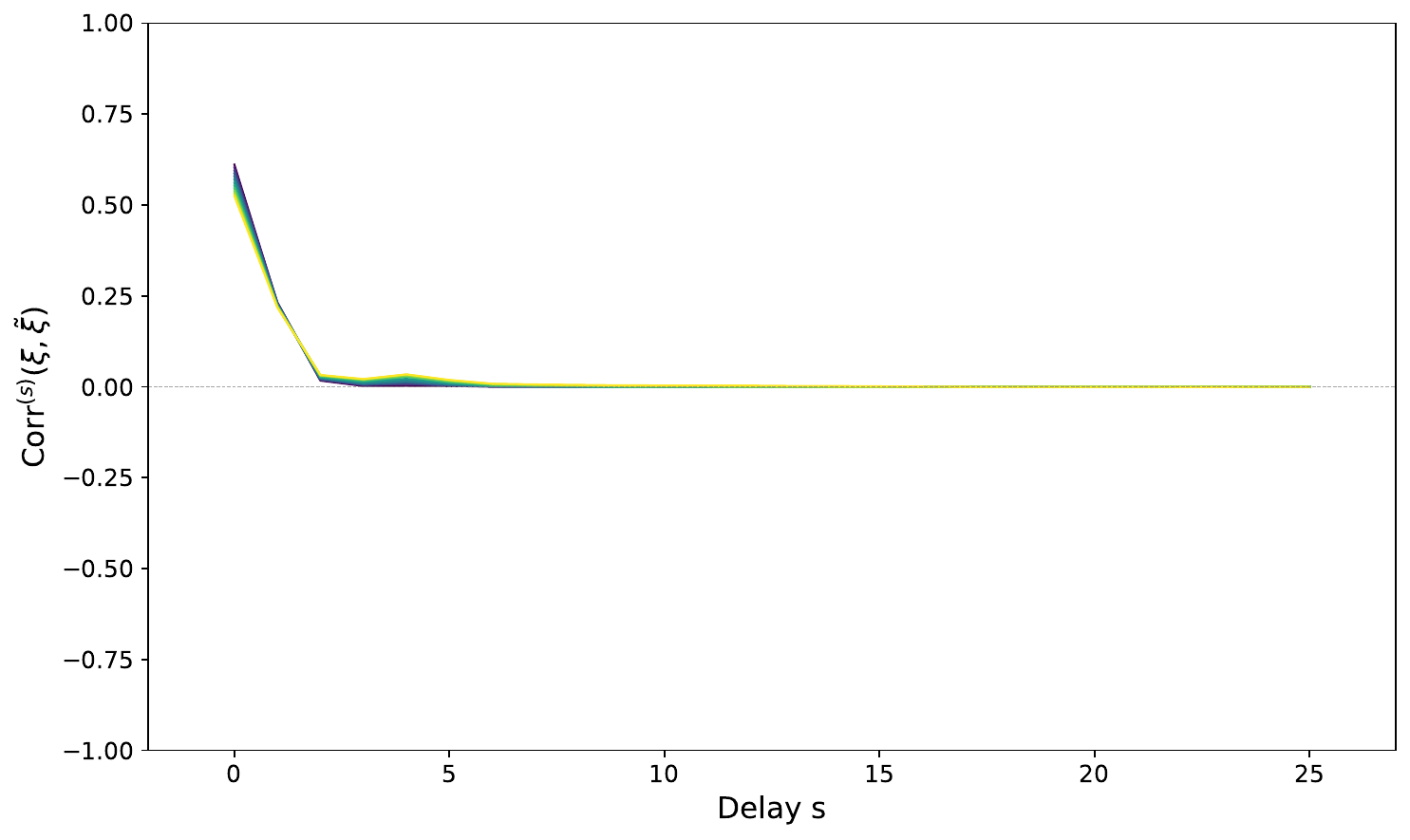}
    \caption{Position 152 and 153.}
    \label{fig:allcorr_allcycles_gr0_152_153_1500}
\end{subfigure}
\hfill
\begin{subfigure}{0.24\linewidth}
    \centering
    \includegraphics[width=\linewidth]{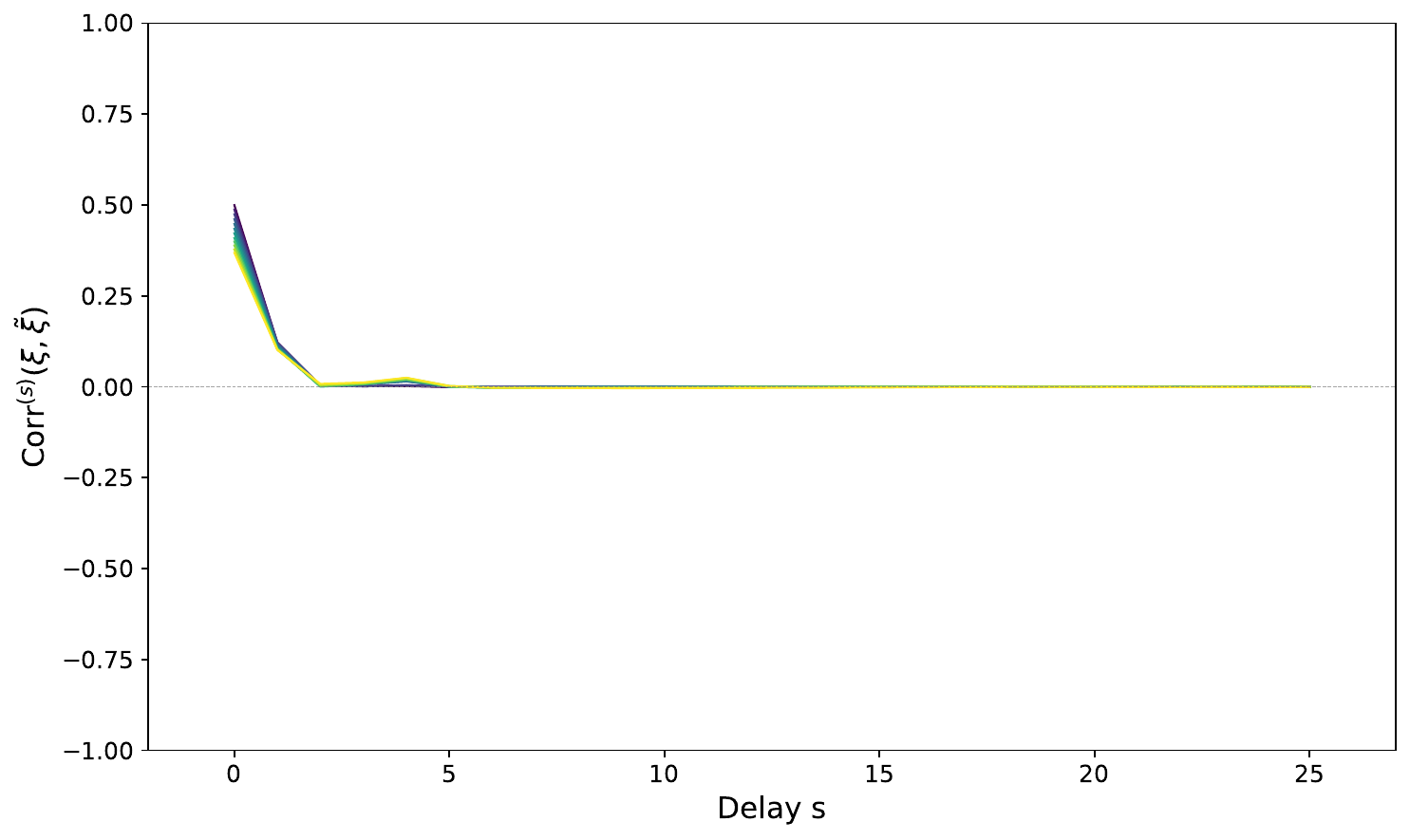}
    \caption{Position 103 and 119.}
    \label{fig:allcorr_allcycles_gr0_103_119_1500}
\end{subfigure}
\hfill
\begin{subfigure}{0.24\linewidth}
    \centering
    \includegraphics[width=\linewidth]{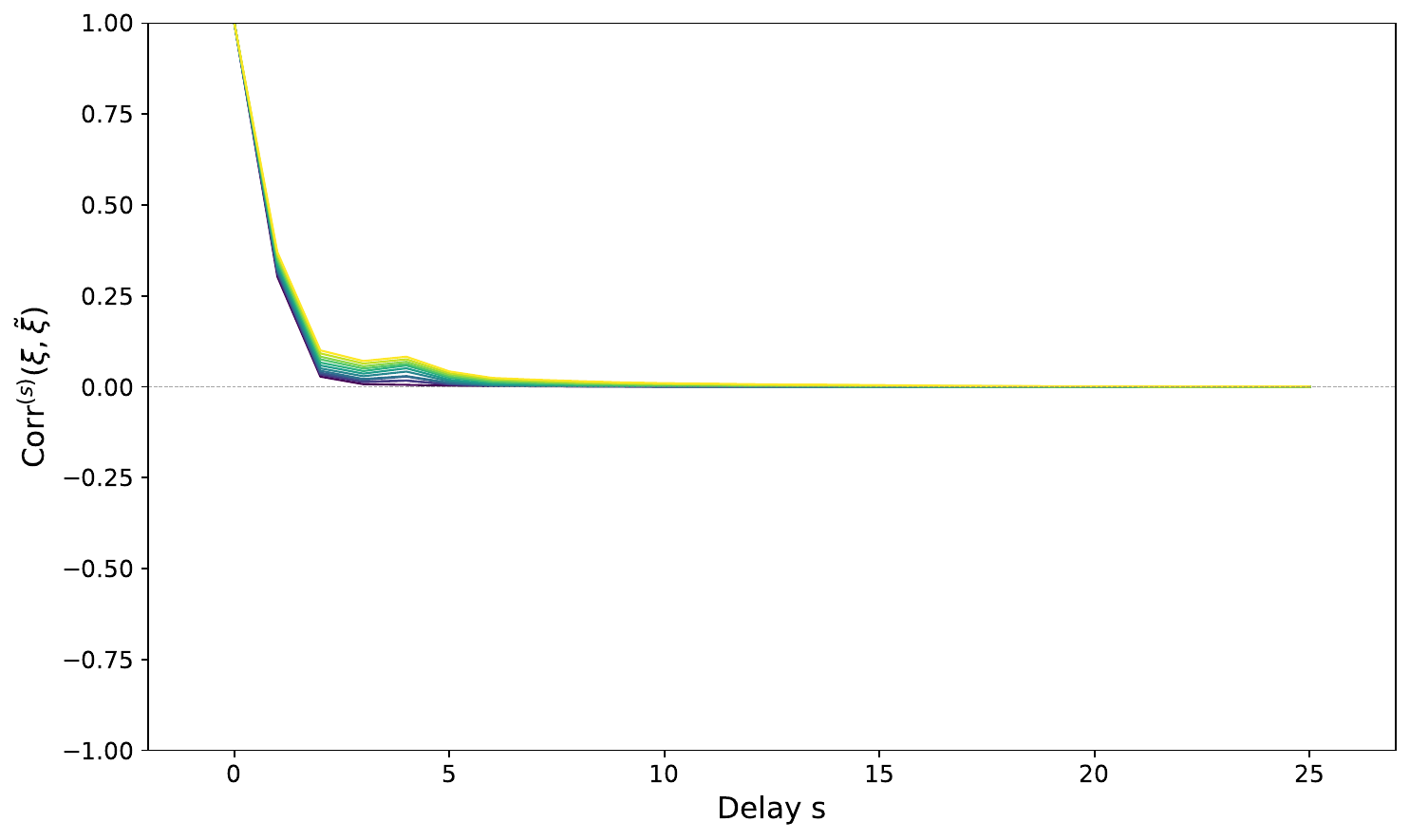}
    \caption{Position 152 and 152.}
    \label{fig:allcorr_allcycles_gr0_152_152_1500}
\end{subfigure}
\hfill
\begin{subfigure}{0.24\linewidth}
    \centering
    \includegraphics[width=\linewidth]{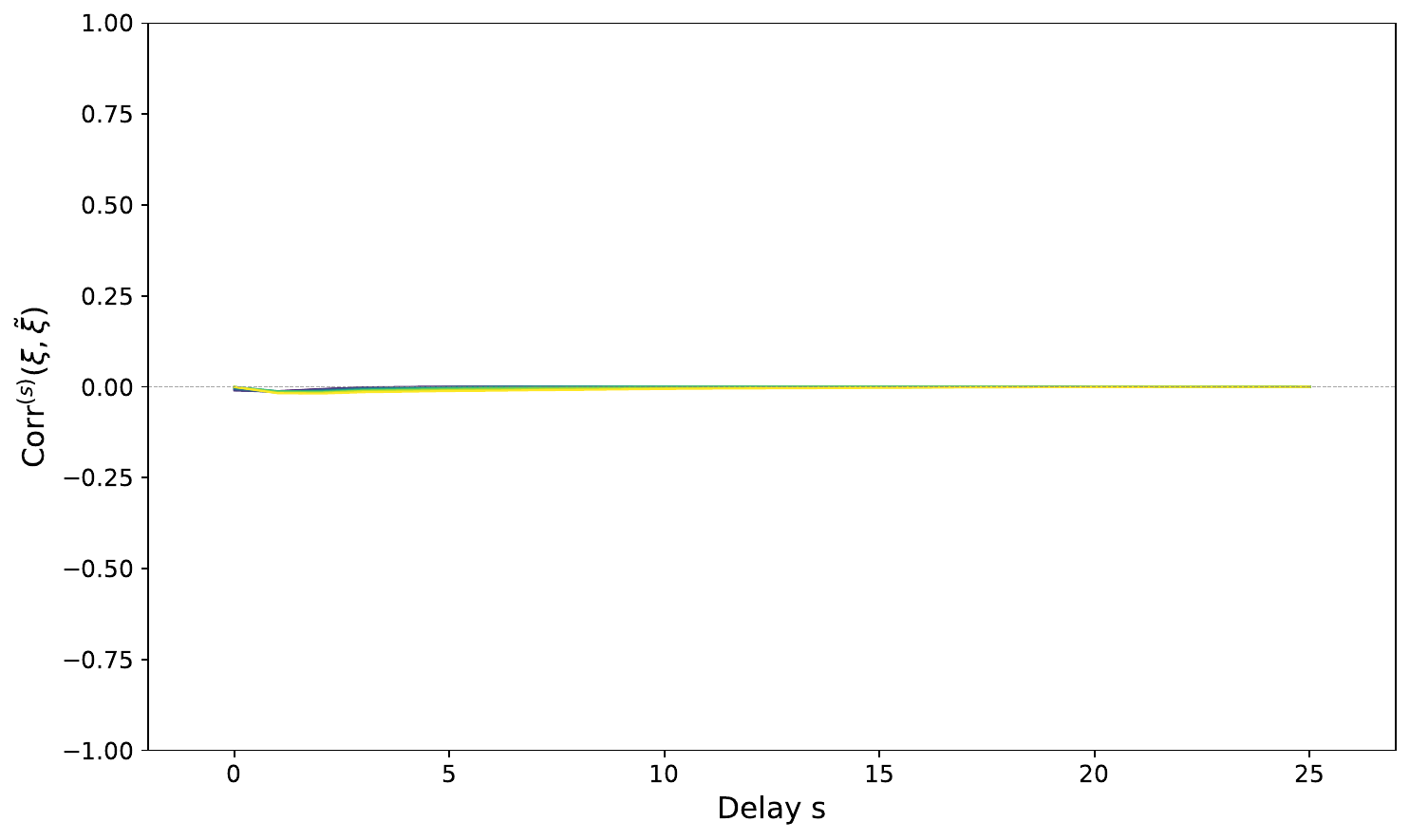}
    \caption{Position 101 and 186.}
    \label{fig:allcorr_allcycles_gr0_101_186_1500}
\end{subfigure}
\hfill
\begin{subfigure}{0.20\linewidth}
    \centering
    \raisebox{0.8cm}{\includegraphics[width=\linewidth]{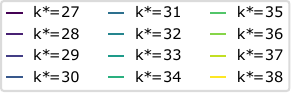}}
\end{subfigure}
\caption{Delayed correlation profiles for three representative pairs of cells for $X_{0-6,1500}$, shown for all $k^*\in S$.}
\label{fig:allcorr_1500_subfigures}
\end{figure}

\FloatBarrier
\clearpage

\begin{table}[]
\centering
\caption{Distribution of Cycle Periods over $X_{0-6,5000}$.}
\label{tab:period_orbits_gr0_5000}
\begin{tabular}{cl}
\toprule
\textbf{\# Orbits} & \textbf{Period} \\
\midrule
3 & 2175, 2183, 2197, 2198 \\
\addlinespace
2 & 2164, 2174, 2178, 2240 \\
\addlinespace
1 & 2162, 2163, 2165, 2166, 2168, 2170, 2171, 2173, 2179, \\
  & 2181, 2182, 2186, 2188, 2191, 2193, 2194, 2200, 2206, \\
  & 2207, 2208, 2211, 2213, 2216, 2217, 2219, 2226, 2231, \\
  & 2233, 2242, 2244, 2245, 2246, 2252, 2257, 2258, 2262, \\
  & 2263, 2279, 2296, 2323, 2356, 2362, 2469, 6543 \\
\bottomrule
\end{tabular}
\end{table}

 \begin{figure}
  \centering
\includegraphics[width=0.6\linewidth]{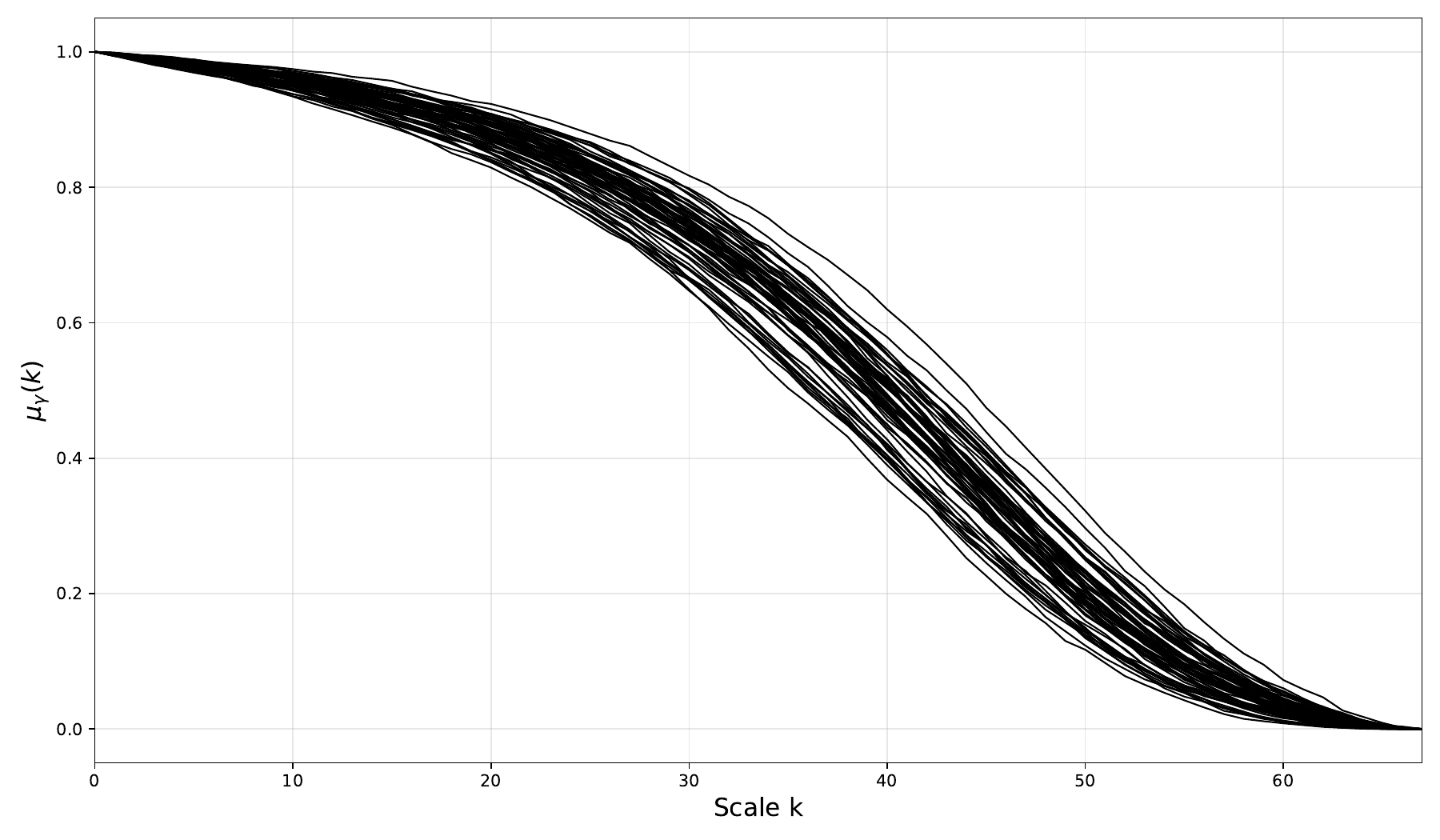}
  \caption{Validation of Criteria I and II for $X_{0-6,5000}$}
\label{fig:sgraph_gr0_5000}
\end{figure}

\begin{figure}
  \centering
  \begin{subfigure}[t]{0.48\linewidth}
    \centering
    \includegraphics[width=\linewidth]{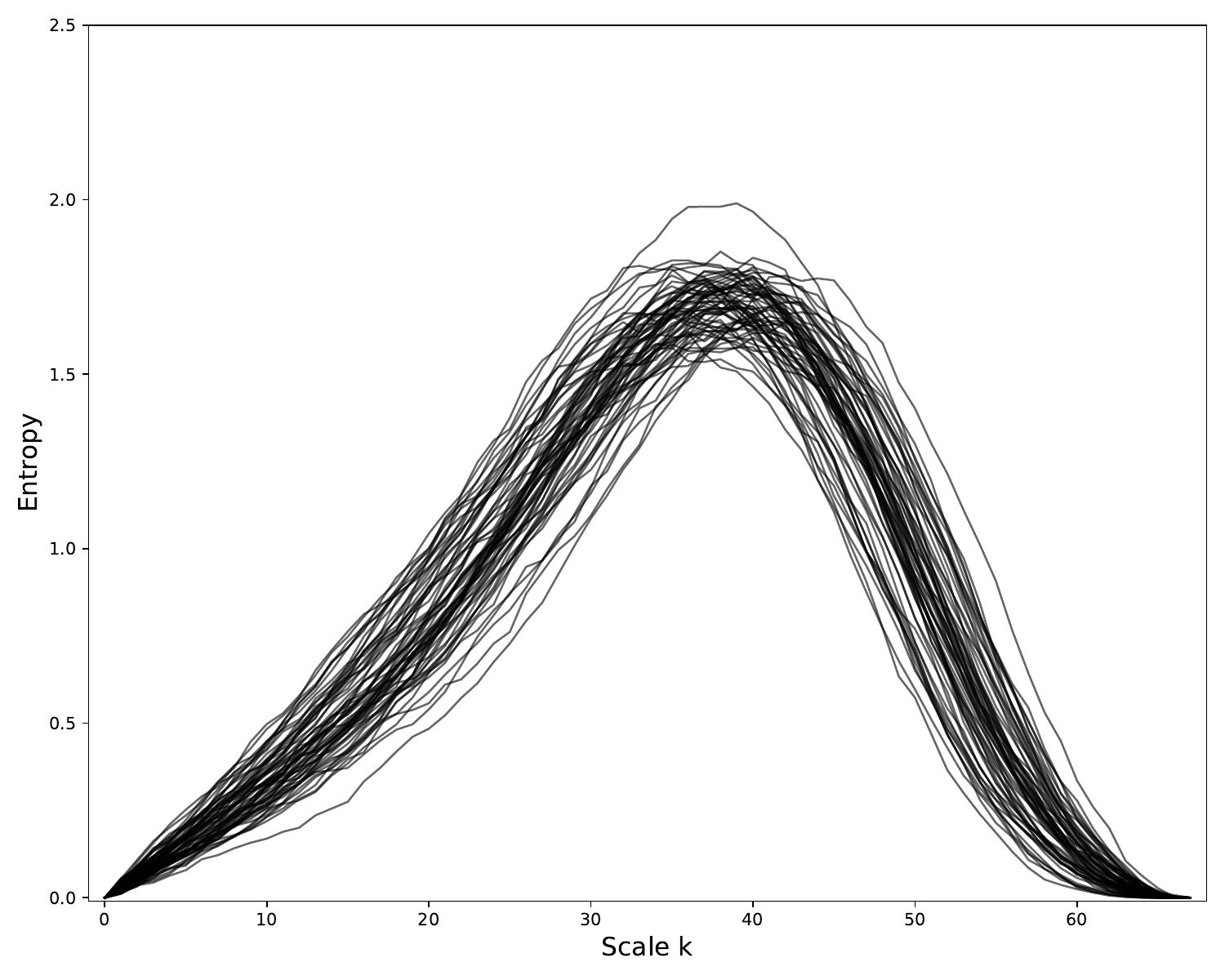}
    \caption{Individual entropy profile for each cycle.}
    \label{fig:entropies_eachcycle_gr0_5000}
  \end{subfigure}
  \hfill
  \begin{subfigure}[t]{0.48\linewidth}
    \centering
    \includegraphics[width=\linewidth]{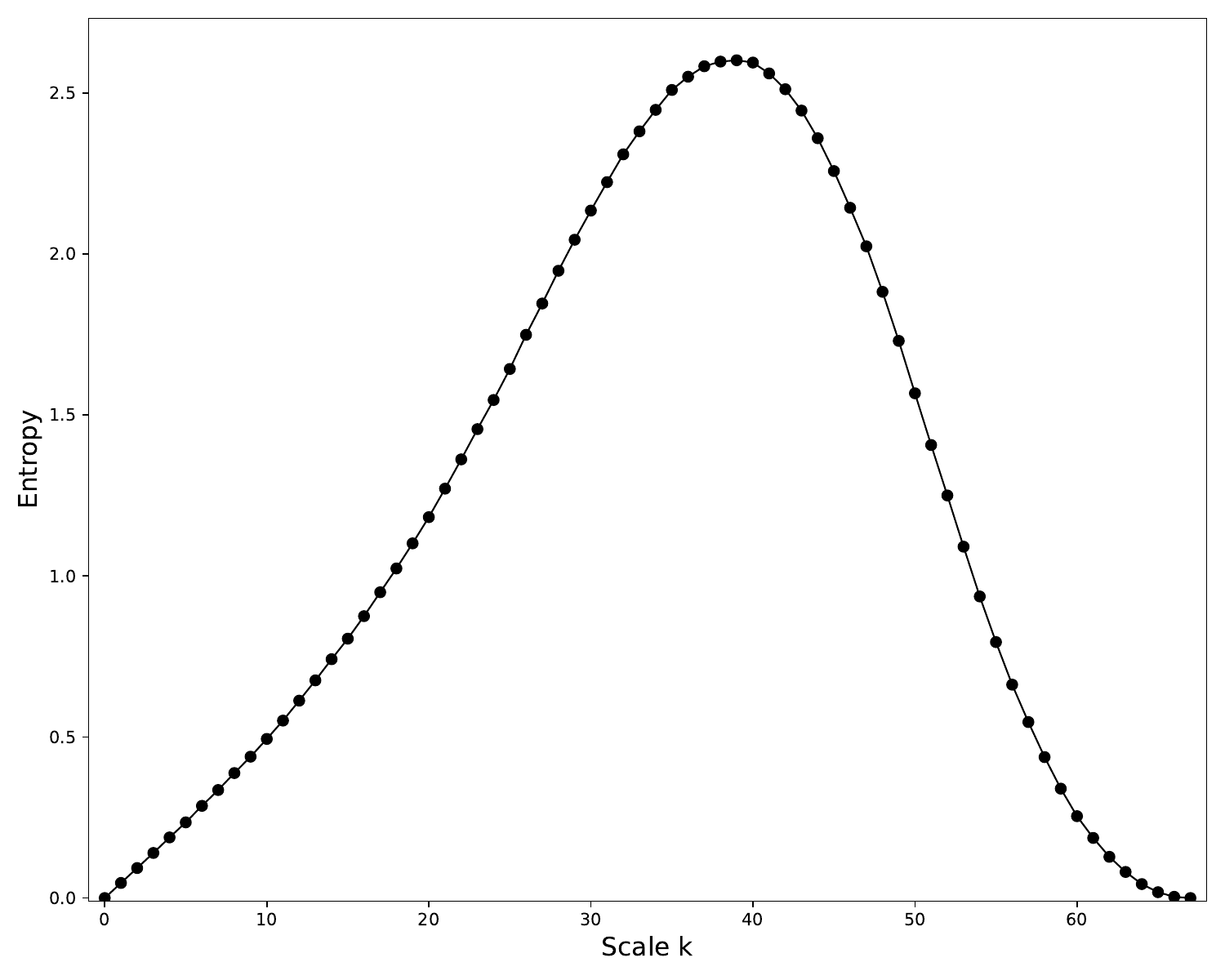}
    \caption{Aggregate entropy profile for all cycles in $\Gamma$.}
    \label{fig:entropiesallgoodcycles_a_gr0_5000}
  \end{subfigure}
  \caption{Entropy profiles for $X_{0-6,5000}$.}
  \label{fig:entropies_gr0_r5000}
\end{figure}


\begin{figure}[ht]
\centering
\begin{subfigure}{0.24\linewidth}
    \centering
    \includegraphics[width=\linewidth]{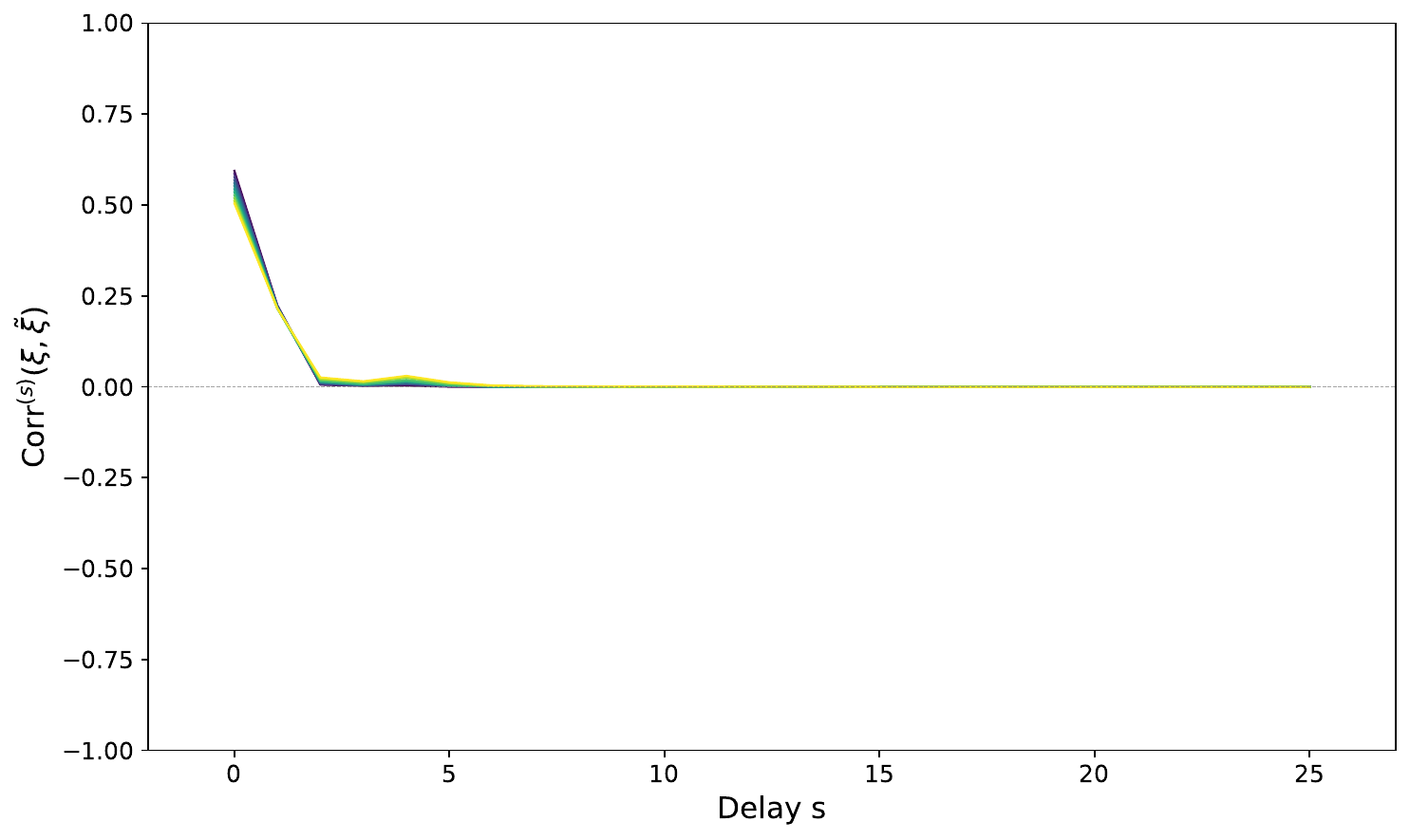}
    \caption{Position 152 and 153.}
    \label{fig:allcorr_allcycles_gr0_152_153_5000}
\end{subfigure}
\hfill
\begin{subfigure}{0.24\linewidth}
    \centering
    \includegraphics[width=\linewidth]{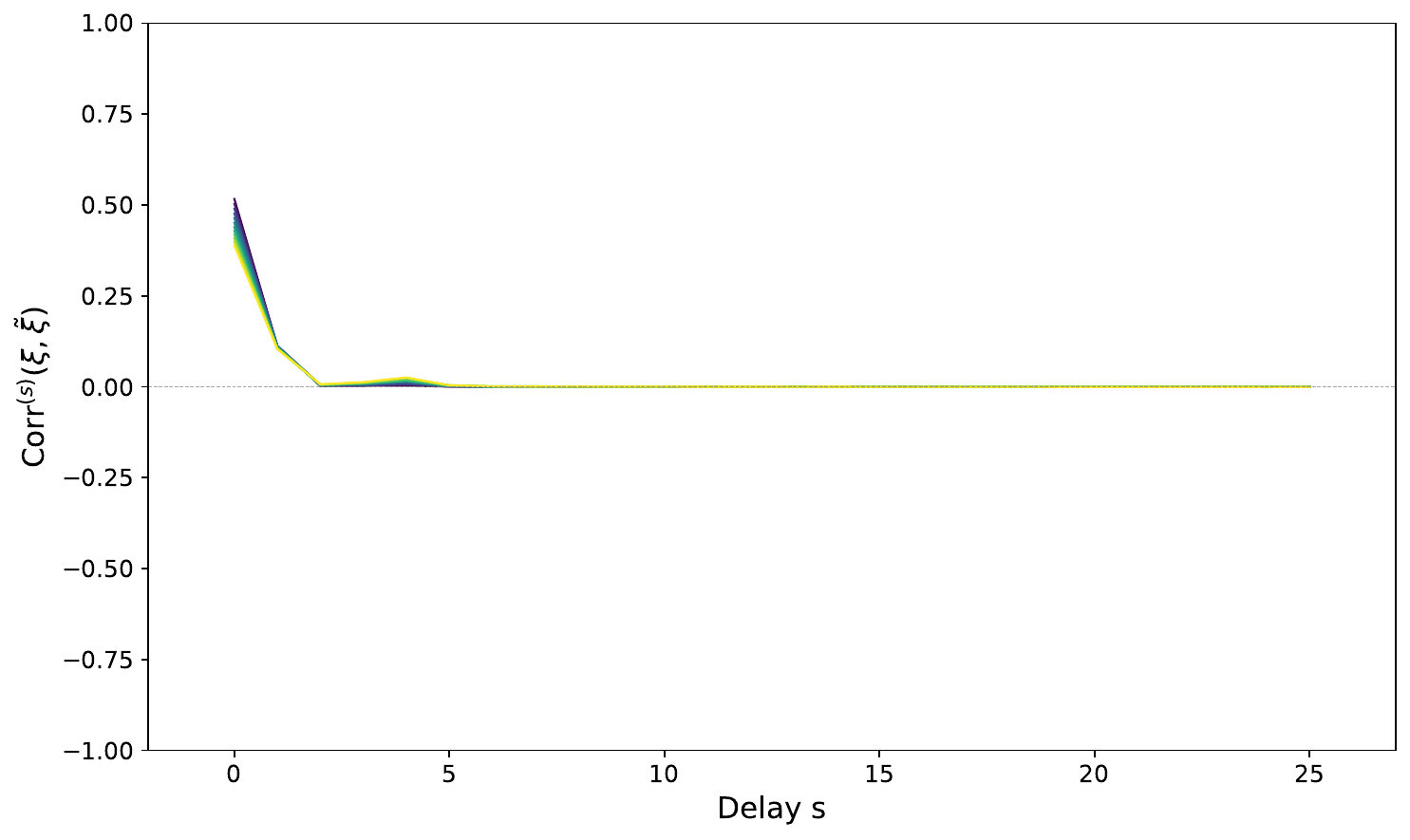}
    \caption{Position 103 and 119.}
    \label{fig:allcorr_allcycles_gr0_103_119_5000}
\end{subfigure}
\hfill
\begin{subfigure}{0.24\linewidth}
    \centering
    \includegraphics[width=\linewidth]{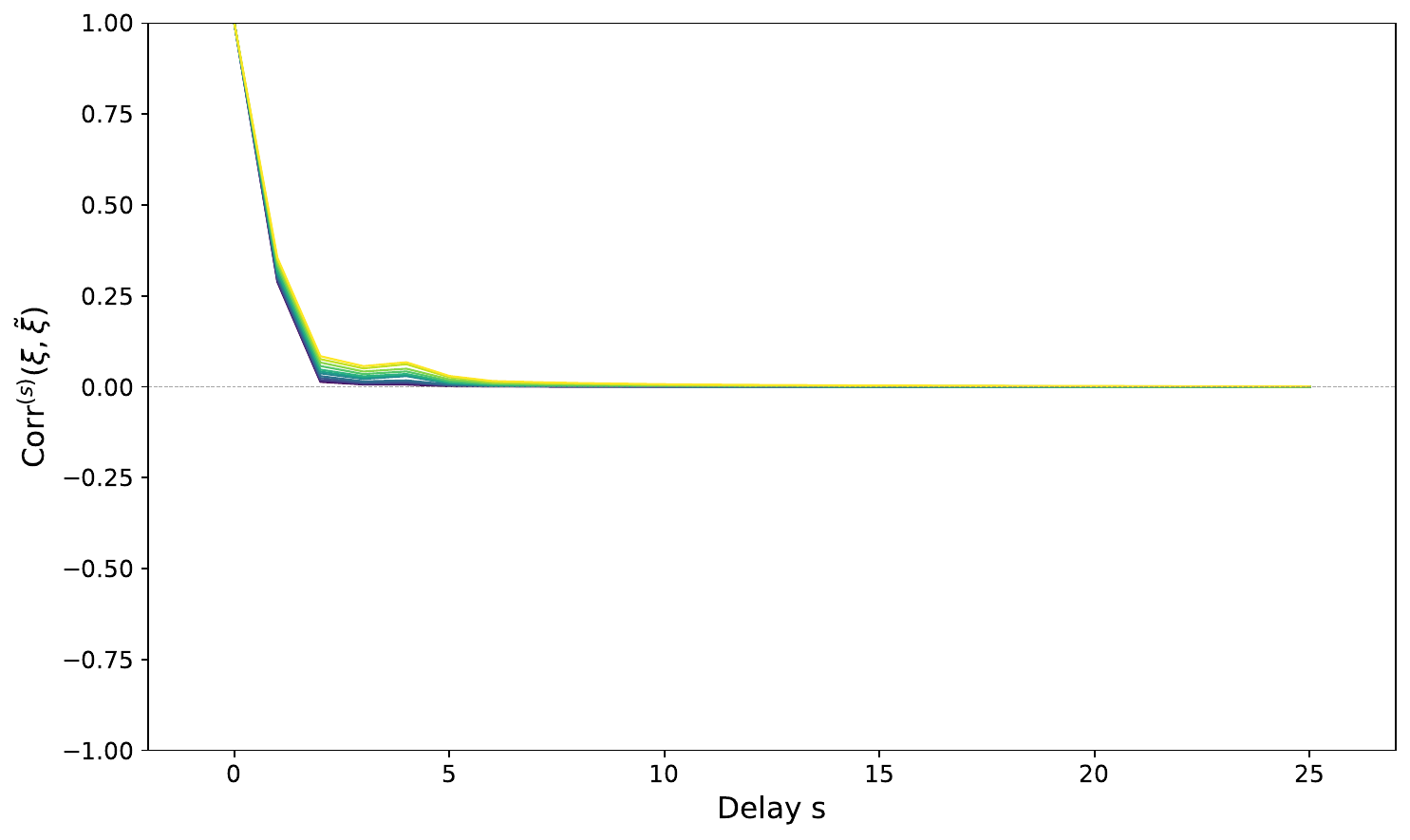}
    \caption{Position 152 and 152.}
    \label{fig:allcorr_allcycles_gr0_152_152_5000}
\end{subfigure}
\hfill
\begin{subfigure}{0.24\linewidth}
    \centering
    \includegraphics[width=\linewidth]{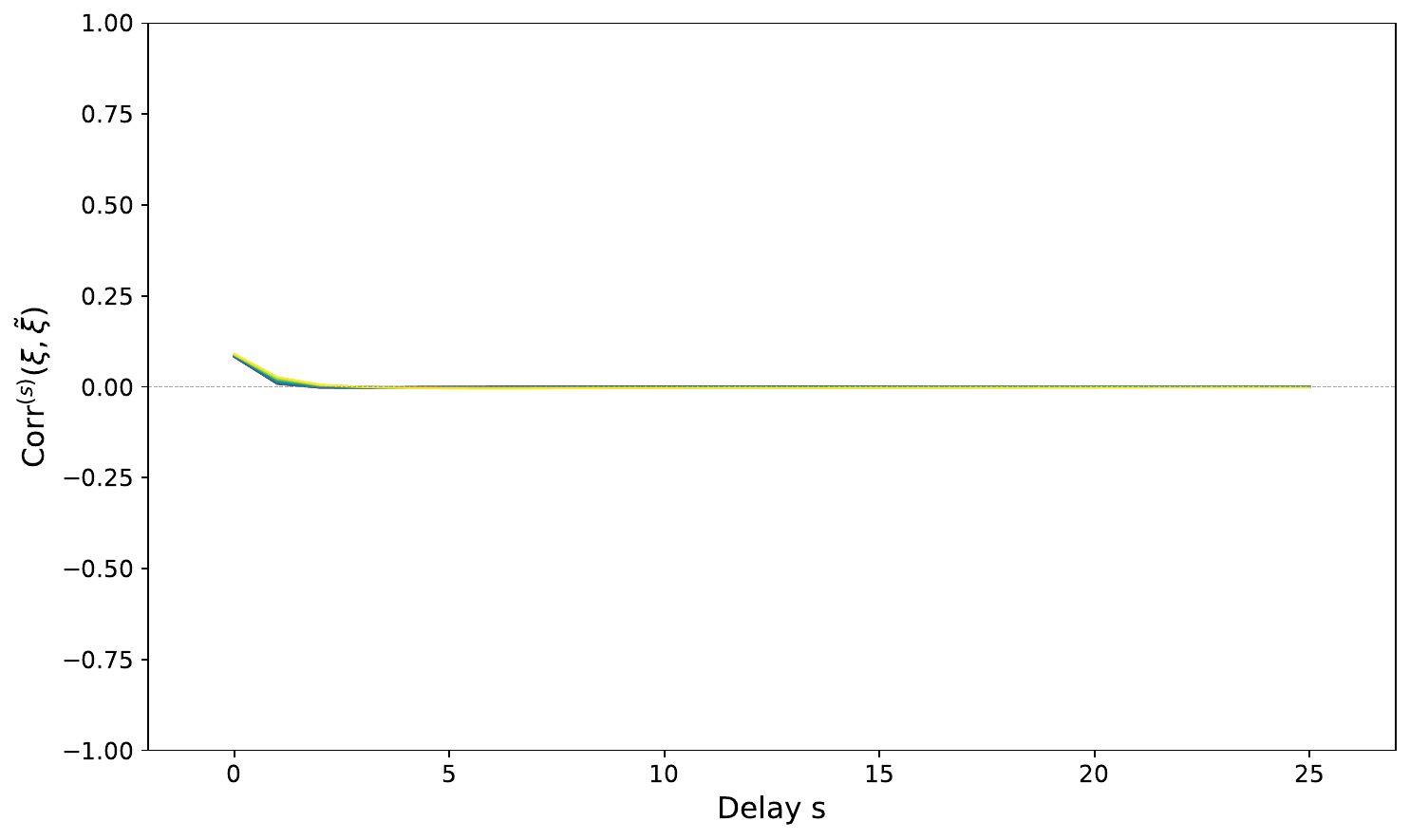}
    \caption{Position 101 and 186.}
    \label{fig:allcorr_allcycles_gr0_101_186_5000}
\end{subfigure}
\hfill
\begin{subfigure}{0.2\linewidth}
    \centering
    \raisebox{0.8cm}{\includegraphics[width=\linewidth]{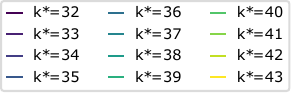}}
\end{subfigure}
\caption{Delayed correlation profiles for three representative pairs of cells for $X_{0-6,5000}$, shown for all $k^*\in S$.}
\label{fig:allcorr_5000_subfigures}
\end{figure}

\FloatBarrier
\clearpage
\appendix
\section*{Appendix II}
In this appendix we provide sketches of proof for various consequences that comes from hyperbolicity, see \cite{BrinGarrett, KatokHasselblatt, Mane, PalisTakens, Petersen} for more details on this topic.

\subsection*{Hyperbolicity Consequence~I} 
For every hyperbolic system there exists a sequence of nested renormalization boxes $U_k$ with $0\leq k< m$ such that the system is equivalent to the corresponding sub-shift $\Sigma_k$ for $k$ large enough. Moreover, given a finite selection of renormalization boxes $U_k$ with $0\leq k< m$, there is a collection of cycles which intersects all renormalization boxes $U_k$ with $k<m$.
\begin{proof}[Idea of the Proof]
Choose any sequence of nested Markov boxes. The corresponding sub-shift of finite type will be equivalent to the whole system.
\end{proof}

\subsection*{Hyperbolicity Consequence~II} 
A hyperbolic system with a physical measure $\mu^*$ has the property that for every $\delta>0$ and for any finite selection of renormalization boxes $U_k$ with $0\leq k\leq m-1$, there exists a non empty finite collection of cycles $\Gamma^*=\{\gamma\}$ such that 
\[
\bigl|\mu_{\gamma}(U_k)-\mu^*(U_k)\bigr| < \delta
\quad\text{holds for every $\gamma\in\Gamma^*$ and $0\leq k\leq m-1$.}
\]

\begin{proof}[Idea of the Proof]
Every invariant measure of a sub-shift of finite type can be approximated by the invariant measure supported on a periodic orbit. 
\end{proof}

\subsection*{Hyperbolicity Consequence~III} 
Given a hyperbolic system with physical measure $\mu^*$ which has entropy $h^*$. Then for every $\delta>0$ there exists a non empty finite collection of cycles $\Gamma^*=\{\gamma\}$ such that for every $\gamma\in\Gamma^*$, the corresponding Markov chains $\bigl(\Sigma_k(\gamma),\mu_k(\gamma), \sigma_k(\gamma)\bigr)$, $k=0,\dots,m$, have entropy function $h_{\gamma}(k):=h_k(\gamma)$ satisfying the following:
\begin{itemize}
\item[-] $h_{\gamma}(0)=0$ and $h_{\gamma}(m)=0$,
\item[-] there exists an interval $K_{\gamma} := \bigl[k(0),k(1)\bigr]$ such that $h_{\gamma}$ is increasing for all $k\leq k(0)$ and it is decreasing for all $k\geq k(1)$,
\item[-] $\bigl|h_{\gamma}(k)-h^*\bigr|<\delta$, for all $k\in K_{\gamma}$,
\item[-] $\bigcap_{\gamma\in\Gamma^*}K_{\gamma}\neq\emptyset$.
\end{itemize}  

\begin{proof}[Idea of the Proof]
Choose a Markov partition. Then use the renormalization scheme to create a nested sequence of finer and finer Markov partitions $\mathcal M_k$. The corresponding sub-shift of finite type are described by the graphs $G_k$. In particular, they have the same structure as the ones used in our method, except that they have countably many loops. Because of hyperbolicity, the diameter of the Markov boxes in these consecutive Markov partitions shrink to zero. 

For every Markov partition the physical measure assigns a weight to each of its Markov boxes. These weights define a Markov measure on the corresponding sub-shift of finite type, $\left(\Sigma_k,\mu_k, \sigma_k\right)$ with entropy $h(k)$, $k=0,1,\dots$. The entropy of these Markov chains converge to $h^*$, i.e.
\[
\lim_{k\to\infty}h(k)=h^*.
\]
Choose a $k(0)$ such that $\left|h(k(0))-h^*\right|<\delta/2$.

Choose a collection $\Gamma=\left\{\gamma\right\}$ such that the cycle measure $\mu_{\gamma}$ of each $\gamma\in\Gamma$, assigns approximately the same weight to the Markov boxes of the Markov partition $k(0)$ as the physical measure does. Hence, the cycle $\gamma$ defines a Markov measure on $G_{k(0)}$ whose entropy is approximately $h_{k(0)}$. In particular,
\[ 
  \left|h_{\gamma}(k(0))-h^*\right|<\delta
  \quad\text{for each $\gamma\in\Gamma$.}
\]
This motivates that point three and four hold. 

Next, observe that, by construction, for $k=0$, the graph $G_0$ has only one vertex and one return branch. Hence $h_{\gamma}(0)=0$. Moreover, given a $\gamma\in\Gamma$ for $k=m$ large enough, the return box $U_m$ contains only one point of $\gamma$ and the cycle measure of $\gamma$ is supported on only one loop of the graph $G_m$. Hence $h_{\gamma}(m)=0$. This motivates the first point. 

To give an idea of the proof of the second point, we have to consider three different regimes. The first one is when  $k< k(0)$. Observe that the cycle measure $\mu_{\gamma}$ assigns approximately the same weight to the Markov boxes of each Markov partition $\mathcal M_k$ with $k\leq k(1)$ for a certain $k(1)>k(0)$. The cycle measure also defines a Markov measure on each $G_k$, $\mu_k(\gamma)$. Observe that $\mu_k(\gamma)(U_k)$ is monotonically decreasing and the number of loops in $G_k$ is increasing for $k<k(0)$. By increasing $k$ the measure moves more and more weight from the root $U_k$ into the loops of the graph $G_k$ while $\mu_k(\gamma)$ restricted to the Markov boxes of $\mathcal{M}_k$ is still close to the induced Markov measure $\mu_k$ coming from the original physical measure. Hence, the entropy increases. 

In the intermediate regime, the entropy $h_k$ for $k(0)\leq k\leq k(1)$ is approximately $h^*$ the entropy of the physical measure. Because the Markov measures induced by the cycle measure and by the physical measure is approximately the same. 

In the third and last regime, for $k$ large, the cycle does not have enough points to approximate well the Markov measure induced by the initial physical measure. By increasing $k$ the support of the $\mu_k(\gamma)$ is a monotonically decreasing collection of loops of the graph $G_k$. In particular, the entropy will be decreasing. 
\end{proof}

\subsection*{Hyperbolicity Consequence~IV} 
Given a hyperbolic system with physical measure $\mu^*$. Then for every $\delta>0$ and $p\in\Natural$ there exists a non empty finite collection of cycles $\Gamma^*=\left\{{\gamma}\right\}$ such that 
\[ |\mu_{\gamma}-\mu^*|_p < \delta
\quad\text{for every $\gamma\in\Gamma^*$.}
\] 

\begin{proof}[Idea of the Proof]
The use of the seminorm to describe the quality of the approximation implies that we only need to approximate the measure on a finite cover of state space. Choose a cover that is fine enough consisting of Markov boxes. This cover defines a sub-shift of finite type. A general fact about a sub-shifts of finite type is that every invariant measure restricted to the algebra generated by the cover of Markov boxes can be arbitrarily precisely approximated by cycle measures. 
\end{proof}

\subsection*{Hyperbolicity robustness consequence} 
Given a hyperbolic system with physical measure $\mu^*$ and entropy $h^*$. Then for every $\delta>0$, there are data characteristics such that a typical data set with this characteristics creates a model $\left(\Sigma_{k^*},\mu_{k^*}, \sigma_{k^*}\right)$ whose entropy $h_{k^*}$ satisfies,
\[ |h_{k^*}-h^*|< \delta.
\] 
 
\begin{proof}[Idea of the Proof]
Take a small enough Markov box. The corresponding renormalization, i.e. the first return map, creates a dynamical partition with the corresponding renormalization graph $G$. The physical measure $\mu^*$ induces a stationary Markov measure on $G$. In particular, the entropy of  $\mu^*$ is arbitrarily close to the entropy of this induced stationary Markov measure. For every $\delta$, there exists a periodic orbit $\gamma$ whose invariant measure $\mu_{\gamma}$ restricted to the pieces of the dynamical partition is arbitrarily close to the corresponding restriction of the  measure $\mu^*$. In particular $\mu_{\gamma}$ and $\mu^*$ induce arbitrarily close Markov measures in $G$. Hence,  the corresponding entropies are approximately the same and close to $h^*$. If typical data consists of typical orbits that are long enough, then the induced measure on the dynamical partition will be arbitrarily close to the restriction of the measure $\mu_{\gamma}$ and hence the corresponding model will still have entropy close to $h^*$. 
\end{proof}
\section*{Acknowledgments} 
M.M. was partially supported by Knut and Alice Wallenberg Foundation, L.P. was partially supported by the VR grant 67578, Renormalization in dynamics, A.P. and O.Ö. acknowledges support from the Wallenberg Autonomous Systems and Software Program (WASP), funded by the Knut and Alice Wallenberg Foundation. All simulations were performed on the Dardel HPC system. The computations were enabled by resources provided by the National Academic Infrastructure for Supercomputing in Sweden (NAISS). 

\FloatBarrier
\clearpage
\printbibliography

@inproceedings{KacprzykVanDerSchaar2025,
  author    = {Kacprzyk, Krzysztof and van der Schaar, Mihaela},
  title     = {No Equations Needed: Learning System Dynamics Without
               Relying on Closed-Form {ODE}s},
  booktitle = {Proceedings of the 13th International Conference on
               Learning Representations (ICLR 2025)},
  year      = {2025},
  doi       = {10.48550/arXiv.2501.18563},
  note      = {\url{https://openreview.net/forum?id=kbm6tsICar}}
}

@book{KaczynskiMischaikowMrozek2004,
  author    = {Kaczynski, Tomasz and Mischaikow, Konstantin and Mrozek, Marian},
  title     = {Computational Homology},
  series    = {Applied Mathematical Sciences},
  volume    = {157},
  publisher = {Springer, New York},
  year      = {2004},
  doi       = {10.1007/b97315}
}

@misc{GameiroGelbMischaikow2025,
  author = {Gameiro, Marcio and Gelb, Brittany and Mischaikow, Konstantin},
  title  = {Rigorously Characterizing Dynamics with Machine Learning},
  year   = {2025},
  note   = {Preprint, arXiv:2505.17302},
  doi    = {10.48550/arXiv.2505.17302}
}

@article{WilliamsKevrekidisRowley2015,
  author  = {Williams, Matthew O. and Kevrekidis, Ioannis G. and Rowley, Clarence W.},
  title   = {A Data-Driven Approximation of the {K}oopman Operator:
             Extending Dynamic Mode Decomposition},
  journal = {Journal of Nonlinear Science},
  volume  = {25},
  number  = {6},
  pages   = {1307--1346},
  year    = {2015},
  doi     = {10.1007/s00332-015-9258-5}
}

@article{Gilpin2024,
  author  = {Gilpin, William},
  title   = {Generative Learning for Nonlinear Dynamics},
  journal = {Nature Reviews Physics},
  volume  = {6},
  number  = {3},
  pages   = {194--206},
  year    = {2024},
  doi     = {10.1038/s42254-024-00688-2}
}

@inproceedings{Schiff2024,
  author    = {Schiff, Yair and Wan, Zhong Yi and Parker, Jeffrey B. and
               Hoyer, Stephan and Kuleshov, Volodymyr and Sha, Fei and
               Zepeda-N{\'u}{\~n}ez, Leonardo},
  title     = {{D}y{SLIM}: Dynamics Stable Learning by Invariant Measure
               for Chaotic Systems},
  booktitle = {Proceedings of the 41st International Conference on
               Machine Learning (ICML 2024)},
  year      = {2024},
  doi       = {10.48550/arXiv.2402.04467}
}

@article{BotvinickGreenhouse2023,
  author  = {Botvinick-Greenhouse, Jonah and Martin, Robert and Yang, Yunan},
  title   = {Learning Dynamics on Invariant Measures Using
             {PDE}-Constrained Optimization},
  journal = {Chaos: An Interdisciplinary Journal of Nonlinear Science},
  volume  = {33},
  number  = {6},
  pages   = {063152},
  year    = {2023},
  doi     = {10.1063/5.0149673}
}

@article{OttoRowley2019,
  author  = {Otto, Samuel E. and Rowley, Clarence W.},
  title   = {Linearly Recurrent Autoencoder Networks for Learning Dynamics},
  journal = {SIAM Journal on Applied Dynamical Systems},
  volume  = {18},
  number  = {1},
  pages   = {558--593},
  year    = {2019},
  doi     = {10.1137/18M1177846}
}

@article{Lusch2018,
  author  = {Lusch, Bethany and Kutz, J. Nathan and Brunton, Steven L.},
  title   = {Deep Learning for Universal Linear Embeddings of Nonlinear Dynamics},
  journal = {Nature Communications},
  volume  = {9},
  pages   = {4950},
  year    = {2018},
  doi     = {10.1038/s41467-018-07210-0}
}

@article{MischaikowActa2002,
  author  = {Mischaikow, Konstantin},
  title   = {Topological Techniques for Efficient Rigorous Computation
             in Dynamics},
  journal = {Acta Numerica},
  volume  = {11},
  pages   = {435--477},
  year    = {2002},
  doi     = {10.1017/S0962492902000065}
}

@incollection{DellnitzJunge2002,
  author    = {Dellnitz, Michael and Junge, Oliver},
  title     = {Set Oriented Numerical Methods for Dynamical Systems},
  booktitle = {Handbook of Dynamical Systems},
  editor    = {Fiedler, Bernold},
  volume    = {2},
  pages     = {221--264},
  publisher = {North-Holland, Amsterdam},
  year      = {2002},
  doi       = {10.1016/S1874-575X(02)80026-1}
}

@incollection{MischaikowMrozek2002,
  author    = {Mischaikow, Konstantin and Mrozek, Marian},
  title     = {Conley Index},
  booktitle = {Handbook of Dynamical Systems},
  editor    = {Fiedler, Bernold},
  volume    = {2},
  pages     = {393--460},
  publisher = {North-Holland, Amsterdam},
  year      = {2002},
  doi       = {10.1016/S1874-575X(02)80030-3}
}

@article{Batko2020,
  author  = {Batko, Bogdan and Mischaikow, Konstantin and Mrozek, Marian and
             Przybylski, Mateusz},
  title   = {Conley Index Approach to Sampled Dynamics},
  journal = {SIAM Journal on Applied Dynamical Systems},
  volume  = {19},
  number  = {1},
  pages   = {665--704},
  year    = {2020},
  doi     = {10.1137/19M1254404}
}

@article{Mischaikow1999,
  author  = {Mischaikow, Konstantin and Mrozek, Marian and Reiss, James and
             Szymczak, Andrzej},
  title   = {Construction of Symbolic Dynamics from Experimental Time Series},
  journal = {Physical Review Letters},
  volume  = {82},
  number  = {6},
  pages   = {1144--1147},
  year    = {1999},
  doi     = {10.1103/PhysRevLett.82.1144}
}

@article{ArbabiMezic2017,
  author  = {Arbabi, Hassan and Mezi{\'c}, Igor},
  title   = {Ergodic Theory, Dynamic Mode Decomposition, and Computation of
             Spectral Properties of the {K}oopman Operator},
  journal = {SIAM Journal on Applied Dynamical Systems},
  volume  = {16},
  number  = {4},
  pages   = {2096--2126},
  year    = {2017},
  doi     = {10.1137/17M1125236}
}

@article{ColbrookTownsend2024,
  author  = {Colbrook, Matthew J. and Townsend, Alex},
  title   = {Rigorous Data-Driven Computation of Spectral Properties of
             {K}oopman Operators for Dynamical Systems},
  journal = {Communications on Pure and Applied Mathematics},
  volume  = {77},
  number  = {1},
  pages   = {221--283},
  year    = {2024},
  doi     = {10.1002/cpa.22125}
}

@article{BudisicMohrMezic2012,
  author  = {Budi{\v{s}}i{\'c}, Marko and Mohr, Ryan and Mezi{\'c}, Igor},
  title   = {Applied {K}oopmanism},
  journal = {Chaos: An Interdisciplinary Journal of Nonlinear Science},
  volume  = {22},
  number  = {4},
  pages   = {047510},
  year    = {2012},
  doi     = {10.1063/1.4772195}
}

@article{KordaMezic2018,
  author  = {Korda, Milan and Mezi{\'c}, Igor},
  title   = {On Convergence of Extended Dynamic Mode Decomposition
             to the {K}oopman Operator},
  journal = {Journal of Nonlinear Science},
  volume  = {28},
  number  = {2},
  pages   = {687--710},
  year    = {2018},
  doi     = {10.1007/s00332-017-9423-0}
}

@article{SauerYorkeCasdagli1991,
  author  = {Sauer, Tim and Yorke, James A. and Casdagli, Martin},
  title   = {Embedology},
  journal = {Journal of Statistical Physics},
  volume  = {65},
  number  = {3--4},
  pages   = {579--616},
  year    = {1991},
  doi     = {10.1007/BF01053745}
}

@incollection{Takens1981,
  author    = {Takens, Floris},
  title     = {Detecting Strange Attractors in Turbulence},
  booktitle = {Dynamical Systems and Turbulence, Warwick 1980},
  editor    = {Rand, David A. and Young, Lai-Sang},
  series    = {Lecture Notes in Mathematics},
  volume    = {898},
  pages     = {366--381},
  publisher = {Springer, Berlin},
  year      = {1981},
  doi       = {10.1007/BFb0091924}
}

@article{BerryDas2025,
  author  = {Berry, Tyrus and Das, Suddhasattwa},
  title   = {Limits of Learning Dynamical Systems},
  journal = {SIAM Review},
  volume  = {67},
  number  = {1},
  pages   = {107--137},
  year    = {2025},
  doi     = {10.1137/24M1696974}
}

@incollection{Takens2010,
  author    = {Takens, Floris},
  title     = {Reconstruction Theory and Nonlinear Time Series Analysis},
  booktitle = {Handbook of Dynamical Systems},
  editor    = {Broer, Henk W. and Hasselblatt, Boris and Takens, Floris},
  publisher = {Elsevier},
  year      = {2010},
  volume    = {3},
  chapter   = {7},
  pages     = {347--377},
  doi       = {10.1016/S1874-575X(10)00315-2},
}

@incollection{Guckenheimer2002,
  author    = {Guckenheimer, John},
  title     = {Numerical Analysis of Dynamical Systems},
  booktitle = {Handbook of Dynamical Systems},
  editor    = {Fiedler, Bernold},
  publisher = {Elsevier},
  year      = {2002},
  volume    = {2},
  chapter   = {8},
  pages     = {347--390},
  doi       = {10.1016/S1874-575X(02)80029-7},
}

@article{Duan2023,
  author    = {Duan, Xiaoyu and Rubin, Jonathan E. and Swigon, David},
  title     = {Qualitative inverse problems: mapping data to the features of trajectories and parameter values of an ODE model},
  journal   = {Inverse Problems},
  year      = {2023},
  volume    = {39},
  number    = {7},
  pages     = {075002},
  doi       = {10.1088/1361-6420/acd414},
}

@article{Engl2009,
  author    = {Heinz W.\ Engl and Christiane Flamm and Philipp K{\"u}gler
               and James Lu and Stefan M{\"u}ller and Peter Schuster},
  title     = {Inverse Problems in Systems Biology},
  journal   = {Inverse Problems},
  year      = {2009},
  volume    = {25},
  number    = {12},
  pages     = {123014},
  doi       = {10.1088/0266-5611/25/12/123014},
}

@misc{Shumaylov2026,
	archiveprefix = {arXiv},
	author = {Zakhar Shumaylov and Peter Zaika and Philipp Scholl and Gitta Kutyniok and Lior Horesh and Carola-Bibiane Sch{\"o}nlieb},
	doi = {10.48550/arXiv.2511.08860},
	eprint = {2511.08860},
	primaryclass = {math.DS},
	title = {When is a System Discoverable from Data? Discovery Requires Chaos},
	year = {2026}}

@article{Stanhope2014,
	author = {Shelby Stanhope and Jonathan E. Rubin and David Swigon},
	doi = {10.1137/130937913},
	journal = {{SIAM} Journal on Applied Dynamical Systems},
	number = {4},
	pages = {1792--1815},
	title = {Identifiability of Linear and Linear-in-Parameters Dynamical Systems from a Single Trajectory},
	volume = {13},
	year = {2014}}

@article{Qiu2022,
	author = {Xing Qiu and Tao Xu and Babak Soltanalizadeh and Hulin Wu},
	doi = {10.1016/j.amc.2022.127260},
	journal = {Applied Mathematics and Computation},
	pages = {127260},
	title = {Identifiability analysis of linear ordinary differential equation systems with a single trajectory},
	volume = {430},
	year = {2022}}

@inproceedings{Casolo2025,
	author = {Casolo, Cecilia and Becker, S{\"o}ren and  Kilbertus, Niki},
	booktitle = {14th International Conference on Learning Representations ({ICLR} 2026)},
	title = {Identifiability Challenges in Sparse Linear Ordinary Differential Equations},
	url = {https://openreview.net/forum?id=BYBKqpZteT},
	year = {2026}
}

@article{Scholl2022,
	author = {Philipp Scholl and Aras Bacho and Holger Boche and Gitta Kutyniok},
	doi = {10.1007/s10994-026-07068-0},
	journal = {Machine Learning},
	note = {Preprint arXiv:2210.08342 (2022)},
	number = {6},
	pages = {139},
	title = {Symbolic Recovery of Differential Equations: The Identifiability Problem},
	volume = {115},
	year = {2026}}

@inproceedings{Scholl2023,
	author = {Philipp Scholl and Aras Bacho and Holger Boche and Gitta Kutyniok},
	booktitle = {{ICASSP} 2023 -- 2023 {IEEE} International Conference on Acoustics, Speech and Signal Processing ({ICASSP})},
	doi = {10.1109/ICASSP49357.2023.10095017},
	pages = {1--5},
	publisher = {{IEEE}},
	title = {The Uniqueness Problem of Physical Law Learning},
	year = {2023}
}

@inproceedings{Hauger2024,
	author = {Hillary Hauger and Philipp Scholl and Gitta Kutyniok},
	booktitle = {{ICASSP} 2025 -- 2025 {IEEE} International Conference on Acoustics, Speech and Signal Processing ({ICASSP})},
	doi = {10.1109/ICASSP49660.2025.10887720},
	pages = {1--5},
	publisher = {{IEEE}},
	title = {Robust Identifiability for Symbolic Recovery of Differential Equations},
	year = {2025}
}

@book{Brunton2022,
	author = {Steven L. Brunton and J. Nathan Kutz},
	doi = {10.1017/9781009089517},
	edition = {2nd},
	publisher = {Cambridge University Press},
	title = {Data-Driven Science and Engineering: {M}achine Learning, Dynamical Systems, and Control},
	year = {2022}}

@incollection{Colbrook:2026aa,
	author = {Matthew Colbrook and Zlatko Drma{\v c} and Andrew Horning},
	booktitle = {Operator Theory},
	doi = {10.1007/978-3-032-16356-1_126},
	edition = {2nd},
	editor = {Daniel Alpay and Fabrizio Colombo and Irene Sabadini},
	pages = {3247--3295},
	publisher = {Springer Verlag},
	title = {An Introductory Guide to Koopman Learning},
	year = {2026}}

@book{Kepler1609,
	address = {Heidelberg},
	author = {Johannes Kepler},
	doi = {10.3931/e-rara-558},
	publisher = {Voegelin},
	title = {{Astronomia Nova}},
	url = {https://doi.org/10.3931/e-rara-558},
	year = {1609}}

@book{Newton1687,
	address = {London},
	author = {Isaac Newton},
	doi = {10.3931/e-rara-440},
	publisher = {Jussu Societatis Regiae ac Typis Josephi Streater},
	title = {{Philosophiae Naturalis Principia Mathematica}},
	url = {https://doi.org/10.3931/e-rara-440},
	year = {1687}}

@book{Fourier1822,
	address = {Paris},
	author = {Jean Baptiste Joseph Fourier},
	doi = {10.3931/e-rara-19706},
	publisher = {Chez Firmin Didot, p{\`e}re et fils},
	title = {Th{\'e}orie analytique de la chaleur},
	url = {https://doi.org/10.3931/e-rara-19706},
	year = {1822}}

@book{Ohm1827,
	address = {Berlin},
	author = {Georg Simon Ohm},
	doi = {10.3931/e-rara-4050},
	publisher = {Riemann},
	title = {{D}ie galvanische {K}ette, mathematisch bearbeitet},
	url = {https://doi.org/10.3931/e-rara-4050},
	year = {1827}}

@article{Fick1855,
	author = {Adolf Fick},
	doi = {10.1002/andp.18551700105},
	journal = {Annalen der Physik},
	number = {1},
	pages = {59--86},
	title = {{U}eber {D}iffusion},
	volume = {170},
	year = {1855}}

@article{HoKalman1966,
	author = {B. L. Ho and Rudolf E. K{\'a}lm{\'a}n},
	doi = {10.1524/auto.1966.14.112.545},
	journal = {at -- Automatisierungstechnik},
	note = {Originally published in \emph{Regelungstechnik}; back-file registered by the publisher under the journal's current title},
	number = {1--12},
	pages = {545--548},
	title = {Effective construction of linear state-variable models from input/output functions},
	volume = {14},
	year = {1966}}

@article{BellmanAstrom1970,
	author = {Richard Bellman and Karl Johan {\AA}str{\"o}m},
	doi = {10.1016/0025-5564(70)90132-X},
	journal = {Mathematical Biosciences},
	number = {3--4},
	pages = {329--339},
	title = {On structural identifiability},
	volume = {7},
	year = {1970}}

@article{Langley1981,
	author = {Pat Langley},
	doi = {10.1111/j.1551-6708.1981.tb00869.x},
	journal = {Cognitive Science},
	number = {1},
	pages = {31--54},
	title = {Data-Driven Discovery of Physical Laws},
	volume = {5},
	year = {1981}}

@book{Langley1987,
	address = {Cambridge, MA},
	author = {Pat Langley and Herbert A. Simon and Gary L. Bradshaw and Jan M. Zytkow},
	doi = {10.7551/mitpress/6090.001.0001},
	isbn = {9780262121163},
	publisher = {{MIT} Press},
	title = {Scientific Discovery: Computational Explorations of the Creative Processes},
	year = {1987}}

@article{NarendraParthasarathy1990,
	author = {Kumpati S. Narendra and Kannan Parthasarathy},
	doi = {10.1109/72.80202},
	journal = {{IEEE} Transactions on Neural Networks},
	number = {1},
	pages = {4--27},
	title = {Identification and control of dynamical systems using neural networks},
	volume = {1},
	year = {1990}}

@article{BongardLipson2007,
	author = {Josh Bongard and Hod Lipson},
	doi = {10.1073/pnas.0609476104},
	journal = {Proceedings of the National Academy of Sciences},
	number = {24},
	pages = {9943--9948},
	title = {Automated reverse engineering of nonlinear dynamical systems},
	volume = {104},
	year = {2007}}

@article{Simpkins2012,
	author = {Alex Simpkins},
	doi = {10.1109/MRA.2012.2192817},
	journal = {{IEEE} Robotics \& Automation Magazine},
	number = {2},
	pages = {95--96},
	title = {System Identification: Theory for the User, 2nd Edition ({L}jung, {L}.; 1999) [{O}n the {S}helf]},
	volume = {19},
	year = {2012}}

@article{YuWang2024,
	author = {Rose Yu and Rui Wang},
	doi = {10.1073/pnas.2311808121},
	journal = {Proceedings of the National Academy of Sciences},
	number = {27},
	pages = {e2311808121},
	title = {Learning dynamical systems from data: An introduction to physics-guided deep learning},
	volume = {121},
	year = {2024}}

@article{North2023,
	author = {Joshua S. North and Christopher K. Wikle and Erin M. Schliep},
	doi = {10.1111/insr.12554},
	journal = {International Statistical Review},
	number = {3},
	pages = {464--492},
	title = {A Review of Data-Driven Discovery for Dynamic Systems},
	volume = {91},
	year = {2023}}

@article{Mezic2005,
	author = {Igor Mezi{\'c}},
	doi = {10.1007/s11071-005-2824-x},
	journal = {Nonlinear Dynamics},
	number = {1},
	pages = {309--325},
	title = {Spectral Properties of Dynamical Systems, Model Reduction and Decompositions},
	volume = {41},
	year = {2005}}

@article{Schmid2010,
	author = {Peter J. Schmid},
	doi = {10.1017/S0022112010001217},
	journal = {Journal of Fluid Mechanics},
	pages = {5--28},
	title = {Dynamic mode decomposition of numerical and experimental data},
	volume = {656},
	year = {2010}}

@book{Kutz2016,
	address = {Philadelphia, PA},
	author = {J. Nathan Kutz and Steven L. Brunton and Bingni W. Brunton and Joshua L. Proctor},
	doi = {10.1137/1.9781611974508},
	publisher = {Society for Industrial and Applied Mathematics},
	title = {Dynamic Mode Decomposition: Data-Driven Modeling of Complex Systems},
	year = {2016}}

@article{Raissi2018,
	author = {Maziar Raissi},
	journal = {Journal of Machine Learning Research},
	number = {25},
	pages = {1--24},
	title = {Deep Hidden Physics Models: Deep Learning of Nonlinear Partial Differential Equations},
	url = {https://www.jmlr.org/papers/v19/18-046.html},
	volume = {19},
	year = {2018}}

@inproceedings{Chen2018,
	author = {Ricky T. Q. Chen and Yulia Rubanova and Jesse Bettencourt and David K. Duvenaud},
	booktitle = {32nd International Conference on Neural Information Processing System ({NIPS'18})},
	pages = {6572--6583},
	title = {Neural Ordinary Differential Equations},
	url = {https://dl.acm.org/doi/10.5555/3327757.3327764},
	volume = {31},
	year = {2018}
}

@inproceedings{Greydanus2019,
	author = {Samuel Greydanus and Misko Dzamba and Jason Yosinski},
	booktitle = {33rd International Conference on Neural Information Processing Systems ({NIPS'19})},
	pages = {15379--15389 (Article No.: 1378)},
	title = {Hamiltonian Neural Networks},
	url = {https://dl.acm.org/doi/10.5555/3454287.3455665},
	volume = {32},
	year = {2019}
}

@misc{Cranmer2020,
	archiveprefix = {arXiv},
	author = {Miles Cranmer and Sam Greydanus and Stephan Hoyer and Peter Battaglia and David Spergel and Shirley Ho},
	doi = {10.48550/arXiv.2003.04630},
	eprint = {2003.04630},
	primaryclass = {cs.LG},
	title = {Lagrangian Neural Networks},
	url = {https://arxiv.org/abs/2003.04630},
	year = {2020}}

@inproceedings{Bilos2021,
	author = {Marin Bilo{\v{s}} and Johanna Sommer and Syama Sundar Rangapuram and Tim Januschowski and Stephan G{\"u}nnemann},
	booktitle = {Advances in Neural Information Processing Systems},
	pages = {21325--21337},
	title = {Neural Flows: Efficient Alternative to Neural {ODE}s},
	url = {https://dl.acm.org/doi/10.5555/3540261.3541892},
	volume = {34},
	year = {2021}}

@misc{Canizares2024,
	archiveprefix = {arXiv},
	author = {Priscilla Canizares and Davide Murari and Carola-Bibiane Sch{\"o}nlieb and Ferdia Sherry and Zakhar Shumaylov},
	doi = {10.48550/arXiv.2412.16787},
	eprint = {2412.16787},
	primaryclass = {cs.LG},
	title = {Symplectic Neural Flows for Modeling and Discovery},
	url = {https://arxiv.org/abs/2412.16787},
	year = {2024}}

@article{Lu2021,
	author = {Lu Lu and Pengzhan Jin and Guofei Pang and Zhongqiang Zhang and George Em Karniadakis},
	doi = {10.1038/s42256-021-00302-5},
	journal = {Nature Machine Intelligence},
	number = {3},
	pages = {218--229},
	title = {Learning nonlinear operators via {D}eep{ON}et based on the universal approximation theorem of operators},
	volume = {3},
	year = {2021}}

@article{Kovachki2023,
	author = {Nikola Kovachki and Zongyi Li and Burigede Liu and Kamyar Azizzadenesheli and Kaushik Bhattacharya and Andrew Stuart and Anima Anandkumar},
	journal = {Journal of Machine Learning Research},
	number = {89},
	pages = {1--97},
	title = {Neural Operator: Learning Maps Between Function Spaces With Applications to {PDE}s},
	url = {https://www.jmlr.org/papers/v24/21-1524.html},
	volume = {24},
	year = {2023}}

@inproceedings{Li2021,
	author = {Zongyi Li and Nikola Borislavov Kovachki and Kamyar Azizzadenesheli and Burigede Liu and Kaushik Bhattacharya and Andrew M. Stuart and Anima Anandkumar},
	booktitle = {International Conference on Learning Representations ({ICLR})},
	title = {{F}ourier Neural Operator for Parametric Partial Differential Equations},
	url = {https://openreview.net/forum?id=c8P9NQVtmnO},
	year = {2021}}

@inproceedings{Kurth2023,
	address = {New York, NY, USA},
	author = {Thorsten Kurth and Shashank Subramanian and Peter Harrington and Jaideep Pathak and Morteza Mardani and David Hall and Andrea Miele and Karthik Kashinath and Anima Anandkumar},
	booktitle = {Proceedings of the Platform for Advanced Scientific Computing Conference ({PASC} '23)},
	doi = {10.1145/3592979.3593412},
	pages = {1--11},
	publisher = {{ACM}},
	title = {{F}our{C}ast{N}et: Accelerating Global High-Resolution Weather Forecasting Using Adaptive {F}ourier Neural Operators},
	year = {2023}}

@article{Gopakumar2024,
	author = {Vignesh Gopakumar and Stanislas Pamela and Lorenzo Zanisi and Zongyi Li and Ander Gray and Daniel Brennand and Nitesh Bhatia and Gregory Stathopoulos and Matt Kusner and Marc Peter Deisenroth and Anima Anandkumar and {the JOREK Team} and {MAST Team}},
	doi = {10.1088/1741-4326/ad313a},
	journal = {Nuclear Fusion},
	number = {5},
	pages = {056025},
	title = {Plasma surrogate modelling using {F}ourier neural operators},
	volume = {64},
	year = {2024}}

@misc{Vafa2025,
	archiveprefix = {arXiv},
	author = {Keyon Vafa and Peter G. Chang and Ashesh Rambachan and Sendhil Mullainathan},
	doi = {10.48550/arXiv.2507.06952},
	eprint = {2507.06952},
	primaryclass = {cs.LG},
	title = {What Has a Foundation Model Found? {U}sing Inductive Bias to Probe for World Models},
	url = {https://arxiv.org/abs/2507.06952},
	year = {2025}}

@inproceedings{LaCava:2021aa,
	author = {William G. La Cava and Patryk Orzechowski and Bogdan Burlacu and Fabr{\'\i}cio Olivetti de Fran{\c{c}}a and Marco Virgolin and Ying Jin and Michael Kommenda and Jason H. Moore},
	booktitle = {Proceedings of the Neural Information Processing Systems Track on Datasets and Benchmarks},
	title = {Contemporary Symbolic Regression Methods and their Relative Performance},
	url = {https://openreview.net/forum?id=cyZtOENkAOY},
	volume = {1},
	year = {2021}}

@misc{Cranmer2023,
	archiveprefix = {arXiv},
	author = {Miles Cranmer},
	doi = {10.48550/arXiv.2305.01582},
	eprint = {2305.01582},
	primaryclass = {astro-ph.IM},
	title = {Interpretable Machine Learning for Science with {P}y{SR} and {S}ymbolic{R}egression.jl},
	url = {https://arxiv.org/abs/2305.01582},
	year = {2023}}

@inproceedings{Petersen2021,
	author = {Brenden K. Petersen and Mikel Landajuela and T. Nathan Mundhenk and Claudio P. Santiago and Sookyung Kim and Joanne Taery Kim},
	booktitle = {9th International Conference on Learning Representations ({ICLR} 2021)},
	title = {Deep Symbolic Regression: Recovering Mathematical Expressions from Data via Risk-Seeking Policy Gradients},
	url = {https://openreview.net/forum?id=m5Qsh0kBQG},
	year = {2021}}

@article{UdrescuTegmark2020,
	author = {Silviu-Marian Udrescu and Max Tegmark},
	doi = {10.1126/sciadv.aay2631},
	journal = {Science Advances},
	number = {16},
	pages = {eaay2631},
	title = {{AI} {F}eynman: A physics-inspired method for symbolic regression},
	volume = {6},
	year = {2020}}

@article{Cornelio2023,
	author = {Cristina Cornelio and Sanjeeb Dash and Vernon Austel and Tyler R. Josephson and Joao Goncalves and Kenneth L. Clarkson and Nimrod Megiddo and Bachir El Khadir and Lior Horesh},
	doi = {10.1038/s41467-023-37236-y},
	journal = {Nature Communications},
	number = {1},
	pages = {1777},
	title = {Combining data and theory for derivable scientific discovery with {AI}-{D}escartes},
	volume = {14},
	year = {2023}}

@misc{Srivastava2025,
	archiveprefix = {arXiv},
	author = {Karan Srivastava and Sanjeeb Dash and Ryan Cory-Wright and Barry Trager and Cristina Cornelio and Lior Horesh},
	doi = {10.48550/arXiv.2509.23004},
	eprint = {2509.23004},
	primaryclass = {cs.AI},
	title = {Bridging the Gap Between Scientific Laws Derived by {AI} Systems and Canonical Knowledge via Abductive Inference with {AI}-{N}oether},
	url = {https://arxiv.org/abs/2509.23004},
	year = {2025}}

@inproceedings{MartiusLampert2017,
	archiveprefix = {arXiv},
	author = {Georg Martius and Christoph H. Lampert},
	booktitle = {5th International Conference on Learning Representations ({ICLR} 2017), Workshop Track Proceedings},
	eprint = {1610.02995},
	primaryclass = {cs.LG},
	title = {Extrapolation and Learning Equations},
	url = {https://arxiv.org/abs/1610.02995},
	year = {2017}}

@inproceedings{Sahoo2018,
	author = {Subham S. Sahoo and Christoph H. Lampert and Georg Martius},
	booktitle = {Proceedings of the 35th International Conference on Machine Learning},
	editor = {Jennifer Dy and Andreas Krause},
	pages = {4442--4450},
	publisher = {{PMLR}},
	series = {Proceedings of Machine Learning Research},
	title = {Learning Equations for Extrapolation and Control},
	url = {https://proceedings.mlr.press/v80/sahoo18a.html},
	volume = {80},
	year = {2018}}

@inproceedings{Scholl2025,
	author = {Philipp Scholl and Katharina Bieker and Hillary Hauger and Gitta Kutyniok},
	booktitle = {The Thirteenth International Conference on Learning Representations ({ICLR} 2025)},
	title = {{ParFam} -- (Neural Guided) Symbolic Regression via Continuous Global Optimization},
	url = {https://openreview.net/forum?id=8y5Uf6oEiB},
	year = {2025}}

@inproceedings{Biggio2021,
	author = {Luca Biggio and Tommaso Bendinelli and Alexander Neitz and Aurelien Lucchi and Giambattista Parascandolo},
	booktitle = {Proceedings of the 38th International Conference on Machine Learning},
	editor = {Marina Meila and Tong Zhang},
	pages = {936--945},
	publisher = {{PMLR}},
	series = {Proceedings of Machine Learning Research},
	title = {Neural Symbolic Regression that Scales},
	url = {https://proceedings.mlr.press/v139/biggio21a.html},
	volume = {139},
	year = {2021}}

@inproceedings{Kamienny2022,
	author = {Pierre-Alexandre Kamienny and St{\'e}phane d'Ascoli and Guillaume Lample and Fran{\c{c}}ois Charton},
	booktitle = {Advances in Neural Information Processing Systems},
	pages = {10269--10281},
	title = {End-to-end Symbolic Regression with Transformers},
	url = {https://openreview.net/forum?id=GoOuIrDHG_Y},
	volume = {35},
	year = {2022}}

@inproceedings{Grayeli2024,
	author = {Arya Grayeli and Atharva Sehgal and Omar Costilla-Reyes and Miles Cranmer and Swarat Chaudhuri},
	booktitle = {Advances in Neural Information Processing Systems},
	pages = {44678--44709},
	title = {Symbolic Regression with a Learned Concept Library},
	url = {https://openreview.net/forum?id=B7S4jJGlvl},
	volume = {37},
	year = {2024}}

@inproceedings{Shojaee2025,
	author = {Parshin Shojaee and Kazem Meidani and Shashank Gupta and Amir Barati Farimani and Chandan K. Reddy},
	booktitle = {The Thirteenth International Conference on Learning Representations ({ICLR} 2025)},
	title = {{LLM-SR}: Scientific Equation Discovery via Programming with Large Language Models},
	url = {https://openreview.net/forum?id=m2nmp8P5in},
	year = {2025}}

@article{Quade2018,
	author = {Markus Quade and Markus Abel and J. Nathan Kutz and Steven L. Brunton},
	doi = {10.1063/1.5027470},
	journal = {Chaos: An Interdisciplinary Journal of Nonlinear Science},
	number = {6},
	pages = {063116},
	title = {Sparse identification of nonlinear dynamics for rapid model recovery},
	volume = {28},
	year = {2018}}

@inproceedings{Qian2022,
	author = {Zhaozhi Qian and Krzysztof Kacprzyk and Mihaela van der Schaar},
	booktitle = {The Tenth International Conference on Learning Representations ({ICLR} 2022)},
	title = {{D-CODE}: Discovering Closed-form {ODE}s from Observed Trajectories},
	url = {https://openreview.net/forum?id=wENMvIsxNN},
	year = {2022}}

@article{Alessandrini1986,
	author = {Giovanni Alessandrini},
	doi = {10.1007/BF01790543},
	journal = {Annali di Matematica Pura ed Applicata},
	number = {1},
	pages = {265--295},
	title = {An identification problem for an elliptic equation in two variables},
	volume = {145},
	year = {1986}}

@article{Acar1993,
	author = {Robert Acar},
	doi = {10.1137/0331058},
	journal = {{SIAM} Journal on Control and Optimization},
	number = {5},
	pages = {1221--1244},
	title = {Identification of the Coefficient in Elliptic Equations},
	volume = {31},
	year = {1993}}

@article{Knowles2001,
	author = {Ian Knowles},
	doi = {10.1016/S0377-0427(00)00275-2},
	journal = {Journal of Computational and Applied Mathematics},
	number = {1--2},
	pages = {175--194},
	title = {Parameter identification for elliptic problems},
	volume = {131},
	year = {2001}}

@article{CobelliDistefano1980,
	author = {Claudio Cobelli and Joseph J. {DiStefano III}},
	doi = {10.1152/ajpregu.1980.239.1.R7},
	journal = {American Journal of Physiology-Regulatory, Integrative and Comparative Physiology},
	number = {1},
	pages = {R7--R24},
	title = {Parameter and structural identifiability concepts and ambiguities: a critical review and analysis},
	volume = {239},
	year = {1980}}

@article{DistefanoCobelli1980,
	author = {Joseph J. {DiStefano III} and Claudio Cobelli},
	doi = {10.1109/TAC.1980.1102439},
	journal = {{IEEE} Transactions on Automatic Control},
	number = {4},
	pages = {830--833},
	title = {On parameter and structural identifiability: Nonunique observability/reconstructibility for identifiable systems, other ambiguities, and new definitions},
	volume = {25},
	year = {1980}}

@article{Miao2011,
	author = {Hongyu Miao and Xiaohua Xia and Alan S. Perelson and Hulin Wu},
	doi = {10.1137/090757009},
	journal = {{SIAM} Review},
	number = {1},
	pages = {3--39},
	title = {On Identifiability of Nonlinear {ODE} Models and Applications in Viral Dynamics},
	volume = {53},
	year = {2011}}

@article{NguyenWood1982,
	author = {V. V. Nguyen and Eric F. Wood},
	doi = {10.1137/1024002},
	journal = {{SIAM} Review},
	number = {1},
	pages = {34--51},
	title = {Review and Unification of Linear Identifiability Concepts},
	volume = {24},
	year = {1982}}

@article{Messenger2021WeakSINDyODE,
  author  = {Messenger, Daniel A. and Bortz, David M.},
  title   = {Weak {SINDy}: {Galerkin}-Based Data-Driven Model Selection},
  journal = {Multiscale Modeling \& Simulation},
  volume  = {19},
  number  = {3},
  pages   = {1474--1497},
  year    = {2021},
  doi     = {10.1137/20M1343166},
}

@conference{Liu:2021aa,
  author    = {Ziming Liu and Yunyue Chen and Yuanqi Du and Max Tegmark},
  booktitle = {NeurIPS 2021 AI for Science Workshop ({NeurIPS-AI4Science})},
  title     = {Physics-Augmented Learning: A New Paradigm Beyond Physics-Informed Learning},
  url = {https://openreview.net/forum?id=suxElmrPNAY},
  year = {2021}
}

@article{Ahmadi:2023aa,
  author  = {Ahmadi, Amir Ali and El Khadir, Bachir},
  title   = {Learning Dynamical Systems with Side Information},
  journal = {SIAM Review},
  volume  = {65},
  number  = {1},
  pages   = {183--223},
  year    = {2023},
  doi     = {10.1137/20M1388644}
}

@inproceedings{Cranmer:2020aa,
  author    = {Cranmer, Miles and Sanchez-Gonzalez, Alvaro and Battaglia, Peter and Xu, Rui and Cranmer, Kyle and Spergel, David and Ho, Shirley},
  title     = {Discovering Symbolic Models from Deep Learning with Inductive Biases},
  booktitle = {34th International Conference on Neural Information Processing Systems ({NeurIPS 2020})},
  pages = {17429--17442 (Article No.:~1462)},
  url = {https://proceedings.neurips.cc/paper/2020/hash/c9f2f917078bd2db12f23c3b413d9cba-Abstract.html},
  year = {2020}
}

@article{Bartlett:2023aa,
  author  = {Bartlett, Deaglan and Desmond, Harry and Ferreira, Pedro},
  title   = {Exhaustive Symbolic Regression},
  journal = {IEEE Transactions on Evolutionary Computation},
  volume  = {28},
  number  = {4},
  pages   = {950--964},
  year    = {2024},
  doi     = {10.1109/TEVC.2023.3280250}
}

@article{Virgolin:2022aa,
  author  = {Virgolin, Marco and Pissis, Solon P.},
  title   = {Symbolic Regression is {NP}-hard},
  journal = {Transactions on Machine Learning Research},
  volume = {10},
  url = {https://openreview.net/forum?id=LTiaPxqe2e},
  year    = {2022}
}

@book{Koza1992,
  author    = {Koza, John R.},
  title     = {Genetic Programming: On the Programming of Computers
               by Means of Natural Selection},
  publisher = {MIT Press},
  address   = {Cambridge, MA},
  year      = {1992}
}

@article{DonohoElad2003,
  author  = {Donoho, David L. and Elad, Michael},
  title   = {Optimally Sparse Representation in General (Nonorthogonal)
             Dictionaries via $\ell^1$ Minimization},
  journal = {Proceedings of the National Academy of Sciences},
  volume  = {100},
  number  = {5},
  pages   = {2197--2202},
  year    = {2003},
  doi     = {10.1073/pnas.0437847100}
}

@article{CohenDahmenDeVore2009,
  author  = {Cohen, Albert and Dahmen, Wolfgang and DeVore, Ronald},
  title   = {Compressed Sensing and Best $k$-Term Approximation},
  journal = {Journal of the American Mathematical Society},
  volume  = {22},
  number  = {1},
  pages   = {211--231},
  year    = {2009},
  doi     = {10.1090/S0894-0347-08-00610-3}
}

@article{Natarajan1995,
  author  = {Natarajan, B. K.},
  title   = {Sparse Approximate Solutions to Linear Systems},
  journal = {SIAM Journal on Computing},
  volume  = {24},
  number  = {2},
  pages   = {227--234},
  year    = {1995},
  doi     = {10.1137/S0097539792240406}
}

@misc{Fasel2021,
  archiveprefix = {arXiv},
  author    = {Urban Fasel and Eurika Kaiser and J. Nathan Kutz and Bingni W. Brunton and Steven L. Brunton},
  doi       = {10.48550/arXiv.2108.13404},
  eprint    = {2108.13404},
  primaryclass = {math.OC},
  title     = {{SINDy} with Control: A Tutorial},
  year      = {2021}}

@misc{Purnomo2023,
  archiveprefix = {arXiv},
  author    = {Adam Purnomo and Mitsuhiro Hayashibe},
  doi       = {10.48550/arXiv:2209.03248},
  eprint    = {2209.03248},
  primaryclass = {eess.SY},
  title     = {Sparse Identification of Lagrangian for Nonlinear Dynamical Systems via Proximal Gradient Method},
  year = {2022}}

@article{Kaheman:2020aa,
  author    = {Kadierdan Kaheman and Nathan J. Kutz and Steven L. Brunton},
  title     = {{SINDy-PI}: a robust algorithm for parallel implicit sparse identification of nonlinear dynamics},
  journal   = {Proceedings of the Royal Society A: Mathematical, physical and engineering sciences},
  year      = {2020},
  volume    = {476},
  number    = {2242},
  pages     = {20200279},
  doi       = {10.1098/rspa.2020.0279}
}

@article{Fasel:2022aa,
  author    = {Urban Fasel and Nathan J. Kutz and Bingni W. Brunton and Steven L. Brunton},
  title     = {{Ensemble-SINDy}: Robust sparse model discovery in the low-data, high-noise limit, with active learning and control},
  journal   = {Proceedings of the Royal Society A: Mathematical, physical and engineering sciences},
  year      = {2022},
  volume    = {476},
  number    = {2260},
  pages     = {20210904},
  doi       = {10.1098/rspa.2021.0904}
}

@article{Messenger:2021aa,
  author  = {Messenger, Daniel A. and Bortz, David M.},
  title   = {Weak {SINDy} for Partial Differential Equations},
  journal = {Journal of Computational Physics},
  volume  = {443},
  pages   = {110525},
  year    = {2021},
  doi     = {10.1016/j.jcp.2021.110525},
}

@article{Fung:2025aa,
  author    = {Lloyd Fung and Urban Fasel and Matthew Juniper},
  title     = {Rapid {Bayesian} identification of sparse nonlinear dynamics from scarce and noisy data},
  journal   = {Proceedings of the Royal Society A: Mathematical, physical and engineering sciences},
  year      = {2025},
  volume    = {481},
  number    = {2307},
  pages     = {20240200},
  doi       = {10.1098/rspa.2024.0200}
}

@article{Horrocks:2020aa,
  author    = {Jonathan Horrocks and Chris T. Bauch},
  title     = {Algorithmic discovery of dynamic models from infectious disease data},
  journal   = {Scientific Reports},
  year      = {2020},
  volume    = {10},
  pages     = {7061},
  doi       = {10.1038/s41598-020-63877-w}
}

@article{Rudy:2017aa,
  author    = {Samuel H. Rudy and Steven L. Brunton and Joshua L. Proctor and J. Nathan Kutz},
  title     = {Data-driven discovery of partial differential equations},
  journal   = {Science Advances},
  year      = {2017},
  volume    = {3},
  number    = {4},
  pages     = {e1602614},
  doi       = {10.1126/sciadv.1602614}
}

@article{Lagergren:2020aa,
  author    = {John H. Lagergren and John T. Nardini and Michael G. Lavigne and Erica M. Rutter and Kevin B. Flores},
  title     = {Learning partial differential equations for biological transport models from noisy spatio-temporal data},
  journal   = {Proceedings of the Royal Society A: Mathematical, physical and engineering sciences},
  year      = {2020},
  volume    = {476},
  number    = {2234},
  pages     = {20190800},
  doi       = {10.1098/rspa.2019.0800}
}

@article{Kaptanoglu:2021aa,
  author    = {Alan A. Kaptanoglu and Jared L. Callaham and Aleksandr Aravkin and Christopher J. Hansen and Steven L. Brunton},
  title     = {Promoting global stability in data-driven models of quadratic nonlinear dynamics},
  journal   = {Physical Review Fluids},
  year      = {2021},
  volume    = {6},
  pages     = {094401},
  doi       = {10.1103/PhysRevFluids.6.094401}
}

@book{Scherzer:2009aa,
	author = {Scherzer, Otmar and Grasmair, Markus and Grossauer, Harald and Haltmeier, Markus and Lenzen, Frank},
	doi = {10.1007/978-0-387-69277-7},
	publisher = {Springer-Verlag},
	series = {Applied Mathematical Sciences},
	title = {Variational Methods in Imaging},
	volume = {167},
	year = {2009}}

@article{Makke:2024aa,
	author = {Nour Makke and Sanjay Chawla},
	journal = {Artificial Intelligence Review},
	pages = {Article No.:~2},
	title = {Interpretable scientific discovery with symbolic regression: a review},
	volume = {57},
	year = {2024},
   	doi = {10.1007/s10462-023-10622-0}
}

@book{Foucart:2013aa,
    author = {Simon Foucart and Holger Rauhut},
    title = {Mathematical Introduction to Compressive Sensing},
    publisher = {Birkh{\"a}user},
    series = {Applied and Numerical Harmonic Analysis},
    year = {2013},
    doi = {10.1007/978-0-8176-4948-7}
}

@article{Candes:2006aa,
  author    = {Emanuel J. Cand{\'e}s and Justin Romberg and Terrence Tao},
  title     = {Robust uncertainty principles: exact signal reconstruction from highly incomplete frequency information},
  journal   = {IEEE Transactions on Information Theory},
  year      = {2006},
  volume    = {52},
  number    = {2},
  pages     = {489--509},
  doi       = {10.1109/TIT.2005.862083}
}

@article{Donoho:1992aa,
  author    = {David L. Donoho and Benjamin F. Logan},
  title     = {Signal Recovery and the Large Sieve},
  journal   = {SIAM Journal on Applied Mathematics},
  year      = {1992},
  volume    = {52},
  number    = {2},
  pages     = {577--591},
  doi       = {10.1137/0152031}
}

@article{Santosa:1986aa,
  author    = {Fadil Santosa and William W. Symes},
  title     = {Linear Inversion of Band-Limited Reflection Seismograms},
  journal   = {SIAM Journal on Scientific and Statistical Computing},
  year      = {1986},
  volume    = {7},
  number    = {4},
  pages     = {1307--1330},
  doi       = {10.1137/0907087}
}

@incollection{Beurling:1938aa,
  author    = {Arne Beurling},
  editor    = {Lennart Carleson and John Neuberger and Paul Malliavin and John Wermer},
  title     = {Sur les int{\'e}grales de {Fourier} absolument convergentes et leur application {\`a} une transformation fonctionnelle},
  booktitle = {Collected Works of {Arne Beurling}. Volume 2: Harmonic Analysis},
  pages     = {39--60},
  publisher = {Birkh{\"a}user},
  address   = {Boston},
  note      = {Paper appeard initially in 9th Scandinavian Mathematical Congress, Helsingfors, pp.~345--366, 1938},
  year      = {1989}
}

@article{Brunton:2016aa,
  author    = {Steven L. Brunton and Joshua L. Proctor and J. Nathan Kutz},
  title     = {Discovering Governing Equations from Data by Sparse Identification of Nonlinear Dynamical Systems},
  journal   = {Proceedings of the National Academy of Sciences},
  year      = {2016},
  volume    = {113},
  number    = {15},
  pages     = {3932--3937},
  doi       = {10.1073/pnas.1517384113},
}

@misc{Hillar:2018aa,
  archiveprefix = {arXiv},
  author    = {Christopher Hillar and Friedrich Sommer},
  doi       = {10.48550/arXiv.1210.7273},
  eprint = {1210.7273},
  primaryclass = {nlin.AO},
  title     = {Comment on the article ``Distilling free-form natural laws from experimental data''},
  year = {2018}}

@article{Schmidt:2009aa,
  author    = {Michael Schmidt and Hod Lipson},
  title     = {Distilling Free-Form Natural Laws from Experimental Data},
  journal   = {Science},
  year      = {2009},
  volume    = {324},
  number    = {5923},
  pages     = {81--85},
  doi       = {10.1126/science.1165893}
}

@article{Weinan:2021aa,
	author = {Weinan E},
	title = {The Dawning of a New Era in Applied Mathematics},
	journal = {AMS Notices},
	volume = {68},
    number = {4},
	pages = {565--571},
	doi = {10.1090/noti2259},
	year = {2021}}

@article{Dukes:2011aa,
	author = {Joseph D. Dukes and Paul Whitley and Andrew D Chalmers},
	doi = {10.1016/j.tim.2003.10.002},
	journal = {BMC Cell Biology},
	pages = {Article no.:~43},
	title = {The {MDCK} variety pack: choosing the right strain},
	volume = {12},
	year = {2011}}

@article{Hurley:2003aa,
	author = {Bryan P. Hurley and Beth A. McCormick},
	doi = {10.1016/j.tim.2003.10.002},
	journal = {Trends in Microbiology},
	pages = {562--569},
	title = {Translating tissue culture results into animal models: the case of {Salmonella typhimurium}},
	volume = {11},
    number = {12},
	year = {2003}}

@article{Chen:2013aa,
	author = {Tsai-Wen Chen and Trevor J. Wardill and Yi Sun and Stefan R. Pulver and Sabine L. Renninger and Amy Baohan and Eric R. Schreiter and Rex A. Kerr and Michael B. Orger and Vivek Jayaraman and Loren L. Looger and Karel Svoboda and Douglas S. Kim},
	doi = {10.1038/nature12354},
	journal = {Nature},
	pages = {295--300},
	title = {Ultrasensitive fluorescent proteins for imaging neuronal activity},
	volume = {499},
	year = {2013}}

@incollection{Webb:2012aa,
	author = {Donna J. Webb and Claire M. Brown},
	booktitle = {Cell Imaging Techniques: Methods and Protocols},
	chapter = {2},
	doi = {10.1007/978-1-62703-056-4_2},
	editor = {Douglas J. Taatjes and J{\"u}rgen Roth},
	pages = {29--59},
	publisher = {Springer-Verlag},
	series = {Methods in Molecular Biology},
	title = {Epi-Fluorescence Microscopy},
	volume = {931},
	year = {2012}}

@article{Zhang:2023aa,
	author = {Yan Zhang and M{\'a}rton R{\'o}zsa and Yajie Liang and Daniel Bushey and Ziqiang Wei and Jihong Zheng and Daniel Reep and Gerard Joey Broussard and Arthur Tsang and Getahun Tsegaye and Sujatha Narayan and Christopher J. Obara and Jing-Xuan Lim and Ronak Patel and Rongwei Zhang and Misha B. Ahrens and Glenn C. Turner and Samuel S.-H. Wang and Wyatt L. Korff and Eric R. Schreiter and Karel Svoboda and Jeremy P. Hasseman and Ilya Kolb and Loren L. Looger},
	doi = {10.1038/s41586-023-05828-9},
	journal = {Nature},
	pages = {884--891},
	title = {Fast and sensitive {GCaMP} calcium indicators for imaging neural populations},
	volume = {615},
	year = {2023}}

@article{Hsieh:2025aa,
	author = {Huai-Ching Hsieh and Qinghua Han and David Brenes and Kevin W. Bishop and Rui Wang and Yuli Wang and Chetan Poudel and Adam K. Glaser and Benjamin S. Freedman and Joshua C. Vaughan and Nancy L. Allbritton and Jonathan T. C. Liu},
	doi = {10.1038/s41592-025-02647-w},
	journal = {Nature Methods},
	pages = {1167--1190},
	title = {Imaging {3D} cell cultures with optical microscopy},
	volume = {22},
	year = {2025}}

@article{Schnell:2012aa,
	author = {Ulrike Schnell and Freark Dijk and Klaas A. Sjollema and Ben N. G. Giepmans},
	doi = {10.1038/nmeth.1855},
	journal = {Nature Methods},
	pages = {152--158},
	title = {Immunolabeling artifacts and the need for live-cell imaging},
	volume = {9},
	year = {2008}}

@article{Wang:2008aa,
	author = {Yingxiao Wang and John Y.-J. Shyy and Shu Chien},
	doi = {10.1146/annurev.bioeng.010308.161731},
	journal = {Annual Review of Biomedical Engineering},
	pages = {1--38},
	title = {Fluorescence Proteins, Live-Cell Imaging, and Mechanobiology: Seeing Is Believing},
	volume = {10},
	year = {2008}}

@article{Allport:2001aa,
	author = {Jennifer R. Allport and Ralph Weissleder},
	doi = {10.1016/S0301-472X(01)00739-1},
	journal = {Experimental Hematology},
	pages = {1237--1246},
	title = {In vivo imaging of gene and cell therapies},
	volume = {29},
    number = {11},
	year = {2001}}

@article{Zhuang:2026aa,
	author = {Zhong Zhuang and Zhichao Feng and Jie Wang and Xinhui Liu and Laijun Song and Chunhui Sun and Hong Liu and Na Ren},
	doi = {10.34133/research.1085},
	journal = {Research},
	number = {Article id.:~1085},
	title = {Advanced Imaging for Live-Cell Spatiotemporal Monitoring: Technologies and Applications},
	volume = {9},
	year = {2026}}

@article{Birkhoff,
	author = {George D. Birkhoff},
	doi = {10.1073/pnas.17.2.656},
	journal = {Proceedings of the National Academy of Sciences},
	number = {12},
	pages = {656-660},
	title = {Proof of the Ergodic Theorem},
	volume = {17},
	year = {1931}}

@book{BrinGarrett,
	address = {Cambridge},
	author = {Brin, Michael and Stuck, Garrett},
	doi = {10.1017/CBO9780511755316},
	isbn = {0-521-80841-3},
	publisher = {Cambridge University Press},
	title = {Introduction to dynamical systems},
	year = {2002}}

@book{Mane,
	address = {Berlin},
	author = {Ma\~n\'e, Ricardo},
	doi = {10.1007/978-3-642-70335-5},
	isbn = {3-540-15278-4},
	publisher = {Springer-Verlag},
	series = {Ergebnisse der Mathematik und ihrer Grenzgebiete (3) [Results in Mathematics and Related Areas (3)]},
	title = {Ergodic theory and differentiable dynamics},
	volume = {8},
	year = {1987}}

@book{Feller,
	address = {New York-London-Sydney},
	author = {Feller, William},
	edition = {Second},
	isbn = {978-0-471-25709-7},
	publisher = {John Wiley \& Sons},
	series = {Wiley Series in Probability and Statistics},
	title = {An introduction to probability theory and its applications - Volume II},
	year = {1991}}

@book{Billingsley,
	address = {New York},
	author = {Billingsley, Patrick},
	edition = {Third},
	isbn = {0-471-00710-2},
	publisher = {John Wiley \& Sons},
	series = {Wiley Series in Probability and Mathematical Statistics},
	title = {Probability and measure},
	year = {1995}}

@article{Boltzmann:1877,
	author = {Boltzmann, Ludwig},
	journal = {Wiener Berichte},
	pages = {373-435},
	title = {{\"U}ber die Beziehung zwischen dem zweiten Hauptsatze der mechanischen W{\"a}rmetheorie und der Wahrscheinlichkeitsrechnung respektive den S{\"a}tzen {\"u}ber das W{\"a}rmegleichgewicht},
	volume = {76},
	year = {1877}}

@article{Floryan:2022aa,
	author = {Daniel Floryan and Michael D. Graham},
	doi = {10.1038/s42256-022-00575-4},
	journal = {Nature Machine Intelligence},
	pages = {1113-1120},
	title = {Data-driven discovery of intrinsic dynamics},
	volume = {4},
	year = {2022}}

@book{KatokHasselblatt,
	address = {Cambridge},
	author = {Katok, Anatole and Hasselblatt, Boris},
	doi = {10.1017/CBO9780511809187},
	isbn = {0-521-34187-6},
	publisher = {Cambridge University Press},
	series = {Encyclopedia of Mathematics and its Applications},
	title = {Introduction to the modern theory of dynamical systems},
	volume = {54},
	year = {1995}}

@book{GuckHolmes,
	address = {New York},
	author = {Guckenheimer, John and Holmes, Philip},
	doi = {10.1007/978-1-4612-1140-2},
	isbn = {0-387-90819-6},
	publisher = {Springer-Verlag},
	series = {Applied Mathematical Sciences},
	title = {Nonlinear oscillations, dynamical systems, and bifurcations of vector fields},
	volume = {42},
	year = {1990}}

@conference{Kacprzyk:2025aa,
	author = {Krzysztof Kacprzyk and Mihaela van der Schaar},
	booktitle = {13th International Conference on Learning Representations ({ICLR 2025})},
    url = {https://openreview.net/forum?id=kbm6tsICar},
	title = {No Equations Needed: Learning System Dynamics Without Relying on Closed-Form {ODEs}},
	year = {2025}}

@article{Kowalewski:2006aa,
	author = {Jacob M. Kowalewski and Per Uhl{\'e}n and Hiroaki Kitano and Hjalmar Brismar},
	doi = {10.1016/j.mbs.2006.03.001},
	journal = {Mathematical Biosciences},
	pages = {232-249},
	title = {Modeling the impact of store-operated {Ca}${}^{2+}$ entry on intracellular {Ca}${}^{2+}$ oscillations},
	volume = {204},
	year = {2006}}

@article{Lorenz:1963,
	author = {Edward N. Lorenz},
	doi = {10.1175/1520-0469(1963)020<0130:DNF>2.0.CO;2},
	journal = {Journal of the Atmospheric Sciences},
	number = {2},
	pages = {130-141},
	title = {Deterministic Nonperiodic Flow},
	volume = {20},
	year = {1963}}

@article{Newhouse:1974,
	author = {Newhouse, Sheldon},
	doi = {10.1016/0040-9383(74)90034-2},
	journal = {Topology},
	pages = {9-18},
	title = {Diffeomorphisms with infinitely many sinks},
	volume = {13},
	year = {1974}}

@book{PalisTakens,
	address = {Cambridge},
	author = {Palis, Jacob and Takens, Floris},
	isbn = {0-521-39064-8},
	publisher = {Cambridge University Press},
	series = {Cambridge Studies in Advanced Mathematics},
	title = {Hyperbolicity and sensitive chaotic dynamics at homoclinic bifurcations},
	volume = {35},
	year = {1993}}

@book{Petersen,
	address = {Cambridge},
	author = {Petersen, Karl},
	doi = {10.1017/CBO9780511608728},
	isbn = {0-521-23632-0},
	publisher = {Cambridge University Press},
	series = {Cambridge Studies in Advanced Mathematics},
	title = {Ergodic theory},
	volume = {2},
	year = {1983}}

@book{Poincare:1892,
	address = {Paris},
	author = {Poincar{\'e}, Henri},
	publisher = {Gauthier-Villars},
	title = {Les m{\'e}thodes nouvelles de la m{\'e}canique c{\'e}leste},
	url = {https://archive.org/details/lesmthodesnouv001poin},
	year = {1892}}

@article{Qraitem:2020aa,
	author = {Maan Qraitem and Dhanushka Kularatne and Eric Forgoston and Ani M. Hsieh},
	doi = {10.1016/j.physd.2020.132736},
	journal = {Physica D: Nonlinear Phenomena},
	pages = {132736},
	title = {Bridging the gap: Machine learning to resolve improperly modeled dynamics},
	volume = {414},
	year = {2020}}

@article{Smale:1967,
	author = {Smale, Stephen},
	doi = {10.1090/S0002-9904-1967-11798-1},
	journal = {Bulletin of the American Mathematical Society},
	number = {6},
	pages = {747-817},
	title = {Differentiable dynamical systems},
	volume = {73},
	year = {1967}}

@article{Yu:2023aa,
	author = {Rose Yu and Rui Wang},
	doi = {10.1073/pnas.2311808121},
	journal = {Proceedings of the National Academy of Sciences},
	number = {27},
	pages = {e2311808121},
	title = {Learning dynamical systems from data: An introduction to physics-guided deep learning},
	volume = {121},
	year = {2023}}

@article{Zhong:2020aa,
	author = {Ming Zhong and Jason Miller and Mauro Maggioni},
	doi = {10.1016/j.physd.2020.132542},
	journal = {Physica D: Nonlinear Phenomena},
	pages = {132542},
	title = {Data-driven discovery of emergent behaviors in collective dynamics},
	volume = {411},
	year = {2020}}

@article {FriedmanOrnstein1970,
    AUTHOR = {Nathaniel A. Friedman and Donald Ornstein},
     TITLE = {On isomorphism of weak {B}ernoulli transformations},
   JOURNAL = {Advances in Mathematics},
    VOLUME = {5},
      YEAR = {1970},
     PAGES = {365--394},
      ISSN = {0001-8708},
   MRCLASS = {28.70},
  MRNUMBER = {274718},
MRREVIEWER = {M.\ A.\ Akcoglu},
       DOI = {10.1016/0001-8708(70)90010-1}}

@article {Ornstein1970,
    AUTHOR = {Ornstein, Donald},
     TITLE = {Bernoulli shifts with the same entropy are isomorphic},
   JOURNAL = {Advances in Math.},
  FJOURNAL = {Advances in Mathematics},
    VOLUME = {4},
      YEAR = {1970},
     PAGES = {337--352},
      ISSN = {0001-8708},
   MRCLASS = {28.70},
  MRNUMBER = {257322},
MRREVIEWER = {U.\ Krengel},
       DOI = {10.1016/0001-8708(70)90029-0}
}

\end{document}